\documentclass[11pt]{amsart}   	
\usepackage[margin=0.9in]{geometry}             
\usepackage[dvipsnames]{xcolor}

\usepackage{graphicx}				
\usepackage{amssymb,mathrsfs,latexsym}
\usepackage{amsthm,amsmath,stmaryrd}
\usepackage{tikz}
\usepackage{tikz-cd}
\usepackage{tikz-3dplot}
\tdplotsetmaincoords{70}{120}
\usetikzlibrary{decorations.pathmorphing}

\tikzset{snake it/.style={decorate, decoration=snake}}

\usepackage{accents,upgreek,enumerate}
\usepackage{bm}
\usepackage{mathtools}
\usepackage[all]{xy}
\usepackage{caption}
\usepackage[euler-digits]{eulervm}
\usepackage[normalem]{ulem}

\usepackage{ wasysym }
\definecolor{lavender}{RGB}{238,230,255}

\usetikzlibrary{shadings,fadings}

\usepackage{setspace}
\tikzset{
  commutative diagrams/.cd, 
  arrow style=tikz, 
  diagrams={>=stealth}
}

\newtheorem{innercustomthm}{Theorem}
\newenvironment{customthm}[1]
  {\renewcommand\theinnercustomthm{#1}\innercustomthm}
  {\endinnercustomthm}

  \makeatletter
\def\@tocline#1#2#3#4#5#6#7{\relax
  \ifnum #1>\c@tocdepth 
  \else
    \par \addpenalty\@secpenalty\addvspace{#2}%
    \begingroup \hyphenpenalty\@M
    \@ifempty{#4}{%
      \@tempdima\csname r@tocindent\number#1\endcsname\relax
    }{%
      \@tempdima#4\relax
    }%
    \parindent\z@ \leftskip#3\relax \advance\leftskip\@tempdima\relax
    \rightskip\@pnumwidth plus4em \parfillskip-\@pnumwidth
    #5\leavevmode\hskip-\@tempdima
      \ifcase #1
       \or\or \hskip 1em \or \hskip 2em \else \hskip 3em \fi%
      #6\nobreak\relax
    \dotfill\hbox to\@pnumwidth{\@tocpagenum{#7}}\par
    \nobreak
    \endgroup
  \fi}
\makeatother

\usetikzlibrary{calc}
\usetikzlibrary{fadings}
\usetikzlibrary{decorations.pathmorphing}
\usetikzlibrary{decorations.pathreplacing}

\newcounter{marginnote}
\DeclareMathAlphabet{\mathpzc}{OT1}{pzc}{m}{it}

\usepackage[backref=page]{hyperref}
\hypersetup{
  colorlinks   = true,          
  urlcolor     = blue,          
  linkcolor    = violet,          
  citecolor   = VioletRed             
}

\newtheorem{theorem}{Theorem}[subsection]
\newtheorem{corollary}[theorem]{Corollary}
\newtheorem{lemma}[theorem]{Lemma}
\newtheorem{proposition}[theorem]{Proposition}
\newtheorem{conjecture}[theorem]{Conjecture}
\newtheorem{quasi-theorem}[theorem]{Quasi-Theorem}

\theoremstyle{definition}
\newtheorem{definition}[theorem]{Definition}

\newtheorem{remark}[theorem]{Remark}

\newtheorem{construction}[theorem]{Construction}
\newtheorem{example}[theorem]{Example}
\newtheorem{blank remark}[theorem]{}

\newtheorem{not1}[theorem]{Notation}

\newcommand{\A}{{\mathbb{A}}}

\newcommand{\PP}{\mathbb{P}}         
\newcommand{\QQ} {{\mathbb Q}}		
\newcommand{\RR} {{\mathbb R}}		
\newcommand{\ZZ} {{\mathbb Z}}

\def\setminus{\smallsetminus}

\newcommand{\cal}{\mathcal}

\def\cD{{\cal D}}

\def\cO{{\cal O}}

\def\cX{{\cal X}}
\def\cY{{\cal Y}}

\begin{document}

\title{The GW/PT conjectures for toric pairs}

\author{Davesh Maulik {\it \&} Dhruv Ranganathan}

\address{Davesh Maulik \\ Massachusetts Institute of Technology, Cambridge, USA}
\email{\href{mailto:maulik@mit.edu}{maulik@mit.edu}}

\address{Dhruv Ranganathan \\ University of Cambridge, Cambridge, UK}
\email{\href{mailto:dr508@cam.ac.uk}{dr508@cam.ac.uk}}

\begin{abstract}
We prove the conjectural correspondence between logarithmic Gromov--Witten theory and logarithmic Donaldson/Pandharipande--Thomas theory for pairs $(Y|\partial Y)$ consisting of a toric threefold $Y$ and any torus invariant divisor $\partial Y$, with primary insertions. The results are the first verifications of this conjecture when $\partial Y$  is singular, i.e. the ``fully logarithmic'' setting, and the first proof of the equivariant toric correspondence for pairs when $\partial Y$ is nonempty. When $\partial Y$ is empty, we get a new proof of the known toric correspondence, but our methods also lead to stronger conclusions. In particular, we show the PT series is a Laurent polynomial in the presence of sufficient positivity and prove a 2008 conjecture of Oblomkov, Okounkov, Pandharipande, and the first author stating the capped vertex is a Laurent polynomial. The methods also verify the logarithmic DT/PT conjecture for toric threefold pairs. Using the constraints of the logarithmic theory, the complete evaluation of toric pairs is determined by a single calculation --- the degree $1$ series of $\mathbb{P}^3$.

\end{abstract}

\maketitle

\setcounter{tocdepth}{1}
\tableofcontents

\section*{Introduction}

Let $Y$ be a smooth threefold. Gromov--Witten (GW) theory concerns intersection theory on moduli of stable maps to $Y$ and Donaldson/Pandharipande--Thomas (DT/PT) theory is about intersection theory on moduli spaces of sheaves on $Y$. The theories are related by a series of conjectures, the {\it GW/DT/PT correspondence}, presented in~\cite{MNOP06a,MNOP06b,PT09}. 

The conjectures admit generalizations to pairs $(Y|\partial Y)$, where $\partial Y$ is a simple normal crossings boundary divisor~\cite{MNOP06b,MR23} and these generalized conjectures have played a key role in the subject for two decades~\cite{BP08,MPT10,OP10,PP12} through degeneration techniques~\cite{Li02,LiWu15}. However, progress has been restricted to when $\partial Y$ is smooth. Even the formulation of the general correspondence requires significant input from logarithmic geometry~\cite{AC11,Che10,GS13,MR20,MR23}. In this paper, we study the case where $\partial Y$ is singular. In future work, we use this to study the absolute GW/PT \textit{descendent} correspondence. 

Logarithmic GW/PT theory constructs an invariant for each choice of curve class $\beta$, tangency along $\partial Y$, genus/Euler characteristic, and incidence conditions\footnote{We do not consider descendent conditions in this paper.}. We have associated generating functions $Z_{\sf GW}(u)$ and $Z_{\sf PT}(q)$ by summing over the genus/Euler characteristic.  The {\it logarithmic GW/PT correspondence} comprises two statements: 
\begin{enumerate}[(i)]
    \item the series $Z_{\sf PT}(q)$ is the Laurent expansion of a rational function in $q$,
    \item the series $Z_{\sf GW}(u)$ and $Z_{\sf PT}(q)$ are equal after setting $-q = e^{iu}$, up to explicit constants.
\end{enumerate}

\subsection{Main results} We study pairs $(Y|\partial Y)$, where $Y$ is a projective toric threefold and $\partial Y$ is any torus invariant divisor, possibly empty. We call these {\it toric threefold pairs}.

\begin{customthm}{A}\label{thm: toric-pair-correspondence}
The logarithmic GW/PT correspondence holds equivariantly for toric threefold pairs. 
\end{customthm}

This gives the first verification of the correspondence for any pair when $\partial Y$ is singular, and the first proof of the equivariant toric correspondence for smooth pairs. Note that the results of~\cite{MOOP,PP12,PP14} do not treat the smooth pair geometry. Even when $\partial Y$ is smooth, our proof passes through pairs with singular boundary. The proofs use several developments in logarithmic enumerative geometry from the last five years~\cite{MR20,MR23,MR25,MR21,R19}.

The methods produce new statements about the absolute PT theory of toric threefolds, when $\partial Y$ is empty. In this case Theorem~\ref{thm: toric-pair-correspondence} specializes to a result of Oblomkov, Okounkov, Pandharipande, and the first author~\cite{MOOP}, but our proof is quite different and leads to the following stronger statement.

\begin{customthm}{B}\label{thm: polynomiality}
Let $(Y|\partial Y)$ be a toric threefold pair. Assume every toric divisor not in $\partial Y$ is nef. Then the equivariant PT series of $(Y|\partial Y)$ is a Laurent polynomial in $q$. 
\end{customthm}

The extreme cases when the hypotheses apply are when $Y$ is a product of projective spaces with $\partial Y = \emptyset$ and when $Y$ is toric and $\partial Y$ is the toric boundary. In the former case, polynomiality is a consequence of a conjecture of Okounkov, Oblomkov, Pandharipande, and the first author~\cite{MOOP}. We prove this conjecture in Section~\ref{sec: polynomiality}.


Allowing an snc boundary introduces a more complicated and larger set of unknown series than just working with smooth pairs.  However, to see the utility of this, we note the following feature of the proof:

\noindent
{\it The set of all primary equivariant GW/PT series of toric threefold pairs is uniquely determined by the series associated to a line passing through two points in $\mathbb P^3$ and compatibility with the logarithmic constraints }\medskip

Said differently, introducing logarithmic targets and the degeneration formula introduce even more relations than it does unknowns. The study of the logarithmic cobordism ring provides a simple illustration in a related setting~\cite{Guz25}.





\subsection{DT/PT}\label{sec: DT-PT} Logarithmic DT theory of ideal sheaves is related to the PT side by a wall-crossing conjecture generalizing~\cite{MNOP06b,PT09}. In~\cite[Conjecture~5.4.1]{MR20}, we proposed a closed form for the degree $0$ logarithmic DT series involving the MacMahon function, recently proved by Guzman~\cite{Guz25}, and conjectured that the PT series equals the DT series divided by the degree $0$ evaluation.

\begin{customthm}{C}\label{thm: dt-pt}
The logarithmic DT/PT correspondence holds equivariantly for toric threefold pairs.
\end{customthm}

We focus on GW/PT to streamline the exposition; Theorem~\ref{thm: dt-pt} follows by repeating the proof of Theorem~\ref{thm: toric-pair-correspondence} step by step. The argument proceeds by inductions compatible with the correspondence and reduces to known base cases.

\subsection{Overview}

Our strategy is based on the recently proved logarithmic degeneration formula~\cite{MR23}. For each invariant, we specialize the insertions to force degenerations and inductively split the problem into ``smaller'' ones. While similar inductive schemes exist in GW theory~\cite{MR25}, they use the target dimension and the domain genus as induction parameters, which are not available on the PT side.

Instead, the key input is the system of relations between GW and PT theory arising from the degeneration formula. We isolate several families of “unbreakable” base cases, and degeneration compatibility then completely constrains the toric pair correspondence. 


Here is a more detailed roadmap:

\begin{enumerate}[I.]
    \item {\it Elementary geometries.} Compatibility of GW/PT with logarithmic degenerations~\cite{MR23} reduces the conjecture to certain elementary targets -- these are iterated $\mathbb P^1$-bundles over toric bases $(B|\partial B)$, where $\partial B$ is the full toric boundary and each $\mathbb P^1$-fiber has log structure along the $0$-section. 
    \item {\it Strata conjectures.} The GW/PT moduli admit stratifications by the same partially ordered set. We formulate a generalized GW/PT correspondence for the strata, with additional insertions from the cohomology of the stack of expansions. By ``rubber calculus'' techniques, a study of exotic insertions, and the degeneration formula~\cite{MR23}, we show the strata correspondences follow from the ordinary correspondence on the components of the universal degeneration on the stratum.
    \item {\it Inductive structure.} The cohomology of an elementary geometry is generated by its boundary strata. This allows incidence conditions to be traded for strata on the GW/PT space. We prove that, in an appropriate ordering, the moduli problems that comprise the strata moduli are smaller. This sets up an induction and reduces the conjecture to several base cases. 
    \item {\it Semi-positive cases.} The first base cases occur when the non-boundary toric strata in $(Y|\partial Y)$ have nef normal bundles. If $\partial Y$ is the full toric boundary, there is an infinite family of cases indexed by the positive integers. We solve these by combining the degeneration formula with tropical arguments of Parker~\cite{Par17}, also used by Bousseau~\cite{Bou17}, which induct on the complexity. The key input is compatibility of the correspondence with the moves above. The arguments can be made to work in the general nef normal bundle setting.
    \item {\it Negative cases.} If the non-boundary toric strata have non-positive normal bundles, there are invariants with no or fewer insertions, and pure degeneration techniques are insufficient. We instead pass to equivariant invariants and take inspiration from the study of the cap and tube geometries in~\cite{OP06vira}. This reduces to certain local curve calculations, which we handle by relating these local curve theories to associated compactified logarithmic theories. This relies on certain basic properties of the logarithmic linear system of curves on a toric surface, developed in Appendix~\ref{app: log-linear-system}. 
\end{enumerate}

To run GW and PT arguments compatibly, we use results of Neguț~\cite{Negut} to rework the evaluation structure in logarithmic PT theory. This makes the GW and PT evaluations coincide --- one moduli space for each choice of contact order, parameterizing {\it ordered} points on strata of $Y$. The ideas may be of interest, for example, in $K$-theoretic PT theory.

\subsection{Broader context} The results here connect to various directions in enumerative geometry. 

\subsubsection{GW/PT} Our results form the basis for work on the GW/PT \textit{descendent} conjectures, aiming to lift the sequence of results by Pandharipande--Pixton~\cite{PP13,PP14,PP12} to the logarithmic setting. The primary GW/PT correspondence for absolute targets has seen significant progress due to Pardon~\cite{Par23}, but the descendent conjectures remain largely open. Our basic aim is to resolve the descendent conjectures for all Fano threefolds, among others.

We highlight two further directions. Our results give the first verification of GW/PT in the fully logarithmic setting. Pardon has recently introduced new methods, involving generic transversality results for holomorphic curves, that prove the primary GW/PT conjecture in the presence of sufficient positivity~\cite{Par23}. In light of this, we raise the following:

\noindent
{\bf Problem.} {\it Adapt the methods of~\cite{Par23} to prove the GW/PT correspondence for positive snc pairs.}

Another question concerns the pair $(\mathbb P^3|5H)$, where the divisor comprises five generic hyperplanes. This is an elementary ``log general type'' target. With $4$ hyperplanes, the pair is toric and the correspondence by the work in this paper. Using descendents, we hope to prove the correspondence for $(\mathbb P^3|5H)$ in forthcoming work, but we state the following:

\noindent
{\bf Problem.} {\it Prove the primary GW/PT correspondence for $(\mathbb P^3|5H)$ without passing to descendents.}

 A solution would imply the correspondence for a number of targets, including the quintic threefold. The first proof of GW/PT for the quintic in~\cite{PP12} used a remarkable argument via the completely inexplicit descendent correspondence. The problem above would give a ``primary to primary'' proof.

\subsubsection{Hodge integrals and triple DR cycles}

The absolute GW theory of toric threefolds can be reduced to the evaluation of cubic Hodge integrals via virtual localization~\cite{GP99}. The logarithmic GW theory of toric threefold pairs shares a similar relationship with higher double ramification (DR) cycles~\cite{HMPPS,HS21,MR21,RUK22}. One way of interpreting our method on the GW side is that we trade cubic Hodge integrals for triple DR cycles, which behave better with respect to the correspondence.

Exact formulas in both cases can be obtained in special situations. Calabi--Yau Hodge integrals can be computed in closed form using the topological vertex~\cite{AKMV}, which is equivalent to the GW/PT correspondence~\cite{MNOP06b}. On the other hand, a large family of rank-$3$ DR cycle integrals can be computed via their relation to the non-commutative KdV hierarchy~\cite{BR21,KHSUK}. It would be interesting to understand how these formulas relate to each other; more generally, further study of the algebraic structure of higher DR integrals and associated closed formulas seems worthwhile.



\subsubsection{Refined tropical curve counting} A rich line of inquiry in tropical geometry is {\it refined} curve counting, which studies polynomial refinements of the classical Severi degrees~\cite{GS14,Mik17}. Severi degrees are topological Euler characteristics of certain moduli spaces associated with linear systems, and the refinements correspond to the $\chi_y$-genera. There is a natural definition of a tropical refined curve count and a partially conjectural interpretation via $\chi_y$-genera; see~\cite{BG16,NPS18}. Bousseau has shown that the tropical refined counts exactly encode the all-genus GW series of the toric threefold pair $S\times \mathbb{P}^1$, with full boundary and zero tangency in the $\mathbb{P}^1$-direction~\cite{Bou17}.  

Via GW/PT, we obtain another sheaf-theoretic interpretation of the refined tropical counts, now as an unrefined PT series. The refinement in~\cite{GS14} is based on $K$-theoretic invariants such as the $\chi_y$-genus, so we observe a refined/unrefined equality. This appears to be related to phenomena in the CY5 theory of Nekrasov--Okounkov~\cite{NO16}; see~\cite[Section~6]{AB25}.

\subsection*{Notation and terminology} We use snc to abbreviate simple normal crossings. We use $Y$ to denote a smooth toric threefold with dense torus $A$ and $\partial Y$ a snc divisor. The $A$-equivariant cohomology will be denoted $H^\star_A(-;\QQ)$. We use $\mathcal X\to B$ for a projective snc degeneration over a curve $B$ with critical value $0$. A map of cone complexes is combinatorially flat if every cone maps surjectively onto a cone. 

\subsection*{Acknowledgments} We thank P. Bousseau, T. Blomme, S. Koyama, P. Kennedy-Hunt, A. Neguț, A. Okounkov, R. Pandharipande, and R. Thomas for discussions, interest, and encouragement. 

\subsection*{Funding} D.M. was supported by a Simons Investigator grant. D.R. was supported by EPSRC Horizon Europe Guarantee EP/Y037162/1 and EPSRC New Investigator Award EP/V051830/1. 

\section{Logarithmic GW and PT theory}

We review the GW and PT theory of pairs and degenerations~\cite{MR20,MR23}. For an introduction, see~\cite{R26-SRI}.

\subsection{Expansions}\label{sec: degenerations} Let $Y$ be a smooth projective toric threefold and $\partial Y$ a torus invariant snc divisor. The pair $(Y|\partial Y)$ determines a cone complex $\Sigma_{Y|\partial Y}$. We use $\Sigma_Y$ or $\Sigma$ when the rest is clear from context. 

\subsubsection{Working up to subdivision} There is a bijection between proper birational log \'etale maps to $(Y|\partial Y)$ and subdivisions of $\Sigma_{Y|\partial Y}$. We use {\it log birational model} and {\it log birational class} for spaces that differ by such maps. Logarithmic GW/PT invariants are unchanged by subdivisions, so we often specify $(Y|\partial Y)$ and $\Sigma_{Y|\partial Y}$ up to subdivision. We adopt two notations -- we write $\underline \Sigma$ for a fan that differs from $\Sigma$ by subdivision, or write the support of the fan: e.g. if $\partial Y$ is the full toric boundary, we use $\mathbb R^3$. 

\subsubsection{Expansion from a polyhedral complex} Let $\Gamma$ be a $1$-dimensional polyhedral complex embedded in $\Sigma$ with integral vertices. View it as lying at height $1$ in $\Sigma\times \RR_{\geq 0}$. Taking its cone, we have:
\[
\mathsf{Cone}(\Gamma)\to \Sigma\times\RR_{\geq 0}.
\]
By toric geometry, we obtain a degeneration 
\[
\mathcal Y_\Gamma\to \A^1
\]
of $\mathcal Y$. Note $\mathsf{Cone}(\Gamma)$ is $2$-dimensional, so the total space $\mathcal Y_\Gamma$ has no codimension $3$ strata. The ``trivial'' expansion is $Y$ minus all closed codimension $2$ strata. See~\cite[Section~1]{MR23}.

\begin{definition}
    The {\it expansion} $Y_\Gamma$ associated to a $1$-dimensional polyhedral complex $\Gamma\subset\Sigma$ is the special fiber of $\mathcal Y_\Gamma$. 
\end{definition}

The irreducible components of $Y_\Gamma$ are in bijection with the vertices of $\Gamma$, and the components of the double locus are in bijection with the edges of $\Gamma$.

\subsubsection{The target at a tropical point} Given an integral point $w$ on $\Sigma$, there is an associated degeneration of $Y$. Take the fan $\Sigma_w\to \Sigma\times\mathbb R_{\geq 0}$ by adding a ray $\rho_w$ through $(w,1)$. This gives a degeneration, the {\it $w$-degeneration to the normal cone}:
\[
\mathcal Y_w\to \A^1.
\]
The degeneration has a logarithmic structure. The exceptional component $Y_w$ inherits a logarithmic structure: one divisor from the intersection of $Y_w$ with the singular locus of the special fiber, and the remaining divisor comes from the intersection of $Y_w$ with the closure of $\partial Y\times (\A^1\setminus \{0\})$. 

If $w$ is in a cone $\sigma$, with locally closed stratum $S^\circ$ with closure $S$, the exceptional component above is an equivariant compactification of the normal torus torsor of $S\subset Y$. The torsors are isomorphic for different $w$ in the interior of $\sigma$. The compactifications are in the same log birational class.

\begin{definition}
For an integral point $w$ in $\Sigma$, a {\it target determined by $w$} is any element in in the logarithmic birational class of the exceptional component in the $w$-degeneration to the normal cone. 
\end{definition}

\subsection{Stable maps and stable pairs} We recall the logarithmic GW/PT correspondence~\cite{AC11,GS13,MR23,R19}. The key construction in~\cite{MR20} is a {\it stack of expansions} $\mathsf{Exp}(Y|\partial Y)$ fitting into a diagram
\[
\begin{tikzcd}
    \mathcal Y\arrow{d} \arrow{r} & Y\\
    \mathsf{Exp}(Y|\partial Y).
\end{tikzcd}
\]
The map $\mathcal Y \to \mathsf{Exp}(Y|\partial Y)$ is a family of expansions of $Y$ along $\partial Y$. Each component $D$ in $\partial Y$ can be specialized to the fibers of $\mathcal Y$. If $\mathcal Y_S$ is the pullback of $\mathcal Y$ to a scheme $S$, denote the union of the specialized boundary divisors by $\partial \mathcal Y_S$. We use $\partial \cdot$ to denote the boundary, in this sense, of any expansion. 

The fibers of $\mathcal Y \to \mathsf{Exp}(Y|\partial Y)$ come with an additional decoration: certain components of each fiber, always $\mathbb{P}^1$-bundles, are designated {\it tube} components. For each expansion over a geometric point, we have a contraction of the tube components $\mathcal Y \to \overline{\mathcal Y}$. 

\begin{remark}[Non-uniqueness]\label{rem: non-uniqueness}
The stack $\mathsf{Exp}(Y|\partial Y)$ is the Artin fan of the cone space $T(Y|\partial Y)$, whose points correspond to $1$-dimensional rational polyhedral complexes in $\Sigma$. While this topological space is canonical, its cone decomposition is not; several distinguished smooth conical structures exist, each determining a stack of expansions, and the collection is closed under subdivision. The precise choice is irrelevant, so we just pick one, and refine it when necessary.
\end{remark}

A family of expansions of $(Y|\partial Y)$ over a scheme $S$ is the pullback of the universal expansion under a map to the stack $\mathsf{Exp}(Y|\partial Y)$. A family of \emph{stable maps to expansions} of $(Y|\partial Y)$ over $S$ consists of a family of expansions $\mathcal Y/S$ and a map from a, possibly disconnected, prestable nodal curve
\[
(C,p_1,\ldots,p_n) \to \mathcal Y \to Y
\]
to an expansion of $(Y|\partial Y)$, subject to the following conditions, over every geometric point:

\begin{enumerate}[(i)]
    \item The preimage of $\partial \mathcal Y$ on $C$ is contained in the set of marked points $\{p_i\}$.
    \item The preimage of the singular locus of $\mathcal Y$ is contained in the singular locus of $C$. 
    \item The tangency order of a node of $C$ with the singular locus, measured after normalization, is equal on the two preimages of the node, i.e. the map is {\it pre-deformable}.
    \item The restriction of the map to a connected component of the preimage of a tube component is a fully ramified Galois cover of a fiber of the $\mathbb P^1$-bundle by a rational curve, ramified only over the $0$ and $\infty$ sections.
    \item After contracting all tube components and their preimages, the induced map $\overline C \to \overline{\mathcal Y}$ has finite automorphism group over $Y$.
\end{enumerate}

In comparison with Li's relative stable maps~\cite{Li01}, the unusual condition is the one about tube components, which can be avoided when $\partial Y$ is smooth but not in general. 

A family of \emph{stable pairs on expansions} of $(Y|\partial Y)$ over $S$ consists of a family of expansions $\mathcal Y/S$ and a pair $(\mathcal E,s)$, where $\mathcal E$ is a pure sheaf of dimension $1$ and $s$ is a section with finite cokernel, subject to the following fiberwise conditions. We will denote the support of $\mathcal E$ by $Z$.

\begin{enumerate}[(i)]
    \item The cokernel of $s$ is supported away from $\partial \mathcal Y$ and from the singular locus of $\mathcal Y$.
    \item The pullback of each component of $\partial Y$ is a non-zero divisor on $Z$.  
    \item For any components $Y_1, Y_2 \subset \mathcal Y$ meeting along $D_{12}$, let $Z_i := Z|_{Y_i}$. Then $D_{12}$ restricts to a non-zero divisor on each $Z_i$, and the scheme-theoretic intersections $Z_1 \cap D_{12}$ and $Z_2 \cap D_{12}$ are equal, i.e., the stable pair is \emph{pre-deformable}.
    \item The restriction of the stable pair to the preimage of a tube component is invariant under the fiberwise $\mathbb G_m$-scaling action on the $\mathbb P^1$-bundle; that is, the pair is pulled back along $\mathcal Y\to\overline{\mathcal Y}$.
    \item The automorphism group of the induced stable pair on $\overline{\mathcal Y}$ from (iv), over $Y$, is finite.
\end{enumerate}

Again, when compared with the Li--Wu theory of relative stable pairs~\cite{LiWu15}, the tube condition is the only unfamiliar part. It is parallel to the GW side.

The discrete data of a stable map are the genus $g$, and the class $\beta$ of the pushforward of $[C]$ to $H_2(Y;\ZZ)$, and the number of marked points $k$ that do not map to $\partial Y$. The discrete data of a stable pair are the Euler characteristic $\chi$ and $\beta$, the pushforward of the support of $\mathcal E$ in $H_2(Y;\ZZ)$.  We have the following~\cite{MR20,R19}:

\begin{theorem}
    The moduli problem of stable maps to expansions of resp. stable pairs on expansions of $(Y|\partial Y)$ is representable by Deligne--Mumford stacks. Fixing discrete data $\beta$ and the genus $g$ resp. Euler characteristic $\chi$, and $k$, the moduli stacks are proper, and denoted $\mathsf{GW}_{\beta}(Y|\partial Y)_g$ resp. $\mathsf{PT}_{\beta}(Y|\partial Y)_\chi$. There are maps:
    \[
    \begin{tikzcd}
        \mathsf{GW}_{\beta,k}(Y|\partial Y)_g\arrow{dr} & &\mathsf{PT}_{\beta}(Y|\partial Y)_\chi\arrow{dl}\\
        &\mathsf{Exp}(Y|\partial Y),&
    \end{tikzcd}
    \]
    equipped with relative perfect obstruction theories. 
\end{theorem}

We will equip both spaces with the pullback logarithmic structure from $\mathsf{Exp}(Y|\partial Y)$\footnote{This is a smaller logarithmic structure on the GW side than is usually used. It is better suited for the correspondence.}.

\subsubsection{GW evaluations} The space $\mathsf{GW}_{\beta,k}(Y|\partial Y)_g$ decomposes by contact orders. A component $D_i$ of $\partial Y$ can be specialized to the universal expansion $\mathcal Y$. Denote it $\mathcal D_i$. The {\it contact order}, i.e. the order of vanishing along $\mathcal D_i$ of the stable map at $p_j$ , is denoted $c_{ij}$. Contact orders are locally constant.

\begin{remark}
    An expansion has no strata of codimension $2$, so we implicitly impose: {\it for each $j$, the contact order $c_{ij}$ is nonzero for at most one $i$.}  The condition always holds after subdivision and this does not affect the virtual class~\cite{AW}. 
\end{remark}

The numbers $c_{ij}$ determine an ordered partition of $\beta\cdot D_i$:
\[
\mu_i:  \sum_j c_{ij} = \beta\cdot D_i,
\]
ignoring the $c_{ij}$ that are $0$. If $\partial Y$ has components $D_1,\ldots, D_r$, we have a vector of partitions
\[
\bm \mu  = (\mu_1,\ldots,\mu_r).
\]
Let $k$ be the number of ``internal'', i.e. tangency $0$ markings. There is a decomposition
\[
\mathsf{GW}_{\beta,k}(Y|\partial Y)_g = \bigcup_{\bm \mu} \mathsf{GW}^{\bm \mu}_{\beta,k}(Y|\partial Y)_g.
\]
Evaluations are given by composing the marked sections with the universal map. A zero contact order marking evaluates to $Y$. If $c_{ij}$ is nonzero, the evaluation factors through $D_i$. Together, we have:
\[
\mathsf{ev}\colon\mathsf{GW}^{\bm \mu}_{\beta,k}(Y|\partial Y)_g\to Y^k\times \mathsf{Ev}^{\bm \mu},
\]
where $k$ is the number of internal marked points. The {\it logarithmic cohomology} of $Y^k\times\mathsf{Ev}^{\bm \mu}$ is the colimit
\[
H^\star_{\sf log}(Y^k\times\mathsf{Ev}^{\bm \mu})= \varinjlim H^\star(\widetilde{Y^k\times\mathsf{Ev}^{\bm \mu}})
\]
taken over subdivisions, with transitions given by pullback. The ordinary cohomology of $Y^k\times\mathsf{Ev}^{\bm \mu}$ is a subalgebra of this direct limit -- its elements are {\it strongly non-exotic} classes. A {\it non-exotic} class comes from the cohomology of a product of blowups of the evaluation factors. The {\it logarithmic Borel-Moore homology} is
\[
H_\star^{\sf log, \mathsf{BM}}(\mathsf{GW}^{\bm \mu}_{\beta,k}(Y|\partial Y)_g;\mathbb Q)= \varprojlim H_\star^{\mathsf{BM}}(\widetilde{\mathsf{GW}^{\bm \mu}_{\beta,k}(Y|\partial Y)_g};\QQ),
\]
with the limit taken over subdivisions and proper pushforward. The virtual class $[\mathsf{GW}^{\bm \mu}_{\beta,k}(Y|\partial Y)_g]^{\sf vir}$ lifts to the logarithmic Borel-Moore homology. For each class $\upalpha$ in $H^\star_{\sf log}(Y^k\times\mathsf{Ev}^{\bm \mu})$ we pullback and integrate to get:
\[
\langle \upalpha |\bm\mu\rangle  = \int_{[\mathsf{GW}^{\bm \mu}_{\beta,k}(Y|\partial Y)_g]^{\sf vir}} \mathsf{ev}^\star(\upalpha) \in \QQ.
\]
We often work with the subspace $H^\star_{\sf log}(Y^k) \otimes H^\star_{\sf log}(\mathsf{Ev}^{\bm \mu})\subset H^\star_{\sf log}(Y^k\times\mathsf{Ev}^{\bm \mu})$ and use notation
\[
\langle \zeta\otimes\delta |\bm\mu\rangle  = \int_{[\mathsf{GW}^{\bm \mu}_{\beta,k}(Y|\partial Y)_g]^{\sf vir}} \mathsf{ev}^\star(\zeta\otimes \delta) \in \QQ.
\]
The evaluation space is insensitive to the genus parameter, so we obtain:
\[
Z_{\mathsf{GW}}\left(Y|\partial Y;u|\zeta\otimes\delta|\bm\mu  \right)_\beta = \sum_g \langle \zeta\otimes \delta\ |\bm\mu\rangle^{\sf GW}_{\beta,g} \cdot u^{2g-2}.
\]

\subsubsection{PT evaluations} For a stable pair $(\mathcal E,s)$ on $\mathcal Y$ and $D_i\subset \partial Y$, restriction to $\mathcal D_i$ gives a map
\[
\mathsf{PT}_{\beta}(Y|\partial Y)_\chi\to  \prod_{i = 1}^r\mathsf{Hilb}^{n_i}(D_i|\partial D_i),
\]
where $n_i$ is equal to $\beta\cdot D_i$. The virtual class of the PT space lifts to logarithmic homology, so given any cohomology class in the target space above we obtain an invariant by integration. 

We say a few words about the cohomology of $\prod_{i = 1}^r\mathsf{Hilb}^{n_i}(D_i|\partial D_i)$. In~\cite[Section~4]{MR23} we construct, for each vector of partitions $\bm \mu$, an injection
\[
H^\star_{\sf log}(\mathsf{Ev}^{\bm\mu})^{\mathsf{Aut}(\bm\mu)}\hookrightarrow H^\star_{\sf log}\left( \prod_{i = 1}^r\mathsf{Hilb}^{n_i}(D_i|\partial D_i)\right)
\]
generalizing the Nakajima correspondence~\cite{Groj96,Nak97}. For each $\zeta\otimes\delta$ in $H^\star_{\sf log}(Y^k) \otimes H^\star_{\sf log}(\mathsf{Ev}^{\bm \mu})^{\mathsf{Aut}(\bm\mu)}$ we have PT invariants $\langle \zeta\otimes \delta|\bm\mu\rangle^{\sf PT}_{\beta,\chi}$.\footnote{The insertions from the first factor $H^\star_{\sf log}(Y^k)$ are usually applied by slant product~\cite[Section~5.4]{MR20}. We rework this definition in the next section and so do need the details here.} Summing over the Euler characteristic, we obtain:
\[
Z_{\sf PT}(Y|\partial Y;q|\zeta\otimes\delta|\bm\mu(\delta))_\beta = \sum_\chi \langle \zeta\otimes \delta|\bm\mu\rangle^{\sf PT}_{\beta,\chi} \cdot q^\chi.
\]
In Section~\ref{sec: Negut-spaces}, we modify the definitions so one can more naturally pull back logarithmic cohomology classes on $Y^k\times\mathsf{Ev}^{\bm\mu}$. 

\subsection{Equivariant invariants} Since $Y$ will be toric and $\partial Y$ torus invariant, the GW/PT invariants lift to equivariant cohomology. The dense torus $A$ acts on this pair, so we have $A$-actions
\[
A\curvearrowright \mathsf{GW}^{\bm\mu}_{\beta,k}(Y|\partial Y)_g, \ \ \mathsf{PT}_\beta(Y|\partial Y)_\chi.
\]
The associated maps
\[
\mathsf{GW}^{\bm\mu}_{\beta,k}(Y|\partial Y)_g, \ \ \mathsf{PT}_\beta(Y|\partial Y)_\chi\to \mathsf{Exp}(Y|\partial Y)
\]
are equivariant with respect to the trivial $A$-action on the stack of expansions. The virtual fundamental classes lift to equivariant Borel--Moore homology. 

The divisors in $\partial Y$ are preserved by $A$, there is an $A$-action on the evaluation spaces, and the evaluation diagrams are all $A$-equivariant. Note also that all subdivisions and the blowups they induce happen at $A$-equivariant centers. By working in equivariant cohomology, we get GW/PT series by the same definitions as above:
\[
Z_{\sf GW} \in H_A^\star(\mathsf{pt};\QQ)(\!(u)\!), \ \ \  Z_{\sf PT}\in H_A^\star(\mathsf{pt};\QQ)(\!(q)\!).
\]

\begin{remark}
Defining equivariant invariants poses no difficulty, but we will not use a virtual localization theorem as in~\cite{GP99}. Such a formula is currently unavailable, but is also not of use for us: the terms of the localization formula are typically not compatible with the correspondence.
\end{remark}

\subsection{Logarithmic GW/PT for primary insertions} Let $d_\beta$ be the anti-canonical degree of $\beta$ on $Y$, let  $\mu_j$ be the entries in the vector of partitions $\bm\mu$. The length of a partition is $\ell(\mu_j)$ and the sum of the entries is $|\mu_j|$. The following is~\cite[Conjecture~A]{MR23}.

\begin{conjecture}[GW/PT correspondence]\label{conj: gw/pt}
Fix a curve class $\beta$, partitions $\bm\mu$, and an insertion $\upalpha$.

    \begin{enumerate}[(i)]
\item The PT series $Z_{\sf PT}(Y|\partial Y;q|\upalpha|\bm\mu)$ is the Laurent expansion of a rational function.
\item Under the change of variables $q = -e^{iu}$, there is an equality of generating functions:
\[
(-q)^{-d_\beta/2}\cdot Z_{\sf PT}(Y|\partial Y;q|\upalpha|\bm\mu)_\beta = (-iu)^{d_\beta+\sum \ell(\mu_j)-|\mu_j|} \cdot Z_{\sf GW}(Y|\partial Y;u|\upalpha|\bm\mu)_\beta.
\]
    \end{enumerate}
\end{conjecture}

The starting point for our proof in the toric case is~\cite[Theorem~C]{MR23}, which we recall.

Let $\mathcal X\to B$ be an snc degeneration, possibly with horizontal divisors. We say the GW/PT correspondence holds for a stratum $S$ if it holds for the component obtained by blowing up $S$ and base changing to make it reduced. 

We have boundary evaluations, from the product over the divisors $D_1,\ldots,D_r$, of the logarithmic Hilbert schemes of these $D_i$. We refer to elements of the usual cohomology of $\prod_{i = 1}^r\mathsf{Hilb}^{n_i}(D_i|\partial D_i)$, viewed inside the logarithmic cohomology, as {\it non-exotic} insertions. 

\begin{theorem}\label{thm: deg-compatibility}
    Let $\mathcal X \to B$ be a projective snc degeneration, possibly with horizontal divisors, over a smooth curve with special point $0$. If the logarithmic GW/PT correspondence holds for all strata of the special fiber, then it holds on the general fiber for non-exotic insertions. 
\end{theorem}

\subsection{Combinatorial structures}\label{combinatorial} Recall that a $1$-complex in $\Sigma$ is a $1$-dimensional rational polyhedral subcomplex embedded in $\Sigma$~\cite{MR20}. It is {\it asymptotic to $\Sigma$} if every unbounded ray in $\Gamma$ is parallel to a ray of $\Sigma$. 

\begin{definition}
    A {\it Chow $1$-complex} in $\Sigma$ is a $1$-complex $\Gamma$ that is asymptotic to $\Sigma$, together with (i) a positive integer $n_E$ assigned to each edge $E$ of $\Gamma$, and (ii) an effective curve class $\beta_V$ in $H_2(S_V;\mathbb{Z})$, where $S_V$ is the stratum of $(Y|\partial Y)$ dual to the cone containing $V$. The {\it curve class} of a Chow $1$-complex is the sum of the pushforwards to $Y$ of its vertex decorations.

A vertex $V$ of a Chow $1$-complex is {\it unstable} if $\beta_V$ is equal to $0$, $V$ is $2$-valent, and $V$ together with the tangent directions at $V$ lie in the relative interior of the same cone. A Chow $1$-complex is {\it stable} if it contains no unstable vertices.
\end{definition}

Chow and stable Chow $1$-complexes with curve class $\beta$ are parameterized by topological spaces underlying cone complexes. These are denoted $\mathcal T^\bullet_\beta(Y|\partial Y)$ for the Chow $1$-complex and $T^\bullet_\beta(Y|\partial Y)$ for the stable ones. The cone complex structure is not unique, in line with Remark~\ref{rem: non-uniqueness}. 

An unstable vertex can be ``erased'', by definition, so we have a natural map $\mathcal T^\bullet_\beta(Y|\partial Y)\to T^\bullet_\beta(Y|\partial Y)$. We can choose cone complex structures such that there is a retraction:
\[
T^\bullet_\beta(Y|\partial Y)\hookrightarrow \mathcal T^\bullet_\beta(Y|\partial Y)\to T^\bullet_\beta(Y|\partial Y),
\]
and furthermore, the second arrow is ``combinatorially flat'' -- it maps every cone of the domain surjectively onto a cone of the target. These results are proved in~\cite{MR23}. 

If $D \subset \partial Y$ is a divisor in the boundary, there is an analogous notion of a Chow $0$-complex: a finite collection of points in the fan $\Sigma_{D|\partial D}$, labeled by multiplicities. If the divisors are $D_1,\ldots, D_r$, and $\bm n$ is the vector $(\beta \cdot D_i)_i$, we can consider the moduli space of Chow $0$-complexes simultaneously on the cone complexes of all $D_i$. These are denoted $P^\bullet_{\bm n}(Y|\partial Y)$, or $P^\bullet_{\bm n}$ when the rest is clear from context. 

We define a map
\[
T^\bullet_\beta(Y|\partial Y) \to P^\bullet_{\bm n}(Y|\partial Y)
\]
as follows: for a $1$-complex in $\Sigma$, the unbounded rays in the direction of a ray $\rho$ of $\Sigma$ determine a $0$-complex in the star fan of $\rho$. The collection of $0$-complexes gives the map. See Figure~\ref{fig: evaluation-trop} below.

\begin{figure}[h!]
    \begin{tikzpicture}[scale=3]

\definecolor{softlavender}{RGB}{223,210,255}

\fill[softlavender, opacity=0.4] (3.8,16.0) rectangle (5.6,17.7);

\draw[violet]   
(4.4991,16.5976) -- (4.8617,16.5976) -- (4.8617,16.9189) --
(4.6887,17.092) -- (4.3261,17.092) -- (4.3261,16.7706) -- cycle;

\draw[violet] (4.8617,16.5976) -- (5.2298,16.2859);
\draw[violet] (4.8617,16.9189) -- (5.1969,16.9272);
\draw[violet] (4.4971,16.3217) -- (4.4948,16.5976);
\draw[violet] (4.4971,16.3217) -- (4.8652,16.01);
\draw[violet] (4.4971,16.3217) -- (3.991,16.3289);
\draw[violet] (5.1969,16.9272) -- (5.1969,17.59);
\draw[violet] (5.1969,16.9272) -- (5.565,16.6155);
\draw[violet] (4.6897,17.5685) -- (4.6887,17.092);
\draw[violet] (4.3261,16.7706) -- (3.82,16.7778);
\draw[violet] (4.0589,17.3213) -- (4.3261,17.092);
\draw[violet] (4.0589,17.3213) -- (3.8299,17.3249);
\draw[violet] (4.0599,17.5757) -- (4.0589,17.3213);

\begin{scope}[xshift=2.4cm]

\draw[violet] (3.9,16.35) -- (3.9,17.55);
\draw[violet] (3.95,17.72) -- (5.05,17.72);
\draw[violet] (4.05,16.35) -- (5.15,17.45);

\shade[ball color=violet] (3.9,16.50) circle (0.6pt);
\shade[ball color=violet] (3.9,16.95) circle (0.6pt);
\shade[ball color=violet] (3.9,17.35) circle (0.6pt);

\shade[ball color=violet] (4.25,17.72) circle (0.6pt);
\shade[ball color=violet] (4.60,17.72) circle (0.6pt);
\shade[ball color=violet] (4.90,17.72) circle (0.6pt);

\shade[ball color=violet] (4.30,16.60) circle (0.6pt);
\shade[ball color=violet] (4.60,16.90) circle (0.6pt);
\shade[ball color=violet] (4.90,17.20) circle (0.6pt);

\end{scope}

\end{tikzpicture}
\caption{The map sending a $1$-complex to the corresponding configuration of points on the ends by taking asymptotics}\label{fig: evaluation-trop}.
\end{figure}
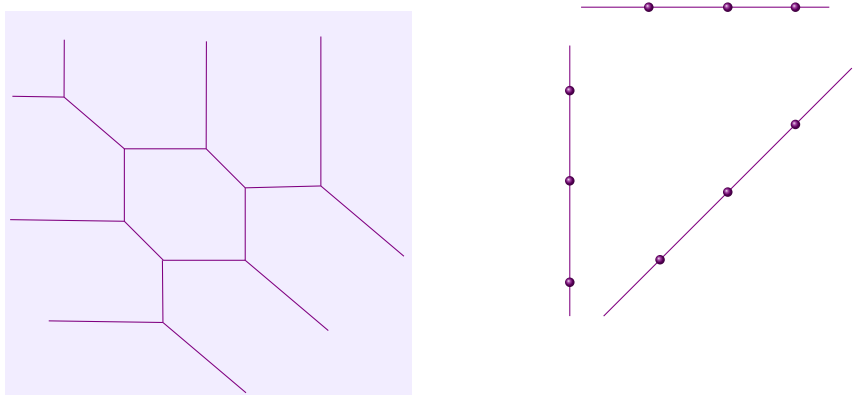

\subsubsection{Variations} We consider some variations on these tropical spaces, constructed similarly.

\begin{enumerate}[(i)]
\item {\it Balancing.} One can impose a balancing condition. This is a local condition at every vertex. We will later use a version of this: we equip $\Sigma$ with a map $\Sigma \to \mathbb{R}^s$ and require that at every vertex, the sum of outgoing directions weighted by the edge multiplicities is $0$. This picks out a subcomplex of the space of (stable) Chow $1$-complexes. Since it is a passage to a subcomplex, we will not change notation when we use the balancing condition, but simply alert the reader.

\item {\it Internal markings.} Fix a positive integer $k$ and partition the set $[k]$ among the vertices of a Chow $1$-complex. Under edge contractions, we take the union of these subsets. An unstable vertex is required to have no elements of $[k]$. Denote these spaces $\mathcal T^\bullet_{\beta,k}(Y|\partial Y)$, $T^\bullet_{\beta,k}(Y|\partial Y)$, etc.

\item {\it Labeled ends.} Given a Chow $1$-complex, the number of rays parallel to each ray in $\Sigma$, counted with the multiplicity labeling, is constant and gives an unordered partition in each direction. We pass to the moduli space that orders this partition. We denote it $\mathcal T_\beta(Y|\partial Y)$, $\mathcal T_{\beta,k}(Y|\partial Y)$, etc.

\item {\it Labeled points.} We can have ordered or unordered $0$-complexes, denoted respectively $P_{\bm n}$ and $P^\bullet_{\bm n}$. 

\item {\it Fixed partition.} Given a Chow $1$-complex and a ray of $\Sigma$, the multiplicities on the ends give a partition $\mu_i$. These partitions have specialization relations: a partition is specialized by combining some of the parts. If we fix a vector of partitions $\bm \mu$, we can consider the subcomplex of $\mathcal T_\beta(Y|\partial Y)$ whose partition in the $i^{\text{th}}$ direction is a specialization of $\mu_i$. We denote these by $\mathcal T^{\bm \mu}_\beta(Y|\partial Y)$ and the stable version by $T^{\bm \mu}_\beta(Y|\partial Y)$. We can also do this without labeling and for $0$-complexes. 
\end{enumerate}

A cone complex $\Sigma$ determines an {\it Artin fan} denoted $\mathsf a\Sigma$. This is the colimit of Artin stacks $[U_\sigma/T_\sigma]$, where $U_\sigma$ is the $T_\sigma$-toric variety of $\sigma$, over faces $\sigma$ in $\Sigma$, see~\cite{CCUW}. There are maps:
\begin{equation}
\begin{tikzcd}\label{eqn: basic-diagram}
    \coprod_{\bm \mu} \mathsf{GW}^{\bm \mu}_{\beta,k}(Y|\partial Y)_g \arrow{dr} & &\mathsf{PT}_{\beta}(Y|\partial Y)_\chi\arrow{dl}\\
        &\mathsf a T^\bullet_\beta(Y|\partial Y).&
\end{tikzcd}
\end{equation}
Similarly, we have:
\[
\begin{tikzcd}
\coprod_{\bm \mu} \mathsf{Ev}^{\bm \mu} \arrow{dr} & &\prod_{i = 1}^r\mathsf{Hilb}^{n_i}(D_i|\partial D_i) \arrow{dl}\\
        &\mathsf aP^\bullet_{\bm n}.&
\end{tikzcd}
\]
These fit together with the evaluation structures, to give a commutative diagram:
\[
\begin{tikzcd}
    \mathsf{GW}^{\bm \mu}_{\beta,g}(Y|\partial Y)\arrow{d}\arrow{r} & \mathsf{Ev}^{\bm \mu}\arrow{d}\\
    \mathsf a T^\bullet_\beta(Y|\partial Y)\arrow{r}& \mathsf a P^\bullet_{\bm n}.
\end{tikzcd}
\]

\subsection{Degeneration formulas} Theorem~\ref{thm: deg-compatibility} is a consequence of parallel degeneration formulas for GW and PT theory, which we recall~\cite{MR23}.

\subsubsection{Terminology and setup} Let $\mathcal X \to B$ be a projective snc degeneration.

\begin{enumerate}[(i)]
    \item The dual polyhedral complex of $\mathcal X_0$ is denoted $\Delta(\mathcal X_0)$, or simply $\Delta$. The space of Chow $1$-complexes in $\Delta$ with curve class $\beta$, denoted $T^\bullet_\beta(\Delta)$, can be given the structure of a polyhedral complex. This structure is not unique, but we simply pick one. 

    \item A Chow $1$-complex $\Gamma$ in $\Delta$ is {\it rigid} if it corresponds to a vertex of $\mathcal T^\bullet_\beta(\Delta)$. Note that this depends on the choice of polyhedral structure. 

    \item For each rigid $1$-complex $\Gamma$, and each vertex $V$ of $\Gamma$, there is a target $(Y_V|\partial Y_V)$ associated to $V$. 

    \item For each vertex and partition of the edge multiplicities, there is an evaluation
    \[
    T^{\bm \mu}_\beta(Y_V|\partial Y_V) \to P^{\bm \mu}_{\bm n}(Y_V|\partial Y_V).
    \]
    We assume cone structures have been chosen so that these maps are combinatorially flat for all vertices of $\Gamma$. We get distinguished birational models of the evaluation space $\mathsf{Ev}^{\bm\mu(V)}$.

    \item Fix, for every edge in $\Gamma$, a partition of its multiplicity. We have a product space
    \[
    \prod_V \mathsf{Ev}^{\bm\mu(V)}(Y_V|\partial Y_V).
    \]
    Away from the boundary divisors, for each part in a partition of each edge, the corresponding divisor appears twice. There is a natural diagonal locus, and its closure is denoted $\mathbb D_{\bm \mu}$. This defines a cohomology class in the product, the \textit{strict diagonal}, and we write its K\"unneth decomposition as
    \[
    [\mathbb D_{\bm \mu}] = \sum_j \bigotimes_V \delta_V^{(j)} \in \bigotimes H^\star(\mathsf{Ev}^{\bm\mu(V)}(Y_V|\partial Y_V)).
    \]

    \item Let $S$ be a set of $k$ cohomology classes in $H^\star_{\sf log}(\mathcal X)$. We will be interested in partitions of $S$.
\end{enumerate}

We state the degeneration formulas in GW/PT theory in terms of generating functions~\cite[Section~8.5]{MR23}. Fix insertions for GW/PT invariants on the general fiber $\cX_\eta$ with non-exotic, non-vanishing insertions, i.e., classes pulled back from $\cX$. Denote the generating series by $Z_{\sf PT}(q)$ and $Z_{\sf GW}(u)$. We define multiplicity $m_{\bm \mu}$ to be the product over all partitions of the products of the parts of the partition.  Similarly, we introduce the sign $(-1)^{\bm \mu}$
to be the product over all partitions of $(-1)^{\ell(\mu_j) - |\mu_j|}$.

\begin{theorem}[GW and PT degeneration formulas]\label{thm: deg-formula}
    For the PT series, we have
    \[
\mathsf{Z}_{\mathsf{PT}}(q)
= 
\sum_{\Gamma} \sum_{\bm{\mu}} \sum_{S = \coprod S_V} \sum_{(j)} \frac{(-1)^{\bm{\mu}}m_{\bm{\mu}}}{\mathrm{Aut}(\bm{\mu})}\frac{1}{q^{|\bm{\mu}|}}
\prod_{V}  
\mathsf{Z}_{\mathsf{PT}}\left(Y_V|\partial Y_V;q |
\prod_{i \in S_V} \zeta_i\otimes \delta^{(j)}_{V}| \bm{\mu}_{V}
\right)_{\beta(V)}.
\]

For the GW series, we have
\[
\mathsf{Z}_{\mathsf{GW}}(u)
= 
\sum_{\Gamma} \sum_{\bm{\mu}} \sum_{S = \coprod S_V} \sum_{(j)} \frac{m_{\bm{\mu}}}{\mathrm{Aut}(\bm{\mu})}\cdot u^{2\ell(\bm{\mu})}
\prod_{V}  
\mathsf{Z}_{\mathsf{GW}}\left(Y_V|\partial Y_V;u |
\prod_{i \in S_V} \zeta_i\otimes \delta^{(j)}_{V}| \bm{\mu}_{V}
\right)_{\beta(V)}.
\] 
In both formulas, the sum over $\Gamma$ is over rigid $1$-complexes, the second is over partitions of the internal edge labels of $\Gamma$, the third sum is the distribution of the internal insertions among the vertices, and the final sum accounts for the K\"unneth factors of the strict diagonal. 
\end{theorem}

We have stated formulas for invariants with internal insertions in the general fiber $\cX_\eta$. If $\cX_\eta$ has boundary divisors, we can include boundary insertions in the formula, provided they insertions are non-exotic. This follows immediately from the cycle theoretic formula in~\cite[Section~8]{MR23}. 

\begin{remark}[Comparison with the traditional formula]
For readers familiar with the double point degeneration formula~\cite{Li02,LiWu15}, we highlight the new features. The first is (iii): the targets $Y_V$ may be projective bundles over strata of $\mathcal X_0$, not just irreducible components. 

The second difference occurs in (iv). For a smooth pair, no blowup is required, and the gluing condition is imposed by the diagonal in $\prod_E D_E^2$. In Step (iv), however, one performs a subdivision that flattens the tropical evaluation, with gluing imposed by the strict transform of the diagonal. After flattening, the evaluation space need not remain a product of divisors over edges. Instead, it can be described as a product of evaluation spaces taken over the \textit{vertices} of $\Gamma$. Each factor is the blowup of the product, over the edges incident to a given vertex, of the corresponding divisors.

The third difference is that the K\"unneth components of $\mathbb D_{\bm\mu}$ induce {\it exotic} insertions. These occur again because for each vertex, the evaluation space is the blowup of a product of divisors. This complicates iterative uses of the formula: if the general fiber carries exotic insertions, the special fiber may not admit a splitting. Iteration requires an exotic/non-exotic conversion rule, which we discuss later.
\end{remark}

\section{Fixed partition PT moduli spaces}\label{sec: Negut-spaces}

We run parallel arguments for stable maps and stable pairs. However, the evaluation spaces differ: on the GW side they are products of target strata and on the PT side we have Hilbert schemes of points for boundary insertions and the slant product for internal insertions. We modify the PT spaces so the evaluation maps match the GW side. For boundary evaluations, we use Neguț's work~\cite{Negut,Negut2} on flag Hilbert schemes; for internal evaluations, we use Oberdieck's {\it marked PT moduli}~\cite{Ob21}.

\subsection{Boundary evaluations and flag Hilbert schemes} Let $S$ be a smooth surface and let 
$$\mu = (\mu^{(1)}, \dots, \mu^{(\ell)})$$ 
be a partition of $d$, with $\mu^{(1)} \geq \mu^{(2)} \dots \geq \mu^{(\ell)}$. Let $Z \hookrightarrow S$ be a subscheme of length $d$. A {\it $\mu$-flag structure} on $Z$ is an increasing flag of subschemes
$$\emptyset = Z_0 \subset Z_1 \subset \cdots \subset Z_{d-1} \subset Z_d = Z$$
such that 
\begin{enumerate}[(i)]
\item for each $0\leq k\leq d$ the length of $Z_k$ is $k$ and
\item there exist points $p_1, \cdots, p_\ell \in S$ such that $\mathrm{supp}(\cO_{Z_{k}}/\cO_{Z_{k-1}}) = p_i$ when 
$\mu^{(1)} + \cdots +\mu^{(i-1)} + 1 \leq k \leq \mu^{(1)} + \cdots +\mu^{(i)}$.
\end{enumerate}

\begin{remark}
    Let $\mu$ be the partition $(1,1,\ldots,1)$. If $Z$ is a union of $d$ distinct points, then a $\mu$-flag structure is just an ordering of the points. In general, we obtain an ordering of the support, but also at fat points, we have a moduli space of one-step flags from the support to the full fat point. 
\end{remark}

We recall Neguț's result from in~\cite[Proposition~2.19]{Negut}. See also~\cite[Section~4]{Negut2}.

\begin{proposition}
The moduli problem of length $d$ subschemes of a smooth surface $S$ with $\mu$-flag structure is representable by a scheme $\mathsf{Hilb}_\mu(S)$. There is a map
\[
\rho\colon \mathsf{Hilb}_\mu(S) \rightarrow \mathsf{Hilb}_d(S) \times S^{\ell}
\]
which sends $(Z_i)_i$ to $(Z, p_1, \dots, p_\ell)$. The map $\rho$ carries a perfect obstruction theory of relative dimension $-d$. The moduli space therefore carries a virtual class $[\mathsf{Hilb}_\mu(S)]^{\sf vir}$ in degree $d+\ell$. 
\end{proposition}

Recall that there is a correspondence 
\[
\widetilde\Gamma_{\mu}\subset \mathsf{Hilb}_d(S)\times S^\ell,
\]
whose action on cohomology identifies the partition $\mu$ subspace of $\mathsf{Hilb}_d(S)$ in the Grojnowski--Nakajima basis for the cohomology of the Hilbert scheme, see~\cite{Nak97}. Neguț's result is compatible with this correspondence~\cite[Theorem~6.19]{Maulik-Negut}:

\begin{proposition}
    The pushforward
$$\rho_\star[\mathsf{Hilb}_\mu(S)]^{\mathrm{vir}} \in \mathsf{CH}_{d+\ell}(\mathsf{Hilb}_d(S) \times S^{\ell})$$
is equal to the class of the Nakajima correspondence $[\widetilde\Gamma_{\mu}]$.
\end{proposition}

\subsection{Flagged moduli space} Let $(Y|\partial Y)$ be a threefold pair. Fix a curve class $\beta$ and, for each divisor $D_i$, an ordered partition $\mu_i$ of $\beta \cdot D_i$. We view these partitions as attached to the rays of $\Sigma$ and denote the collection by $\bm \mu = \{\mu_i\}$. We also have $k$ internal markings.  

Given discrete data $[\beta, k, \bm \mu]$, let $\mathcal T_{\beta,k}^{\bm \mu}(Y|\partial Y)$ denote the tropical space of Chow $1$-complexes with fixed tangency $\bm \mu$, and let $T_{\beta,k}^{\bm \mu}(Y|\partial Y)$ denote the subcomplex of stable Chow $1$-complexes. Recall the stabilization/inclusion maps
\[
T_{\beta,k}^{\bm \mu}(Y|\partial Y) \subset \mathcal T_{\beta,k}^{\bm \mu}(Y|\partial Y) \rightarrow T_{\beta,k}^{\bm \mu}(Y|\partial Y),
\]
and assume the stabilization map is combinatorially flat.

\begin{definition}
    The {\it marked PT moduli space} 
$$\pi:  \mathsf{PT}^{\bm \mu}_{\beta, k}(Y)_{\chi} \rightarrow a\mathcal T_{\beta,k}^{\bm \mu}(Y|\partial Y)$$ 
parametrizes a choice of expansion $\mathcal Y$ of $Y$ 
equipped with 
\begin{enumerate}[(i)]
\item a stable pair $\mathcal{O}_{\mathcal Y} \rightarrow \mathcal E$ on an expansion of $(Y|\partial Y)$ with discrete invariants $\beta$ and $\chi$.
\item $k$ ordered points on the rank $0$ locus of $\mathcal Y$
\item for each ray $r_i$ of $\mathsf v$, a $\mu_i$-flag structure along $Z_i:= \mathrm{supp}(\mathcal E)\cap \mathcal{D}_i$.
\end{enumerate}
The morphism $\pi$ is strict and virtually smooth.
\end{definition}

A marking of $[\beta,k, \bm \mu]$ consists of either an internal marking from $[k]$ or one of the parts of a partition $\mu_k$ in $\bm \mu$.   
We have an evaluation map
\[
\mathsf{ev}: \mathsf{PT}^{\bm \mu}_{\beta, k}(Y)_{\chi}  \rightarrow Y^k \times \mathsf{Ev}^{\bm \mu}.
\]
The virtual class of $[\mathsf{PT}^{\bm \mu}_{\beta, k}(Y)_{\chi}]^{\sf vir}$ lifts to logarithmic homology. Evaluation classes can be taken from $H^\star_{\sf log}(Y^k\times\mathsf{Ev}^{\bm \mu})$. As before, we typically work with insertions from the subspace:
\[
H^\star_{\sf log}(Y^k) \otimes H^\star_{\sf log}(\mathsf{Ev}^{\bm \mu})\subset H^\star_{\sf log}(Y^k\times\mathsf{Ev}^{\bm \mu}).
\]

We have a moduli space of ordered $0$-complexes on $\Sigma_{Y|\partial Y}$ for internal insertions and on $\Sigma_{D|\partial D}$ for each divisor $D$ for boundary insertions. The evaluation is unaffected by stabilization, so we have:
\[
\mathcal T_{\beta,k}^{\bm \mu}(Y|\partial Y) \rightarrow T_{\beta,k}^{\bm \mu}(Y|\partial Y) \rightarrow P_{\beta,k}^{\bm \mu}(Y|\partial Y).
\]

\noindent
{\bf Evaluation Diagram.} Combining the evaluation structure and the combinatorial picture, passing to Artin fans, we get a commutative diagram:
\[
\begin{tikzcd}
    \coprod_\chi \mathsf{PT}^{\bm \mu}_{\beta, k}(Y)_\chi  = \mathsf{PT}^{\bm \mu}_{\beta, k}(Y) \arrow{d} \arrow{r}& \mathsf{Ev}_{\beta,k}^{\bm \mu}(Y|\partial Y)\arrow{d}\\
     \mathsf a T_{\beta,k}^{\bm \mu}(Y|\partial Y)\arrow{r} & \mathsf a P_{\beta,k}^{\bm \mu}(Y|\partial Y).
\end{tikzcd}
\]
In the left vertical arrow, we are mapping to the Artin fan of {\it stable} Chow $1$-complexes. The map is combinatorially flat but not strict.

\subsection{Marked PT spaces} We suppress discrete data from the notation and use $\mathsf{PT}$, $T$, $\mathsf{Ev}$, and $\mathsf P$ with partial discrete data when the rest can be inferred from context.

Let $\mathcal Y \to \mathsf{PT}$ denote the universal expansion over the marked PT space. It carries a universal stable pair $[\mathcal{O} \to \mathcal{E}]$ and is equipped with sections
\[
s_i: \mathsf{PT} \to \mathcal Y, \quad i=1,\ldots,k,
\] 
associated to the internal markings. Fix discrete data $[\beta,k, \bm \mu]$ and a cohomology class $\upalpha \in H^\star_{\sf log}(\mathsf{Ev})$. For fixed holomorphic Euler characteristic $\chi$, define the PT invariant
\[
\langle \upalpha | \bm{\mu}\rangle^{\mathsf{PT}}_{\beta,\chi}
= 
\frac{1}{m_{\bm \mu}}\int_{[\mathsf{PT}_{\chi}]^{\mathsf{vir}}} \prod_{i=1}^{k} s_{i}^\star(\mathsf{ch}_{2}(\mathcal{E})) \cdot \mathsf{ev}^\star(\upalpha).
\]


For insertions in $H^\star_{\sf log}(Y^k) \otimes H^\star_{\sf log}(\mathsf{Ev}^{\bm \mu})$, written $\zeta\otimes\delta$, denote the invariant $\langle\zeta\otimes \delta|\bm\mu\rangle^{\sf PT}_{\beta,\chi}$ to keep the boundary structure distinct from the internal structure. Summing over $\chi$, we have:
\[
\mathsf{Z}_{\mathsf{PT}}\left(Y|\partial Y;q | \zeta \otimes \delta| \bm{\mu} \right)_{\beta}
= \sum_{\chi} \langle \zeta\otimes \delta| \bm{\mu}\rangle^{\mathsf{PT}}_{\beta,\chi} \cdot {q^{\chi}}.
\]
The evaluation insertion $\zeta\otimes \delta$ in logarithmic cohomology is {\it non-exotic} if it is pulled back from the cohomology of a product of blowups. It is {\it strongly non-exotic} if it is pulled back from $Y^k\times\mathsf{Ev}^{\bm\mu}$. The following lemma is shown for smooth pairs $(Y|D)$ in~\cite{Ob21}.

\begin{lemma}
For strongly non-exotic insertion, the above definition of $\mathsf{Z}_{\mathsf{PT}}(q)$ agrees with the original one.
\end{lemma}

\begin{proof}
    Oberdieck's argument in the smooth pair case~\cite[Section~3.3]{Ob21} uses only formal compatibilities of pullbacks and Chern characters, and applies without change here. 
\end{proof}

With this lemma in hand, going forward, we will only work with marked PT moduli spaces associated to discrete data $[\beta, k, \bm \mu]$.

\section{Elementary geometries}

The degeneration compatibility of Theorem~\ref{thm: deg-compatibility} reduces the toric correspondence to what we call {\it elementary geometries}. We will show, any toric pair can be built from these by gluing and smoothing.

\subsection{List of elementary geometries}\label{sec: elementary-list}

Elementary geometries are certain toric pairs $(Y|\partial Y)$, defined up to logarithmic blowups and blowdowns. The tropicalization is well-defined up to subdivision. 

For any toric pair $(Y| \partial Y)$, the tropicalization $\Sigma_{Y}$ can be described as follows.  Take the fan of $Y$ inside $\mathbb R^3$ and remove every ray associated to a non-boundary toric divisor, along with the interior of every cone containing such a ray as a face.  The underlying subset $\underline{\Sigma}_Y$ has an embedding  $\underline{\Sigma}_Y \subset \mathbb R^3.$ The logarithmic map from $Y$ relative to its full boundary to $(Y|\partial Y)$ induces a piecewise linear surjection on cone complexes:
\[
\pi_Y: \mathbb R^3 \rightarrow \underline{\Sigma}_Y
\]
which is the identity on $\underline{\Sigma}_Y$ and zero on the rays in $\mathbb R^3$ corresponding to the non-boundary divisors.

We now list the elementary geometries and their tropicalizations: 

\begin{enumerate}[(i)]
\item {\bf Full boundary.} The space $Y$ is a projective toric threefold and $\partial Y$ is the full toric boundary divisor. These form a single logarithmic birational equivalence class. The tropicalization is $\mathbb R^3$.
\item {\bf $1$-Non-boundary.} The space $Y$ is a logarithmic birational model of a $\mathbb P^1$-bundle
\[
\mathbb P(L\oplus \mathcal O)\to S,
\]
where $S$ is a toric surface, $L$ is any line bundle. The divisor $\partial Y$ on the total space is given by the pullback of the full toric boundary on $S$ together with the infinity section. 
To describe the tropicalization, we can associate to $(S,L)$ a piecewise linear function $\psi_L$ on $\underline{\Sigma}_S = \mathbb R^2$.  The tropicalization is given by
\[
\underline{\Sigma}_Y = \left\{(x,y,z) | z \geq \psi_L(x,y)\} \subset \mathbb R^3\right\}.
\]

We introduce the further distinction based on the line bundle $L$:
\begin{enumerate}[(a)]
\item{\bf Straight.} $L = \mathcal O_S$.
\item{\bf Ruled.} $L = g^\star\mathcal O(a)$ for a fibration $g: S \rightarrow \mathbb P^1$
\item{\bf General.} $L$ is not of the form described above.
\end{enumerate}

\item {\bf $2$-Non-boundary.} The space $Y$ is a logarithmic birational model of a $\mathbb P^2$-bundle
\[
\mathbb P(L_1\oplus L_2\oplus \mathcal O)\to \mathbb P^1,
\]
where $L_1$ and $L_2$ are any line bundles on $\mathbb P^1$. The divisor $\partial Y$ is given by the pullback of the toric boundary on $\mathbb P^1$ together with the infinity divisors in the $L_1$ and $L_2$ coordinates. 
In this case, we associate to $(L_1, L_2)$ two piecewise linear functions $\psi_1, \psi_2$ on $\mathbb R$, and the tropicalization is
\[
\underline\Sigma_Y = \left\{(x,y,z) | y \geq \psi_1(x), z \geq \psi_2(x)\} \subset \mathbb R^3\right\}.
\]
We introduce the further distinction:
\begin{enumerate}[(a)]
\item{\bf Straight.} $L_1 = L_2 = \mathcal O$
\item{\bf General.} $L_1, L_2$ not both trivial.
\end{enumerate}

\item {\bf $3$-Non-boundary.} The space $Y$ is $\PP^1\times\PP^1\times\PP^1$ and $\partial Y$ is the union of the infinity divisors.  
The tropicalization is $\mathbb R_{\geq 0}^3$.
\end{enumerate}

We will refer to a target $(Y|\partial Y)$ that is of the form above as an {\it elementary} threefold pair. 

\subsection{Reduction to elementary geometries}

\begin{proposition}
    Every smooth and projective toric threefold pair admits a toric snc degeneration to a union of elementary threefold pairs. 
\end{proposition}

Elementary geometries are stable under subdivisions by definition, so once such a degeneration exists, any other that is obtained by blowups and base change retains the elementary property. 

\begin{proof}
    We construct a polyhedral decomposition of the fan of $Y$ whose recession fan is $\Sigma_Y$, following~\cite{KKMSD}. For $\mathbb R_{\geq 0}$, consider the decomposition $\mathcal P^1$ into $[0,1]$ and $[1,\infty)$. The product of $k$ copies gives a decomposition of $\mathbb R_{\geq 0}^k$; denote it by $\mathcal P^k$. Its recession fan is $\mathbb R_{\geq 0}^k$. The cone complex $\Sigma$ is a colimit of a diagram of cones $\mathbb R_{\geq 0}^k$. Replacing each copy of $\mathbb R_{\geq 0}^k$ in this diagram with $\mathcal P^k$ gives a new colimit, which is a polyhedral decomposition of $\Sigma_Y$; call it $\mathcal P$.

Some cells of $\mathcal P$ are of the form $\square \times \mathbb R^m$, where $\square$ is a square or cube. These correspond to quadratic singularities in the associated degeneration. Subdividing these squares or cubes resolves the singularities and gives a new decomposition $\mathcal P'$. The associated degeneration produces a degeneration of $Y$ over $\mathbb A^1$ of the required form. See Figure~\ref{fig: elementary-degeneration} for an analogous picture in 2-dimensions. 
\end{proof}

\begin{figure}[h!]
\begin{tikzpicture}

\begin{scope}[scale=3.5] 
  \fill[lavender] (-0.1,-0.1) rectangle (1.1,1.1);

  \coordinate (A) at (0,0);
  \coordinate (B) at (1,0);
  \coordinate (C) at (0,1);
  \draw[violet, thick] (A) -- (B) -- (C) -- cycle;

  \shade[ball color=violet] (A) circle (0.5pt);
  \shade[ball color=violet] (B) circle (0.5pt);
  \shade[ball color=violet] (C) circle (0.5pt);

  \coordinate (A1) at (0.2,0.2);
  \coordinate (B1) at (0.5,0.2);
  \coordinate (C1) at (0.2,0.5);
  \draw[violet, thick] (A1) -- (B1) -- (C1) -- cycle;

  \draw[violet, thick] (A1)--(0.2,0);
  \draw[violet, thick] (A1)--(0,0.2);
  \draw[violet, thick] (B1)--(0.5,0);
  \draw[violet, thick] (B1)--(0.65,0.35);
  \draw[violet, thick] (C1)--(0,0.5);
  \draw[violet, thick] (C1)--(0.35,0.65);

\end{scope}

\begin{scope}[xshift=-4cm, yshift=1.75cm,scale=0.95] 

  \fill[lavender] (-2.2,-2.2) rectangle (2.2,2.2);

  \draw[violet, thick,->] (0,0) -- (2,0);
  \draw[violet, thick,->] (0,0) -- (0,2);
  \draw[violet, thick,->] (0,0) -- (-2,-2);

  \draw[violet, thick,->] (1,1) -- (2,1);
  \draw[violet, thick,->] (1,1) -- (1,2);

  \draw[violet, thick,->] (-1,0) -- (-1,1);
  \draw[violet, thick,->] (-1,0) -- (-2,-1);

  \draw[violet, thick,->] (0,-1) -- (1,-1);
  \draw[violet, thick,->] (0,-1) -- (-1,-2);

  \draw[violet, thick] (1,1) -- (1,0);
  \draw[violet, thick] (1,1) -- (0,1);

  \draw[violet, thick] (-1,0) -- (0,1);
  \draw[violet, thick] (-1,0) -- (-1,-1);

  \draw[violet, thick] (0,-1) -- (1,0);
  \draw[violet, thick] (0,-1) -- (-1,-1);

  \shade[ball color=violet] (0,0) circle (0.6pt);
  \shade[ball color=violet] (1,1) circle (0.6pt);
  \shade[ball color=violet] (-1,0) circle (0.6pt);
  \shade[ball color=violet] (0,-1) circle (0.6pt);
  \shade[ball color=violet] (1,0) circle (0.6pt);
  \shade[ball color=violet] (0,1) circle (0.6pt);
  \shade[ball color=violet] (-1,-1) circle (0.6pt);

\end{scope}

\end{tikzpicture}
\caption{A degeneration of $\mathbb P^2$ into elementary geometries, depicted by the polyhedral decomposition on the left, and the moment polytope complex on the right. On the right picture, the edges of each cell are toric boundary divisors. The edges that are contained within those of the outer triangle are non-boundary. Note that the quadratic singularities here have not been resolved.}\label{fig: elementary-degeneration}
\end{figure}
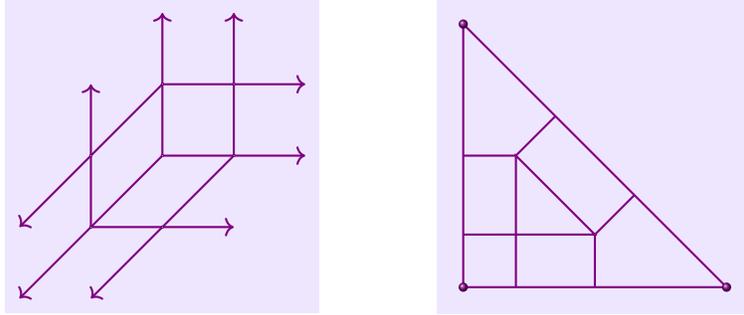

\subsection{Cohomology of elementary geometries} 
The use of elementary geometries in degeneration arguments comes from the fact that all cohomology classes are represented by boundary strata.

\begin{proposition}\label{prop: elementary-cohomology}
    Let $(Y|\partial Y)$ be an elementary toric threefold pair with dense torus $A$. The $A$-equivariant cohomology of $Y$ is spanned, as an $H^\star_A(\mathsf{pt})$-module, by classes of strata. 
\end{proposition}

For a smooth projective toric threefold, equivariant cohomology surjects onto ordinary cohomology, so the proposition also holds non-equivariantly. 

\begin{proof}
For a smooth projective toric threefold, the cycle map is an isomorphism~\cite[Chapter~5]{Ful93}, so the statement can be proved in Chow, where it follows from excision. Let $\mathsf{CH}^k_A(Y)$ denote the equivariant Chow group of codimension $k$, and let $U_k \subset Y$ be the open subset corresponding to cones of dimension $k-1$. The equivariant Chow cohomology of $U_k$ vanishes in degree $k$. Indeed, for elementary geometries $Y \setminus \partial Y$ is a product of an affine space and a torus, which is Chow trivial in all degrees, and $U_k$ has no strata of codimension $k$, so the claim follows from excision. Applying excision for $U_k \subset Y$ shows that the Chow groups of $Y$ are generated by codimension-$k$ strata contained in $\partial Y$. 
\end{proof}

\begin{remark}[Artin fan]
    We use a related fact. The pair $(Y|\partial Y)$ has an Artin fan $\mathsf a\Sigma$ and a map
    \[
    Y\to \mathsf a\Sigma.
    \]
    Since the action of the torus $A$ preserves $\partial Y$, we 
    have an snc log structure for the quotient stack $(Y/A, \partial Y/A)$.  From \cite{AW}, since the quotient map $Y \rightarrow Y/A$ is strict and smooth, it induces an isomorphism on Artin fans; in particular, we have a factorization
    \[ 
    Y \to Y/A \to \mathsf a\Sigma.
    \]
    So, if we consider $a\Sigma$ with the {\it trivial action} of $A$, the map $Y \rightarrow \mathsf{a}\Sigma$ is $A$-equivariant. The equivariant cohomology of $a\Sigma$ is
    \[
    H^\star_A(a\Sigma;\QQ) = H^\star(a\Sigma;\QQ)\otimes H^\star_A(\mathsf{pt};\QQ).
    \]
    For elementary geometries, Proposition \ref{prop: elementary-cohomology} can be restated:   the $A$-equivariant cohomology of $Y$ is additively spanned by equivariant classes pulled back from $a\Sigma$. The same is true for the divisorial strata. 
    
    The evaluation spaces are built from blowups and products of $Y$ and its codimension one strata. By K\"unneth and blowup formulas, this implies that the $A$-equivariant cohomology of the evaluation space is spanned by the $A$-equivariant cohomology of the space $\mathsf aP^{\bm \mu}_{\beta,k}(Y|\partial Y)$.
\end{remark}

\subsection{Stars} 

The notion of a star specifies partial discrete data for a GW/DT/PT enumerative problem. It is {\it partial} because it does not fix the genus/Euler characteric, nor the tangency at the boundary. 

Let $\underline{\Sigma}$ be a cone complex of an elementary toric threefold pair, viewed with its embedding $\underline{\Sigma} \subset \mathbb R^3$ and surjection $\pi: \mathbb R^3 \rightarrow \underline{\Sigma}$.
Given a point $p \in \underline{\Sigma}$ there is a well-defined set of {\it tangent vectors} at $p$ in $\mathbb R^3$; such a tangent vector based at $p$ is {\it allowable} if it points into $\underline{\Sigma}$. 
The set $\underline{\Sigma}_p$ of admissible tangent vectors is itself a cone complex, defined up to subdivision; by considering the action of $\pi$ on tangent spaces, we have a surjection $\pi_p: \mathbb R^3 \rightarrow \underline{\Sigma}_p$.  As we soon explain, $\underline{\Sigma}_p$ is the tropicalization of an elementary geometry $(Y_p|\partial Y_p)$ associated to an expansion of $Y$.



\begin{definition}[Stars]\label{def: star}
A {\it star} in $\Sigma$ is given by data $\{p,(v_1, n_1),\ldots,(v_m, n_m),k\}$, denoted $\mathsf v$.  The point $p$ is an integral point in $\Sigma$, which we call the {\it base} of the star; the vectors $v_i$ are primitive, integral, admissible tangent vectors based at $p$, each equipped with a positive integer weight $n_i$.   Finally, the non-negative integer $k$ denotes the number of internal markings.   The vectors are required to satisfy a {\it balancing condition} that we will describe next.
\end{definition}

\begin{definition}
Given a star $\mathsf v$ based at $p \in \Sigma$ and a linear projection $\phi:  \mathbb R^3 \rightarrow \mathbb R$, we say $\mathsf v$ is {\it balanced with respect to} $\phi$, if
\[
\sum_{i=1}^{m} n_i \phi(v_i) = 0.
\]
A star $\mathsf v$ is {\it balanced} if it is balanced with respect to every linear projection $\phi$ that factors through $\pi_p$, i.e. projections of the form:
\[
\phi: \mathbb R^3 \rightarrow \underline{\Sigma}_p \rightarrow \mathbb R.
\]
\end{definition}
Notice this balancing condition is local, in the sense that it only depends on $\mathsf v$ and $\underline{\Sigma}_p$. In the full boundary situation, with $\underline{\Sigma} = \mathbb R^3$, this balancing corresponds to the usual balancing condition in tropical geometry, but the condition on projections becomes more restrictive as we add non-boundary divisors.  In the $3$-non-boundary geometry, with $\underline{\Sigma} = \mathbb R_{\geq 0}^3$ and $p = \mathbf{0}$, the balancing condition is vacuous.

A star encodes geometric data, as follows. 

\begin{construction}
Let $\mathsf v$ be a star, with data $\{p,(v_1, n_1),\ldots,(v_m, n_m),k\}$ as above. As explained in Section~\ref{sec: degenerations}, the vertex of a star determines a toric threefold pair up to subdivision. Call it $Y_{\mathsf v}$; its tropicalization is the cone complex $\underline{\Sigma}_p$, up to subdivision.  Since $Y$ is elementary $Y_{\mathsf v}$ will be as well.

The primitive vectors $v_i$ determine rays in the fan of $Y_{\mathsf v}$ which correspond to divisors in the boundary $\partial Y_{\mathsf v}$.  The multiplicities $n_i$ record intersection numbers of a curve class $\beta$ with these divisors.

A star does not explicitly list the curve class. But in elementary geometries the intersection multiplicities $n_i$ uniquely specify a curve class.  Indeed, the boundary divisors generate the Picard group. The balancing condition ensures such a $\beta_{\sf v}$ exists.
\end{construction}

Given an elementary geometry $(Y|\partial Y)$ with tropicalization $\Sigma$, the above remark lets us use a star $\mathsf v$, based at $0 \in \Sigma$, to encode the data of curve class $\beta$ and contact orders.   Similarly, when defining Chow $1$-complexes in $\Sigma$, we can omit $\beta_{\sf v}$-labels.  As defined in Section \ref{combinatorial}, we have tropical moduli spaces of Chow (resp. stable Chow) $1$-complexes
\[\mathcal T_\mathsf{v}(Y|\partial Y), \ \mathcal T^{\bm \mu}_\mathsf{v}(Y|\partial Y), \qquad (\textnormal{resp.} \  T_\mathsf{v}(Y|\partial Y), \  T^{\bm \mu}_\mathsf{v}(Y|\partial Y)).\]



\subsection{Visibility}\label{sec: visibility} Stars arise when we express strata of GW/PT spaces in terms of components in degenerations. We will need to understand curves in non-rigid targets. As explained in~\cite{CN21}, this is much more subtle than in the smooth pair setting. The following combinatorial condition will help.

Fix a cone complex structure $\Sigma$.  Let $\mathsf v$ be a star based at $p$ which lies in the relative interior $\sigma^\circ$ of a unique cone $\sigma$ of $\Sigma$,  Let $L_p \subset \mathbb Z^3$ be the lattice associated to $\sigma$; we can
think of $L_p$ as the integral tangent vectors $w$ at $p$ which point along $\sigma$.

A similar lattice can be associated to a primitive admissible tangent vector $v$ based at $p$.  Choose $q \in \Sigma$ such that $q - p$ is in the direction as $v$, sufficiently close to $p$ so that it is contained in a cone adjacent to $\sigma$, and take $L_q$ as in the last paragraph.  If we again identify $L_q \subset \mathbb Z^3$, there are natural inclusions $L_p \subset L_q$ and $\mathbb{Z}v \subset L_q$.  We denote the quotient of the second inclusion by $L_v$.  Composition with the first inclusion defines a map of lattices
$$L_p \subset L_q \rightarrow L_v$$
which is independent of the choice of $q$.

\begin{definition}[Visibility]
    We say that a star ${\sf v}$ is {\it visible} if the induced map of tangent directions:
    \[
    L_p\to \prod_i L_{v_i}
    \]
    is injective. 
\end{definition}

\begin{remark}
    The visibility condition has a geometric meaning. A point $p$ in $\Sigma$ determines an elementary toric pair. Such a pair is a compactified torus bundle over a stratum of $(Y|\partial Y)$. It has a ``fiber torus'', i.e. the torus of this bundle, and $L_p$ is the associated cocharacter space. Each tangent direction determines a boundary divisor, itself a bundle over a stratum, and has its own fiber torus with cocharacter space $L_{v_i}$. The fiber torus of $L_p$ acts on the one of $L_{v_i}$ -- essentially always non-faithfully. The visibility condition is that, when all directions are taken together, the action is faithful. Visibility will play a role in the rigidification process, when we express strata invariants, which are associated with non-rigid targets, in terms of rigid invariants.
\end{remark}

We extend visibility to Chow $1$-complexes in $\Sigma$ in a way that lets us constrain automorphisms in the stack of expansions.
Given a Chow $1$-complex $\Gamma$, for each vertex $V$, we have a lattice $L_V$ as before; for each edge $E$ adjacent to $V$ we can define $L_E$ and we have a map $L_V \rightarrow L_E$.
We associate a lattice to $\Gamma$ by considering vectors in each $L_V$ which agree upon mapping to $L_E$ for a common edge:
\[
L_\Gamma: = \mathsf{ker}( \prod_{V} L_V \rightrightarrows \prod_{E} L_E).
\]
For each edge $E$, there is a well-defined map $L_\Gamma \rightarrow L_E$ defined using either vertex $V$ adjacent to $E$; similarly, for each unbounded ray $H$ of $\Gamma$,
there is a map $L_\Gamma \rightarrow L_H$.  Combining these, we
have:
\[
g_\Gamma\colon L_\Gamma \to\prod_E L_E \times \prod_{H} L_H.
\]
\begin{definition}
 A Chow $1$-complex $\Gamma$ is called {\it visible} if $g_\Gamma$ is injective. 
\end{definition}

Recall that in this section we work exclusively with elementary target geometries. 

\begin{lemma}\label{stable-visible}
Every stable Chow $1$-complex in an elementary geometry $\Sigma$ is visible.
\end{lemma}
\begin{proof}
First, observe that we have a Cartesian diagram
\[
\begin{tikzcd}
L_\Gamma \arrow{d}\arrow{r} &\prod_E L_E \times \prod_{H} L_H \arrow{d}\\
\prod_{V} L_V \arrow{r} &\prod_{E} (L_E\times L_E) \times\prod_{H} L_H.
\end{tikzcd}
\]
The bottom arrow factors as a product over stars of $\Gamma,$ so injectivity at each vertex would imply injectivity of the top arrow.
So it suffices to use stability to show that every star in $\Gamma$ is visible.

This can be checked with some casework; we do two of the cases, as the rest are similar. First, suppose $\underline{\Sigma}$ is $\mathbb R^3$. Stability implies that every vertex is at least trivalent, and there are  at least three distinct pairwise non-parallel tangent directions. Visibility follows by elementary linear algebra. Next, suppose $\underline{\Sigma}$ is $\mathbb R^2\times \mathbb R_{\geq 0}$. If the star is not based along $0$ in the third coordinate, we are in the previous case. The projection to $\mathbb R^2$ has to be balanced in the usual sense.  With two exceptions, stability again implies that every vertex is at least trivalent with non-parallel edge directions, in which case the argument for $\mathbb R^3$ applies. The two exceptions are when the vertex lies on a ray in the fan structure or $\mathbb R^2$, but with edges not contained in the ray, or when the vertex lies at the origin. In the latter case, visibility condition is vacuous, and in the former case, evaluation along either of the tangent direction verifies visibility. The remaining cases are similar. 
\end{proof}

We explain the relevance to isotropy groups. Fix a cone complex structure on $T_{\sf v}(\Sigma)$ and choose a cone $\tau \subset  T_{\sf v}(\Sigma)$, with {\it relative interior} denoted $\tau^\circ$.  Let $L_\tau$ denote the integer linear span of this cone.
This is the cocharacter lattice of the isotropy group of the point in the stack of expansions determined by the point associated to $\tau$, see~\cite{MR20}.  This lattice is sensitive to the choice of cone structure on
$T_{\sf v}(\Sigma)$.

The combinatorial type (i.e. the topological type of the $1$-complex, the cell structure, the cones in which the vertices and edges lie, and edge directions) is constant in the relative interior of a cell. We pick, arbitrarily, a Chow $1$-complex $\Gamma$
corresponding to a point in $\tau^\circ$. For each vertex $V$ in $\Gamma$, lying in the relative interior $\sigma^\circ$ of a cone of $\Sigma$, there is a natural map
\[
\tau^\circ\to \sigma^\circ
\]
recording the position of a vertex in its cone. Passing to tangent vectors, we have a map of lattices $L_\tau\to L_V$.
If we take the product over vertices, the map
\[
L_\tau\to \prod L_V,
\]
is injective since a $1$-complex is determined by the position of its vertices.
Compatibility with edges forces this map to factor through an inclusion
\[
L_\tau \hookrightarrow L_\Gamma.
\]
By combining this with Lemma \ref{stable-visible}, we have the following corollary.

\begin{corollary}
Given an elementary geometry $\Sigma$, for any cone $\tau$ in $T_{\sf v}(\Sigma)$
the map
\[
L_\tau \rightarrow \prod_{E} L_E \times \prod_H L_H
\]
is injective.
\end{corollary}

\subsection{Partial ordering on stars}


In this section, we define a partial ordering on pairs $(\underline{\Sigma}, \mathsf v)$
where $\underline{\Sigma}$ is the tropicalization of an elementary geometry and $\mathsf v$ is
a star based at $0$.

We first define a partial ordering on the targets themselves.  We will identify two tropicalizations $\underline{\Sigma}$ and $\underline{\Sigma}'$ if there exists a linear isomorphism of $\mathbb R^3$ 
which descends, via the piecewise linear surjections $\pi_{\Sigma}, \pi_{\Sigma'}$ to an isomorphism of cone complexes up to subdivision. Recall also that 
$\underline{\Sigma}_p$ is the cone complex (up to subdivision) of admissible tangent vectors at $p$.

\begin{definition}
Given $\underline{\Sigma}$ and $\underline{\Sigma'}$, we say that $\underline{\Sigma}'$ is {\it smaller} than $\underline{\Sigma}$, written
as $\underline{\Sigma'} \preccurlyeq \underline{\Sigma}$, if there
exists a point $p \in \underline{\Sigma}$ such that 
$$\underline{\Sigma}_p = \underline{\Sigma}'$$
after applying a linear change of coordinates to $\mathbb R^3$.   
\end{definition}

The condition says a log birational model of $(Y'| \partial Y')$ occurs as a component in an expansion of $(Y|\partial Y)$.

If 
$\underline{\Sigma'} \preccurlyeq \underline{\Sigma}$ then either $\underline{\Sigma'}$ occurs earlier in the list of elementary geometries
or $\underline{\Sigma}' = \underline{\Sigma}$ up to change of coordinates.

Before defining the partial ordering on stars, we first introduce a few preliminary ideas.  Given the tropicalization $\underline{\Sigma}$ of an elementary geometry,
there exists a subgroup $\mathbb{G}_\Sigma$ of translations which preserve $\underline{\Sigma}$.
Concretely, this group is the lattice of integral vectors $w$ for which $w, -w \in \underline{\Sigma}$.  We can describe them explicitly:
\begin{enumerate}[(i)]
\item {\bf Full boundary.} $\mathbb G_\Sigma = \mathbb Z^3$.
\item {\bf $1$-Non-boundary.} $\mathbb G_\Sigma = \mathbb Z^2$ in the straight geometry, $\mathbb G_\Sigma = \mathbb Z$ in the ruled geometry, and $\mathbb G_\Sigma = 0$ in the 
general case.  
\item {\bf $2$-Non-boundary.} $\mathbb G_\Sigma = \mathbb Z$ in the straight geometry, and $\mathbb G_\Sigma = 0$ in general geometry; 
\item {\bf $3$-Non-boundary.} $\mathbb G_\Sigma = 0$.
\end{enumerate}

Given a Chow $1$-complex $\Gamma$ in $\underline{\Sigma}$, we have two ways to assign stars to $\Gamma$. First, the asymptotic star $\mathsf v$ of $\Gamma$ is the star based at the $0$ obtained from the unbounded rays of $\Sigma$, adding mulitplicities of rays with the same direction. More precisely,
we construct the cone 
\[
\mathsf{Cone}(\Gamma) \subset \underline{\Sigma} \times \mathbb{R}_{\geq 0}
\]
and take the fiber over $0$, with appropriate weights on the rays.
We say that $\Gamma$ is a \textit{degeneration} of $\mathsf v$.  

Second, if we take a vertex $p$ of $\Gamma$, then the star $\mathsf{Star}_\Gamma(p)$ is obtained by forgetting all other vertices and extending the half-edges adjacent to $p$ into rays.

\begin{definition}
A star $(\underline{\Sigma}', \mathsf w)$ is {\it smaller} than $(\underline{\Sigma}, \mathsf v)$, written as
$\mathsf w \preccurlyeq \mathsf v$, if either
\begin{enumerate}[(i)]
\item $\underline{\Sigma}' \prec \underline{\Sigma}$, or
\item after a linear transformation, $\underline{\Sigma}' = \underline{\Sigma}$ and there exists a degeneration $\Gamma$ of $\mathsf v$, and a vertex $p$ of $\Gamma$ such
that $\mathsf w = \mathrm{Star}_\Gamma(p)$ after applying a translation in $\mathbb G_\Sigma$.
\end{enumerate}
\end{definition}

The definition is captured by Figure~\ref{fig: ordering-of-stars} below. 

\begin{figure}[h!]
\begin{tikzpicture}[x=0.75pt,y=0.75pt,yscale=-1,xscale=1]

\definecolor{lavender}{RGB}{238,230,255}
\fill[lavender] (0,60) rectangle (600,400);

\coordinate (A1) at (120,300);
\draw[violet, ->, line width=0.75] (A1) -- ++(40,-40);   
\draw[violet, ->, line width=0.75] (A1) -- ++(-40,-40);  
\draw[violet, ->, line width=0.75] (A1) -- ++(40,40);    
\draw[violet, ->, line width=0.75] (A1) -- ++(-40,40);   
\shade[ball color=violet] (A1) circle (1.5pt);
\node[violet, anchor=south] at (A1) {$\mathsf v$};

\def\scaleSquare{0.8}
\def\xshift{5}
\coordinate (B1orig) at (300,150);
\coordinate (B2orig) at (400,150);
\coordinate (B3orig) at (400,250);
\coordinate (B4orig) at (300,250);

\coordinate (B1) at ($(\xshift,0) + \scaleSquare*(B1orig)$);
\coordinate (B2) at ($(\xshift,0) + \scaleSquare*(B2orig)$);
\coordinate (B3) at ($(\xshift,0) + \scaleSquare*(B3orig)$);
\coordinate (B4) at ($(\xshift,0) + \scaleSquare*(B4orig)$);

\draw[violet, line width=0.75] (B1) -- (B2) -- (B3) -- (B4) -- (B1);
\draw[violet, ->, line width=0.75] (B2) -- ++(16,-16);   
\draw[violet, ->, line width=0.75] (B1) -- ++(-16,-16);  
\draw[violet, ->, line width=0.75] (B3) -- ++(16,16);    
\draw[violet, ->, line width=0.75] (B4) -- ++(-16,16);   

\foreach \v in {B1,B2,B3,B4} {
    \shade[ball color=violet] (\v) circle (1.5pt);
}

\draw[violet, line width=0.75] (B3) circle (6pt);
\node[violet, anchor=west] at ($(B3)+(8,0)$) {$p$};
\node[violet, anchor=south] at ($(B1)+(40,-10)$) {$\Gamma$};

\coordinate (C1) at (480,300);
\draw[violet, ->, line width=0.75] (C1) -- ++(-40,0);    
\draw[violet, ->, line width=0.75] (C1) -- ++(0,-40);    
\draw[violet, ->, line width=0.75] (C1) -- ++(40,40);    
\shade[ball color=violet] (C1) circle (1.5pt);
\draw[violet, line width=0.75] (C1) circle (6pt);
\node[violet, anchor=west] at ($(C1)+(6pt,0)$) {$\mathsf w$};

\node[violet] at ($(A1)+(0,60)$) {Initial star $\mathsf v$};
\node[violet] at ($($(B1)!0.5!(B3)$)+(0,75)$) {A degeneration of $\mathsf v$};
\node[violet] at ($(C1)+(0,60)$) {A smaller star $\mathsf w$};

\end{tikzpicture}
\caption{A depiction of the ordering on stars induced by taking degenerations and then stars at vertices. In this case $\mathsf w$ on the right is smaller than $\mathsf v$.}\label{fig: ordering-of-stars}
\end{figure}
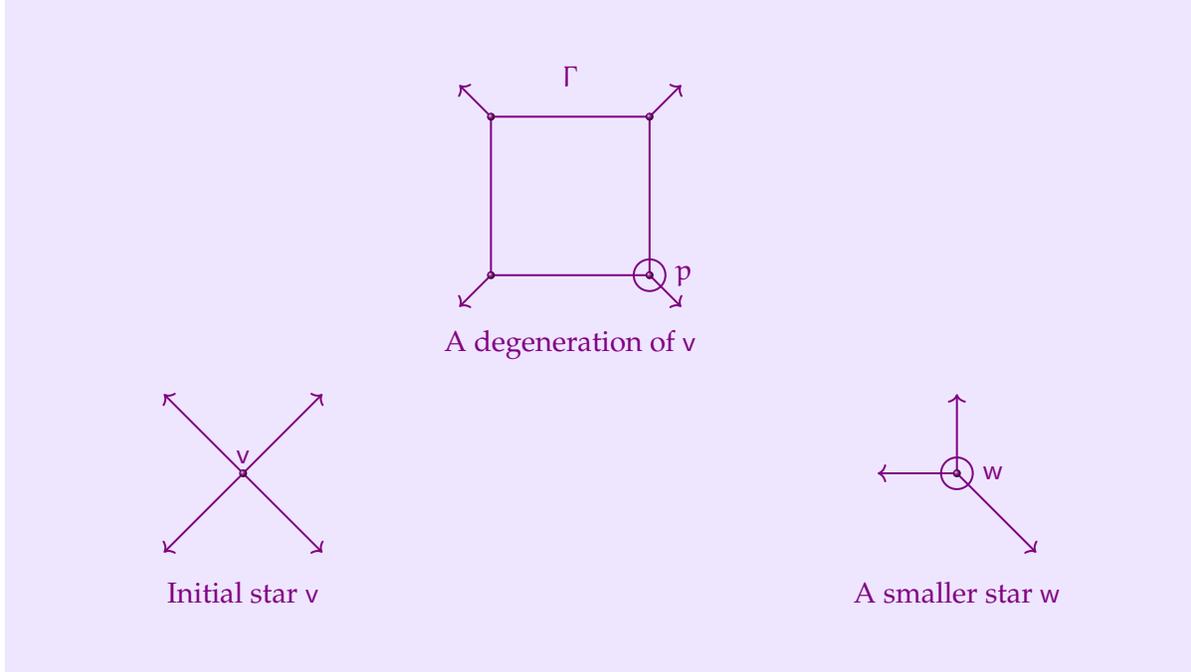

The following proposition shows this partial ordering is well-suited for inductive arguments.  

\begin{proposition}

\begin{enumerate}[(i)]
\item If $(\underline{\Sigma}', \mathsf w)$ is smaller than $(\underline{\Sigma}, \mathsf v)$ and vice versa, then they are equal to each other.

\item For a fixed pair $(\underline{\Sigma}, \mathsf v)$, the set of pairs $(\underline{\Sigma}, \mathsf w)$ with $\mathsf w \preccurlyeq \mathsf v$
is finite.
\end{enumerate}
\end{proposition}
In the second part, notice that we are considering stars with the same target as $\mathsf v$.

\begin{proof}

These are basically proven in~\cite[Propositions~3.2.2,~3.3.2]{MR25}. In the full boundary case, the partial ordering here is the same as that of \cite[Section 3.2]{MR25}.
In the remaining cases, suppose $\Gamma$ is a degeneration of $\mathsf v$ and $\mathsf w$ is the star of $\Gamma$ at the vertex $p$, 
Note that $p$ can be translated to $0$ by an element of $\mathbb G_\Sigma$. Replacing $\Gamma$ with its translation, we assume $p = 0$.  
For elementary geometries, the curve class is determined by the weights on edges/rays so we can promote $\Gamma$ and $\mathsf v$ to 
Chow $1$-complexes as used in~\cite{MR23,MR25}; after making the changes, the ordering $\mathsf w \preccurlyeq \mathsf v$ agrees with the ordering in \cite[Definition~3.3.1]{MR25}. The proposition follows.
\end{proof}

\section{Exotic insertions and rubber calculus}

To prove the correspondence, we induct on the discrete data via the star. 
This requires a formulation of the correspondence for strata. 
Although such an extension holds, care is needed: strata invariants need not be rational functions, so we must avoid working directly with them; see Section~\ref{sec: strata-conj}.

Even starting from non-exotic invariants, the induction necessarily passes through strata invariants, whose insertions are exotic. So, to apply induction, two conversion rules are required. 

For the exotic-to-non-exotic conversion, we adapt ideas from our earlier work on GW reconstruction algorithms~\cite{MR25}, formulating a generalized GW/PT correspondence suited to this reconstruction. 
This is carried out at the beginning of this section.

The strata-to-non-exotic conversion is more involved---we begin in Section~\ref{sec: rubber-calc-plan} and it occupies the remainder of the section. 
The conversion rule is based on a ``rubber calculus'' that rigidifies the strata and applies the degeneration formula to express strata invariants in terms of non-exotic invariants of rigid targets, up to lower-order corrections. 
A key subtlety—and the reason the GW reconstruction techniques of~\cite{MR25} do not apply—is that this procedure must respect the correspondence.

\subsection{An identity without virtual structures.}

The GW and PT moduli spaces share two structures: a map to the evaluation space $\mathsf{Ev}$ and a map to the stack $\mathsf a T$ of stable Chow $1$-complexes. Combining them yields a common space with trivial virtual structure. Consider the fiber product
\[
\begin{tikzcd}
\mathscr M \arrow{d} \arrow{r} & \mathsf{Ev} \arrow{d}\\
\mathsf a T \arrow{r} & \mathsf a \mathsf P.
\end{tikzcd}
\]
The space $\mathscr M$ parametrizes expansions of $Y$ with a labeled collection of points compatible with the discrete data of $\mathsf P$. We obtain maps
\[
\begin{tikzcd}
    \coprod_{g,\bm\mu} \mathsf{GW}^{\bm\mu}_{\beta,k}(Y|\partial Y)_g 
    \arrow[swap]{dr}{\rho_{\sf GW}} 
    & & 
    \coprod_{\chi,\bm\mu} \mathsf{PT}^{\bm\mu}_{\beta,k}(Y|\partial Y)_\chi
    \arrow{dl}{\rho_{\sf PT}}\\
    & \mathscr M.
\end{tikzcd}
\]
Both morphisms are combinatorially flat and admit virtual pullbacks. We write $\rho$ when working simultaneously on both sides.


Exotic insertions are defined by taking a subdivision $\mathsf{Ev}^{\diamond} \rightarrow \mathsf{Ev}$ and a compatible map $\mathscr{M}^{\diamond} \rightarrow \mathsf a T^{\diamond}$. 


Given a cone $\sigma \subset T^{\diamond}$, the closure of the corresponding stratum defines a closed substack 
$S_\sigma \subset \mathsf a T^{\diamond}$, and the image of $\sigma$ in $\mathsf P^{\diamond}$ determines 
a stratum $\mathsf{Ev}_{\sigma} \subset \mathsf{Ev}^{\diamond}$.  
Given in addition a class $\vartheta \in H^\star_{A}(\mathsf T)$ and $\upalpha \in H^\star_{A}(\mathsf{Ev}_{\sigma})$, 
we pull $\vartheta$ and $\upalpha$ back to $S_\sigma$ and push forward to $\mathscr{M}$ to obtain
\[
[\![\sigma, \vartheta, \upalpha]\!] \in H^\star_{A}(\mathscr{M}).
\]
We call $[\![\sigma, \vartheta, \upalpha]\!]$ a non-exotic stratum class if $\upalpha$ is pulled back from a stratum in a product of blowups of $\mathsf{Ev}_i$, and strongly non-exotic if it is pulled back from a stratum of $\mathsf{Ev}$ itself.

A key proposition of \cite[Section~4]{MR25} is the following.

\begin{proposition}\label{prop: precise-version}
Any, possibly exotic, class of the form $[\![ \sigma,\vartheta,\upalpha ]\!]$ can be expressed in the form
\[
[\![ \sigma,\vartheta,\upalpha ]\!] = \sum_j [\![ \sigma_j,\vartheta_j,\upalpha_j ]\!]
\]
where $\sigma_j$ is a cone of $T$ containing $\sigma$ in its closure and $\upalpha_j$ is non-exotic. 
\end{proposition}

We do not recall the proof here, but note that it follows by approximating an exotic class by a non-exotic one, and then systematically using the blowup formula to account for the difference. 


\subsection{Strata conjectures}\label{sec: strata-conj}

We define stratum GW/PT invariants by pulling back classes in $H^\star_{A}(\mathscr M)$ along $\rho$ and pairing with the appropriate virtual class:
\[
\langle [\![\sigma, \vartheta, \upalpha]\!]\mid \bm{\mu}\rangle^{\mathsf{GW}}_{\beta,g}, \qquad
\langle [\![\sigma, \vartheta, \upalpha]\!]\mid \bm{\mu}\rangle^{\mathsf{PT}}_{\beta,\chi}.
\]
The generating functions $\mathsf{Z}_{\mathsf{GW}}(u)$ and $\mathsf{Z}_{\mathsf{PT}}(q)$ are defined in the usual way, and we extend the primary GW/PT correspondence to stratum invariants. We maintain the notation of Conjecture~\ref{conj: gw/pt}.

\begin{conjecture}
The PT series
\[
\mathsf{Z}_{\mathsf{PT}}\!\left(Y\mid\partial Y; q \mid [\![\sigma, \vartheta, \upalpha]\!]\mid \bm{\mu}\right)_{\beta}
\]
is the Laurent expansion of a rational function in $q$. Furthermore,
\[
(-iu)^{d_\beta+\sum \ell(\mu_j)-|\mu_j|}
\, \mathsf{Z}_{\mathsf{GW}}\!\left(Y\mid\partial Y; u \mid [\![\sigma, \vartheta, \upalpha]\!]\mid \bm{\mu}\right)_{\beta}
=
(-q)^{-d_\beta/2}
\, \mathsf{Z}_{\mathsf{PT}}\!\left(Y\mid\partial Y; q \mid [\![\sigma, \vartheta, \upalpha]\!]\mid \bm{\mu}\right)_{\beta}
\]
under the identification $-q = e^{iu}$.
\end{conjecture}

\begin{remark}[Importance of stability]
It is important in this version of the conjecture that we work with strata from the space $T(Y| \partial Y )$ of {\it stable} Chow $1$-complexes. If we work with strata on the unstable moduli space $\mathcal{T}(Y|\partial Y)$ the PT series fails to be rational. We give an example below. 
\end{remark}

\begin{example}
Consider $(Y,\partial Y) = \mathbb{C}^2\times (\mathbb{P}^1|\infty)$, which has tropicalization $\Sigma_Y = \mathbb{R}_{\geq 0}$, and consider $\beta = [\mathbb{P}^1]$.  In $\mathcal{T}(Y|\partial Y)$, there is a ray spanned by the (unstable) $1$-complex with vertices at $0$ and $1$.  The corresponding strata PT invariant is computed in \cite{OP10} to 
be proportional to $\mathrm{log}(1+q)$.  

\end{example}

\begin{corollary}
For fixed discrete data, the GW/PT correspondence for an exotic strata invariant $[\![ \sigma,\vartheta,\upalpha ]\!]$ follows from the GW/PT correspondence for nonexotic strata invariants
$[\![ \sigma_j,\vartheta_j,\upalpha_j ]\!]$ with $S_{\sigma_{j}}$ contained in the closure of the image of $S_\sigma$.
\end{corollary}

\begin{proof}
    Apply virtual pullback along $\rho$ to the statement Proposition~\ref{prop: precise-version}.
\end{proof}

\subsection{Rubber calculus: plan}\label{sec: rubber-calc-plan}

Our goal is to convert non-exotic strata invariants into standard ones. Given a non-exotic stratum insertion $[\![\sigma, \vartheta, \upalpha]\!]$, we use $\sigma$ to construct an snc degeneration of $(Y|\partial Y)$. The resulting GW/PT moduli space has a distinguished component, generically a bundle over $S_\sigma$, which can be analyzed via the degeneration formula. 
The class $\vartheta$ is then converted to a deeper stratum, yielding an inductive reduction to lower-order invariants.

For smooth pairs, this collection of techniques is known as ``rubber calculus''~\cite{MP06}; we extend it to the snc setting. A related procedure in GW theory appears in~\cite[Section~5]{MR25}, but it involves steps incompatible with GW/PT, including induction on the genus. 
We present a different approach based on the structure of elementary geometries.

\subsection{Strata to degenerations}\label{sec: rigidification}

Fix a star $\mathsf v$ with $k$ internal markings and partitions $\bm\mu$ for the directions in $\mathsf v$ to record the tangency. Also fix a stable Chow $1$-complex $\Gamma$ in $\Sigma_Y$ with fixed tangency, which is a specialization of $\mathsf v$, i.e., the asymptotic star of $\Gamma$ is $\mathsf v$. This includes the possibility that $\Gamma = \mathsf v$.

Taking the cone over $\Gamma$ gives a map $\Sigma_\Gamma \rightarrow \RR_{\geq 0}$, inducing a degeneration
\[
\cY \rightarrow \A^1.
\]  
There is a space of stable Chow $1$-complexes in $\Sigma_\Gamma$ that are ``vertical'' with respect to the projection to $\RR_{\geq 0}$, denoted $T_\beta(\Sigma_\Gamma/ \RR_{\geq 0})$. This is the underlying tropical moduli space for the GW and PT spaces associated with the degeneration $\cY$. The fiber over $0$ of this space is $T_{\mathsf v}$; see~\cite[Section~7]{MR23}.

The space $T_\beta(\Sigma_\Gamma/ \RR_{\geq 0})$ contains a distinguished ray $\rho_\Gamma$, whose primitive generator is the $1$-complex $\Gamma$ itself. We get a cone complex:
\[ T_\Gamma := \mathsf{Star}(\rho_\Gamma)\]
whose Artin fan $\mathsf{a}T_\Gamma$ is an irreducible component in the degeneration of $\mathsf{a}T_{\sf v}$ associated to $T_\beta(\Sigma_\Gamma/ \RR_{\geq 0})$. The space $T_\Gamma$ determines a virtual irreducible component in each of the GW/PT spaces for the degenerate fiber $\cY_0$. Picking these out, we have a Cartesian square:
\[
\begin{tikzcd}
    \mathsf{GW}_\Gamma, \ \mathsf{PT}_\Gamma\arrow{d}\arrow{r} & \mathsf{GW}_\beta(\mathcal Y_0), \ \mathsf{PT}_\beta(\mathcal Y_0)\arrow{d}\\
    \mathsf a T_\Gamma\arrow{r} & \mathsf a T_\beta(\Sigma_\Gamma/ \RR_{\geq 0}).
\end{tikzcd}
\]

\subsection{Evaluation structure in the strata}

For each marking $s$ of $\Gamma$—either an internal marking or a part of a partition in $\bm \mu$—we have tropical evaluation spaces $P_s$ and evaluation spaces $\mathsf{Ev}_s$, which are the strata of $(Y|\partial Y)$ determined by $s$, as appropriate.

We also have evaluation spaces associated with the \emph{bounded edges} of $\Gamma$. For an edge $E$ of $\Gamma$ labelled with a multiplicity $n_E$, there is a space $P_E$ of $n_E$ unordered points in $\mathsf{Star}_E(\Gamma)$. The partition is not fixed at this stage, and it is important for GW/PT compatibility that it does not need to be.

There is a universal surface $\cD_E \rightarrow \mathsf{a}P_E$, and we define $\mathsf{Ev}_E$ to be the relative symmetric power
\[
\mathsf{Ev}_E := \mathsf{Sym}^{n_E}(\cD_E) \rightarrow \mathsf{a}P_E.
\]
We use a symmetric power because, unlike the unbounded rays, the tangency profiles are not fixed.  

We denote 
\[
P_\Gamma := \prod P_s \times \prod P_E, \qquad 
\mathsf{Ev}_\Gamma := \prod \mathsf{Ev}_s \times \prod \mathsf{Ev}_E.
\]
We have evaluation maps
\[
\mathsf{GW}_\Gamma, \mathsf{PT}_\Gamma \rightarrow \mathsf{Ev}_\Gamma.
\]
As before, to make uniform statements on the GW/PT sides, we define a new space via pullback:
\[
\begin{tikzcd}
\mathscr M^{\mathsf{node}}_\Gamma \arrow{d} \arrow{r} & \mathsf{Ev}_\Gamma \arrow{d} \\
\mathsf{a} T_\Gamma \arrow{r} & \mathsf{a}P_\Gamma.
\end{tikzcd}
\]
This space receives virtually smooth maps from $\mathsf{GW}_\Gamma$ and $\mathsf{PT}_\Gamma$ spaces.

\subsection{Stratum to non-exotic} Fix a cone $\sigma \subset T_\Gamma$, with associated stratum $S_\sigma \subset \mathsf{a}T_\Gamma$.
 The image of $\sigma$ in $\mathsf P_\Gamma$ defines
a stratum $\mathsf{Ev}_{\sigma} \subset \mathsf{Ev}_\Gamma$.
If we have cohomology classes $\vartheta \in H^\star_{A}(\mathsf{a}T_\Gamma)$ and  $\upalpha \in H^\star_{A}(\mathsf{Ev}_\sigma)$,
we again have a cohomology class
\[
[\![\sigma, \vartheta,\upalpha]\!]_{\Gamma} \in H^\star_{A}(\mathscr M^{\mathsf{node}}_\Gamma).
\]
\begin{definition}[Broken stratum invariants]
    The degree of the virtual pullback of classes of the form $[\![\sigma, \vartheta,\upalpha]\!]_{\Gamma}$ to the GW/PT moduli spaces is referred to as a {\it broken non-exotic stratum invariant}.
\end{definition}

We formulate the GW/PT correspondence as in the last section, but only work with {\it non-exotic} insertions, so $\upalpha$ is a class on the product of evaluation spaces.
If $\Gamma = \mathsf v$, the invariants are non-exotic stratum invariants discussed earlier. We show these can be determined inductively. 

Write $\sigma = \cdot$ for the zero-dimensional cone, corresponding to $S_\sigma = \mathsf{a}T_\Gamma$:

\begin{proposition}
\begin{enumerate}[(i)]

\item  When $\sigma = \cdot$ and $\vartheta = 1$, 
the GW/PT series for $[\![\sigma, \vartheta,\upalpha]\!]_{\Gamma}$
is determined by exotic invariants associated to the vertices of $\Gamma$.

\item When $\mathrm{dim}(\sigma) > 0$ and $\vartheta = 1$, the GW/PT series for $[\![\sigma, \vartheta,\upalpha]\!]_{\Gamma}$
equals the GW/PT series for 
$$[\![\cdot, 1,\upalpha']\!]_{\Upsilon}$$
where $\Upsilon$ is obtained from $\Gamma$ by taking a degeneration.

\item When $\mathsf{deg}(\vartheta) > 0$, then 
the GW/PT series for 
$[\![\sigma, \vartheta,\upalpha]\!]_{\Gamma}$
is determined by a $H^{\ast}_{A}(\mathrm{pt})$-linear combination of 
the GW/PT series for
$$[\![\sigma_i, \vartheta_i,\upalpha_i]\!]_{\Gamma_i}$$
where 
$$\mathrm{deg}(\vartheta_i) < \mathrm{deg}(\vartheta)$$
and
each $\Gamma_i$ is obtained from $\Gamma$ by taking a degeneration.

\end{enumerate}

In all cases, these reductions are compatible with the GW/PT correspondence.

\end{proposition}

In combination with the conversion from exotic to non-exotic stratum invariants, this proposition gives a recursive procedure for converting exotic invariants into nonexotic invariants, compatibly with the GW/PT correspondence.  Statement (iii) lets us replace non-exotic stratum terms with nontrivial $\vartheta$ with broken nonexotic strata invariants with $\vartheta = 1$. Statement (ii) lets us replace these terms with broken invariants with $\sigma = \cdot, \vartheta = 1$.  Finally, statement (i) expresses these in terms of exotic invariants of vertices which are lower-order than the star we started with. 

We record a general construction that will be helpful in the proof. 

\begin{remark}\label{rem: degeneration-trick}
    Let $\Sigma\to\RR_{\geq 0}$ a surjective map of cone complexes. Let $\rho$ be a ray that maps surjectively  to the base $\mathbb R_{\geq 0}$ and let $\tau$ be a cone that contains $\rho$ as a face. Let $t$ be a point in the interior of $\tau$. Let $\widetilde\Sigma\to \RR_{\geq 0}$ be any subdivision that introduces the ray through $t$ as one of the rays and call this ray $\delta$. 

    The cone complexes $\Sigma\to\RR_{\geq 0}$ and $\widetilde \Sigma\to\RR_{\geq 0}$ induce \textit{relative} Artin fans\footnote{Abramovich insiste sur le fait que les fans d'Artin devraient en réalité être appelés fans d'Olsson. En revanche, Wise propose ``éventails d'Olsson'' pour désigner les éventails d’Artin relatifs.} over $\mathbb A^1$, by taking the induced maps of Artin fans and pulling back via $\A^1\to[\A^1/\mathbb G_m]$. Call these $\mathfrak a\Sigma\to\A^1$ and $\mathfrak a{\widetilde\Sigma}\to\A^1$; we have used $\mathfrak a$ rather than $\mathsf a$ to distinguish these from ordinary Artin fans. The maps are flat with $0$-dimensional fibers, and we make base changes so the fibers are reduced. 

    The components dual to $\delta$ and $\rho$ are Artin fans. The stratum $\mathfrak B_\tau$ corresponding to $\tau$ is an irreducible stack of dimension $-(\dim \tau-1)$; specifically, it is a $B\mathbb G_m^{(\dim \tau-1)}$-torsor over an Artin fan. By the usual toric/fan dictionary, we have maps
    \[
    \mathsf a {\widetilde \Sigma}(\delta)\to \mathfrak B_\tau\subset \mathsf a \Sigma(\rho).
    \]
    We have used the notation ${\widetilde \Sigma}(\delta)$ and $\Sigma(\rho)$ for the star fan -- the quotient of the cones containing $\delta$ (resp. $\rho$) by the action of $\delta$ (resp. $\rho$). 

    Again by the structure theory of toric morphisms, the first arrow is a ``toric fibration'', meaning an equivariant compactification of a torus torsor.
\end{remark}

\begin{proof}

For part (i), since there are no strata or $\vartheta$ terms and since the evaluation insertion $\upalpha$ is nonexotic, 
we can apply the degeneration formula to split the GW/PT series across the vertices of $\gamma$.  The nodal contribution to $\upalpha$ is a product of contributions $\upalpha_E \in H^\star_{A}(\mathsf{Ev}_E)$ of each edge.
When we split an edge $E$, we can arbitrarily choose which of the two vertices incident to $E$ is assigned $\upalpha_E$.  For each partition $\mu$ occurring in the degeneration formula, the relevant Nakajima correspondence lives over the symmetric power $\mathsf{Sym}^{n_E}(\cD_E)$, so it is automatically compatible with
multiplication by $\upalpha_E$.

The argument for (ii) and (iii) is more involved and is carried out below. The stack $S_\sigma$ is negative-dimensional, so the idea is to rigidify its generic stabilizer by replacing it with a bundle, given by $\mathsf aT_\Upsilon$. Doing so allows us to trade the $\vartheta$-insertion 
for a boundary class.  We use the term "rubber calculus" as shorthand for this series of moves. Additional details on a similar procedure can be found in~\cite[Section~5]{MR25}. Now the details.

\noindent
{\sc Step I. Further degeneration.} We use the construction outlined in Remark~\ref{rem: degeneration-trick}, and set it up as follows. We fix an initial choice of $\Gamma$. The cone over it determines a family $\Sigma_\Gamma\to \RR_{\geq 0}$. By taking the associated family of vertical $1$-complexes, we get a family of moduli spaces $T_\beta(\Sigma_\Gamma/\RR_{\geq 0})\to \RR_{\geq 0}$. Note that by construction, the space $T_\beta(\Sigma_\Gamma/\RR_{\geq 0})$ is a subdivision of the constant family $T_\beta(\Sigma)\times\RR_{\geq 0}$.

Tautologically, the cone complex $T_\beta(\Sigma_\Gamma/\RR_{\geq 0})\to \RR_{\geq 0}$ has a distinguished ray $\rho_\gamma$ corresponding to $\Gamma$. Now, the choice of $\Upsilon$ determines a point in a cone $T_\beta(\Sigma_\Gamma/\RR_{\geq 0})$ adjacent to $\rho_\gamma$. As in the remark, we can subdivide along the point $[\Upsilon]$ and take the induced maps on components. We get maps  
\[
\mathsf aT_\Upsilon\to S_\sigma\subset \mathsf a T_\Gamma. 
\]
We have used shorthand $\mathsf aT_\Upsilon$ for the component determined by the ray through $[\Upsilon]$ and $\mathsf aT_\Gamma$ for the one determined by $\Gamma$.

We view $\mathsf aT_\Upsilon$, and also the pullbacks of the GW/PT spaces along the map to $S_\sigma$, as a rigidified version of the stratum.

\noindent
{\sc Step II. Rigidification.} The rigidified spaces have larger dimension, so we use a ``rigidifying insertion'' to express the series of the stratum in terms of that of the rigidified space. As in the definition of $\mathscr M^{\mathsf{node}}_{\gamma}$, we can take a symmetric power of the universal surfaces over the stratum $S_\sigma$. This records evaluation at the markings and the nodes. The result
$\mathscr S^{\mathsf{node}}_{\sigma}$ fits in a Cartesian diagram
\[
\begin{tikzcd}
\mathscr M^{\mathsf{node}}_{\Upsilon} \arrow{d}\arrow{r} & \mathscr S^{\mathsf{node}}_{\sigma} \arrow{d}\\
\mathsf a T_{\Upsilon} \arrow{r} &S_\sigma.
\end{tikzcd}
\]
Both horizonal maps are toric fibrations, whose general fibers are compactifications of a torus $A_\sigma$ -- the torus with cocharacter space equal to the groupification of $\sigma$.  We have an evaluation space
$\mathsf{Ev}_\sigma:= \prod \widetilde{\mathsf{Ev}}_E \times \prod \widetilde{\mathsf{Ev}}_s$, for $\mathscr S^{\mathsf{node}}_{\sigma}$, equipped with an evaluation
map
\[
\mathscr S^{\mathsf{node}}_{\sigma} \rightarrow \mathsf{Ev}_\sigma
\]
that fits in a commutative diagram
\[
\begin{tikzcd}
\mathscr M^{\mathsf{node}}_{\Upsilon} \arrow{d}{f}\arrow{r} & 
\mathsf{Ev}_\Upsilon \arrow{d}\\
\mathscr S^{\mathsf{node}}_{\sigma} \arrow{r} &\mathsf{Ev}_\sigma.
\end{tikzcd}
\]
The right arrow is a toric fibration; the general fiber is a compactification of $\prod A_E \times \prod A_s$. The diagram is {\it not} Cartesian -- the generic fiber of the left vertical map is a compactification of $A_\sigma$.

The following result lets us trade invariants on $S_\sigma$ for invariants on $T_{\Upsilon}$
\begin{lemma}
There exists a cohomology class $\overline{\epsilon} \in H^\star_{A}(\mathsf{Ev}_{\Upsilon})$ such that 
\[
f_\star(\overline{\epsilon}\cap[\mathscr M^{\mathsf{node}}_{\Upsilon}]) = [\mathscr S^{\mathsf{node}}_{\sigma}].
\]
\end{lemma}

This allows us to conclude the analogue of (ii) for the non-virtual spaces $\mathscr M_\Upsilon^{\sf node}$. 

\begin{proof}

Let $F_\Upsilon$ denote the fiber product of $\mathsf{Ev}_\Upsilon$ and $\mathscr S^{\sf node}_\sigma$ over $\mathsf{Ev}_\sigma$. We of course have a map
\[
\mathscr M^{\mathsf{node}}_{\Upsilon}\to F_\Upsilon
\]
Over the generic point of $\mathscr S_\sigma^{\sf node}$, the fiber product $F_\Upsilon$ is a compactified torsor under the product torus $\prod A_E\times \prod A_s$. The group $A_\sigma$ has a natural map 
\[
A_\sigma \rightarrow \prod A_E \times \prod A_s.
\]
We claim this map is finite onto its image. Indeed, this can be checked at the level of cocharacter lattices, and amounts to the map on cocharacter lattices being injective. Since we are working in elementary geometries, the results of Section~\ref{sec: visibility} exactly guarantee this.

The map $\mathsf{Ev}_\Upsilon\to\mathsf{Ev}_\sigma$ is generically a toric variety fiber bundle. We can assume this map is also projective, so factors through a constant projective space bundle. Choose a fiberwise hyperplane. Consider pullbacks of powers of this divisor to $F_\Upsilon$, and then to $\mathscr M_\Upsilon^{\sf node}$. The map $\mathscr M^{\mathsf{node}}_{\Upsilon}\to F_\Upsilon$ is finite -- this follows from visibility. By taking a power, we obtain a cycle in $\mathscr M_\Upsilon^{\sf node}$ that is generically finitely onto $\mathscr S^{\sf node}_\sigma$. By scaling, we can ensure that it cycle pushes forward to the fundamental class. Any equivariant lift of the class on $\mathsf{Ev}_\Upsilon$ that pulls back to this cycle can be taken as $\overline \epsilon$.
\end{proof}

\noindent
{\sc Step III. Removal of the $\vartheta$ class.} The action of $A$ on our Artin fans is trivial, so we can write $H^\star_{A}(\mathsf T_\Gamma) = H^\star(\mathsf T_\Gamma)\otimes H^\star_{A}(\mathrm{pt})$.  
By linearity, it suffices to consider $\vartheta$ of the form $\vartheta' \otimes 1$ with $\deg \vartheta' > 0$.  Since $\mathsf{a} T_{\Upsilon}$ is an Artin fan, the pullback $g^\star\vartheta$ is supported on boundary strata. This follows, for example, from the description of the (non-$A$-equivariant) cohomology of the Artin fan as piecewise polynomials~\cite[Theorem~B]{MR21}. So there are cones $\sigma_i \subset T_\Upsilon$ and cohomology classes $\vartheta_i$ such that
\[
g^\star\vartheta = \sum_i \vartheta_i [S_{\sigma_{i}}]
\]
with $\deg \vartheta_i < \deg \vartheta$.
If we set $\epsilon = \overline{\epsilon}\cdot f^\star\upalpha$ where $\overline\epsilon$ is the rigidifying insertion, we see that
\begin{equation}\label{theta-removal-identity}
\upalpha \cap q^\star(\vartheta\cdot[S_\sigma]) = \sum f_\star\left( \epsilon \cap p^\star(\vartheta_i [S_{\sigma_{i}}]\right) \ \ \textnormal{in } \ H^\star_{A}(\mathscr S^{\mathsf{node}}_{\sigma}).
\end{equation}
This allows us to conclude the analogue of (iii) on the non-virtual spaces $\mathscr M_\Upsilon^{\sf node}$.

\noindent
{\sc Step IV. Virtual structure.} We can now conclude (ii) and (iii) in the proposition. It suffices to apply virtual pullback to the identity in Equation~\eqref{theta-removal-identity}. To explain the details, denote the preimage of the strata $S_{\sigma}\subset \mathsf{a}T_\gamma$ in our geometric moduli spaces by $\mathsf{GW}_\sigma \subset \mathsf{GW}_\Gamma$ and $\mathsf{PT}_\sigma \subset \mathsf{PT}_\Gamma$.
The morphisms to $S_\sigma$ factor through $\mathscr S^{\mathsf{node}}_{\sigma}$ and we have a Cartesian diagram:
\[
\begin{tikzcd}
\mathsf{GW}_\Upsilon, \mathsf{PT}_\Upsilon \arrow{d}\arrow{r} &\mathsf{GW}_\sigma, \mathsf{PT}_\sigma \arrow{d}\\
\mathscr M^{\mathsf{node}}_{\Upsilon} \arrow{r} &S^{\mathsf{node}}_{\sigma}.
\end{tikzcd}
\]
Since the horizontal morphisms are flat toric fibrations, we can combine virtual push/pull compatibility with \eqref{theta-removal-identity} to complete the proof.  As we are pulling back the same identity to both GW and PT moduli spaces, 
we have compatibility with the GW/PT correspondence is automatic.
\end{proof}

The key consequence for our overall goal is the following.

\begin{corollary}\label{cor: non-exotic-strata}
Given an exotic stratum class $[\![ \sigma,\vartheta,\upalpha ]\!]$, with $1$-complex $\Upsilon$ in the interior of $\sigma$, the GW/PT correspondence holds for the GW/PT series associated to  $[\![ \sigma,\vartheta,\upalpha ]\!]$
if it holds for the nonexotic invariants of the stars $v'$ of $\Upsilon$ and any stars less than these stars in the partial ordering.
\end{corollary}

\section{Full toric boundary}

The goal of this section is to prove the following:

\begin{theorem}\label{thm: full-boundary}
    The logarithmic GW/PT correspondence holds equivariantly, with primary insertions, for toric threefold pairs $(Y|\partial Y)$ where $\partial Y$ is the full toric boundary. 
\end{theorem}

\subsection{Reduction to trivalent stars} We write discrete data as $[\mathsf v,k,\bm\mu]$ -- standing for a star, $k$ internal insertions, and partitions of the ray multiplicities. As noted earlier, this determines the curve class. Recall the virtual dimension of the GW/PT moduli spaces is
\[
\mathsf{vdim}: = k + \sum \ell_i,
\]
where $\ell_i$ ranges over the lengths of the partitions in $\bm\mu$. 

\begin{proposition}\label{prop: trivalent-reduction}
If 
\[
\mathsf{vdim} = k + \sum \ell_i \geq 4
\]
then any primary GW/PT series for $[\mathsf v, k, \mu]$ is effectively determined by the GW/PT series of strictly lower-order stars in $\RR^3$, compatibly with the GW/PT correspondence.  Furthermore, the PT series
$\mathsf{Z}_{\mathsf{PT}}\left(Y|\partial Y;q | \upalpha| \bm{\mu} \right)_{\sf v}$
 is expressed as a polynomial function of $q$, $q^{-1}$, and the PT series for stars $\mathsf v'\prec\mathsf v$.  

\end{proposition}

\begin{proof}
Without loss of generality, we can assume $T_{\mathsf v} \rightarrow \mathsf{P}$ is combinatorially flat.
Consider the diagram
\[
\begin{tikzcd}
\mathscr M\arrow{d}\arrow{r} & \mathsf{Ev} \arrow{d}\\
\mathsf a T_{\sf v}\arrow{r} &\mathsf{a} \mathsf P.
\end{tikzcd}
\]
and fix a class $\upalpha \in H^\star_{A}(\mathsf{Ev})$.

By Proposition~\ref{prop: elementary-cohomology}, we can assume $\upalpha$ descends to a stratum class $[S_\tau] \in H_A^\star(\mathsf a \mathsf P)$ associated to a cone $\tau$ of $\mathsf P$, and we can pull back $[S_\tau]$ to $\mathscr M$ via $\mathsf aT_{\mathsf v}$. 
By flatness, the pullback of $[S_\tau]$ to $\mathsf aT_{\mathsf v}$ is the same as its cycle-theoretic preimage; this is an integer linear combination of strata $S_\sigma$ associated 
to the cones $\sigma$ of $\mathsf T$ which map isomorphically onto $\tau$. By pulling back each $[S_\sigma]$ to $\mathscr M$, it suffices to prove the GW/PT correspondence for strata insertions $[\![\sigma, 1,1]\!]_{\mathsf v}$.

Using the virtual dimension constraint, we see that the dimension of $\sigma$ must be at least $4$ to have a non-vanishing invariant.  
On the other hand, the subcomplex inside $T$ consisting of translations of $\mathsf v$ (with its internal markings) is $3$-dimensional, since such a translation is determined by the location of the central vertex. Therefore,
if we let $\Upsilon$ denote a Chow $1$-complex in the interior of $\sigma$, we see 
that $\Upsilon$ cannot be a translated star and its vertices are all strictly smaller than $\mathsf v$ in our partial ordering. By Corollary~\ref{cor: non-exotic-strata}, we deduce the primary GW/PT correspondence for $[\mathsf v, k, {\bm \mu}]$ from the stars at these vertices and all lower-order stars.
\end{proof}

Given this proposition, the cases that need to be analyzed are (i) \emph{linear stars}, where $\mathsf v$ has two rays in opposite directions, with $\ell_1 =1$ and $k=0$ or $1$, and (ii) \emph{trivalent stars}, where $\mathsf v$ has $k=0$ and three rays with maximal contact order, i.e., $\ell_i=1$ for all $i$. We take them in turn.

\subsection{Linear stars}\label{sec: linear-stars}

The linear cases can be analyzed by hand.
First, consider $k=0$ and $\ell_1 = 1$. We can assume
$Y = \mathbb P^1 \times \mathbb P^1 \times \mathbb P^1$ with $\beta= d[\mathbb P^1 \times 0 \times 0]$ and partitions $(d)$ and $\mu$ along $D_0=\pi_1^{-1}(0)$ and $D_\infty = \pi_1^{-1}(\infty)$ respectively.
Let $D = \mathbb{P}^1 \times \mathbb{P}^1$, which we can identify with both $D_0$ and $D_\infty$.
\begin{lemma}\label{lem: linear-calculation}
If $\ell(\mu) > 1$, then all GW/PT invariants vanish.
If $\mu = (d)$, then 
\begin{eqnarray*}
Z_{\mathsf{GW}}\left(Y|\partial Y;u| \delta_1 \otimes \delta_2|(d);(d)  \right)_\beta = \frac{1}{d} u^{-2} \int_{D}\delta_1\cup\delta_2,
\\
Z_{\mathsf{PT}}\left(Y|\partial Y; q| \delta_1 \otimes \delta_2|(d);(d)  \right)_\beta = \frac{(-1)^{d-1}}{d}q^{d} \int_{D}\delta_1\cup\delta_2.
\end{eqnarray*}
\end{lemma}


\begin{proof}
For both moduli spaces, since the support of the map/sheaf is a fiber class with connected image, the evaluation maps factor through evaluation at $D_0$:
\[\mathsf{GW}_{\mathsf v},\ \mathsf{PT}_{\mathsf v} \rightarrow D \hookrightarrow D \times D^{\ell(\mu)},\]
where the second map is the small diagonal.
The virtual dimension is $\ell(\mu) +1$, so a dimension count gives the vanishing statement when $\ell(\mu) > 1$ and
\[Z_{\mathsf{GW}}\left(Y|\partial Y;u| \delta_1 \otimes \delta_{2}|(d);(d)  \right)_\beta = a(u) \int_{D}\delta_1\cup\delta_2\]
when $\mu = (d)$, where $a(u)$ is a scalar function of $u$.
If we degenerate $Y$ along the first factor into two copies of $Y$, the degeneration formula gives
$a(u) = d\cdot a(u)^2.$
Since we know $a(u)$ is nonzero, by an explicit calculation at $g=0$, the result follows. The same analysis also gives the PT calculation, where the sign comes from the corresponding sign in the PT degeneration formula.
\end{proof}

The case $k =\ell_1 = \ell_2 =1$ can be analyzed in a similar way as follows:

\begin{lemma}
\begin{eqnarray*}
    Z_{\mathsf{GW}}\left(Y|\partial Y;u| \zeta \otimes\delta_1 \otimes \delta_2|(d);(d)  \right)_\beta &=&  u^{-2} \int_{Y}\zeta\cup \delta_1\cup\delta_2,\\
Z_{\mathsf{PT}}\left(Y|\partial Y; q|\zeta \otimes  \delta_1 \otimes \delta_2|(d);(d)  \right)_\beta &=& q^{d} \int_{Y}\zeta \cup \delta_1\cup\delta_2.
\end{eqnarray*}
\end{lemma}

\begin{proof}
Evaluation at the internal marked point gives a factorization
\[\mathsf{GW}_{\mathsf v}, \mathsf{PT}_{\mathsf v} \rightarrow Y \hookrightarrow Y \times D \times D.\]
The virtual class pushes forward to a scalar multiple of $Y$. The scalar can be evaluated at $\zeta = [p \otimes 1 \otimes 1]$, $\delta_1 = [p \otimes p]$ and $\delta_2 = 1$. Here,
the divisor equation can be used to remove the internal insertion and reduce to $k=0$, in which case it follows from Lemma~\ref{lem: linear-calculation}.
\end{proof}

\subsection{Trivalent stars} 
Let $\mathsf v$ be a trivalent star in $\mathbb R^3$, given by vectors $v_1,v_2,v_3$ with multiplicities $n_i$ and $\sum n_i v_i = 0$. By the dimension analysis, the partitions have length one along each ray, i.e. we consider maximal tangency for the star. We consider GW/PT moduli spaces over all genus/Euler characteristics. The invariants only depend on $\mathsf v$ up to an $\mathsf{GL}_3(\ZZ)$-change of coordinates.

\subsubsection{Organizing the evaluation} The span of the $v_i$ is $2$-dimensional in $\mathbb R^3$, so we change basis and assume
\[
v_i = (\star,\star,0).
\]
We will use this identification $\mathbb R^3 = \RR^2\times\RR$ and refer to this last copy of $\mathbb R$ as the \textit{third coordinate}. 

The picture extends to the geometric side. The threefold determined by the star has the form
\[
Y = S\times\mathbb P^1.
\]
The $\mathbb P^1$ is the direction corresponding to the third coordinate in our $\mathbb R^3$, referred to above. The divisor corresponding to $v_i$, call it $D_i$, has the form $\overline D_i\times\mathbb P^1\subset S\times\mathbb P^1$, where $\overline D_i\subset S$ is a boundary curve. Evaluating to $D_i$ and projection gives ${\sf GW}_{\sf v}, \ {\sf PT}_{\sf v}\to (\mathbb P^1)^3$. We make a basic observation.

\begin{lemma}
The maps ${\sf GW}_{\sf v},\ {\sf PT}_{\sf v}\to (\mathbb P^1)^3$ factor through the diagonal $\mathbb P^1\subset(\mathbb P^1)^3$.
\end{lemma}

\begin{proof}
The cycle underlying the stable map/pair is contained in a surface of the form $S\times\{\sf pt\}$.
\end{proof}

Split the evaluation space into
\[
\overline{E} \times (\mathbb P^1)^3.
\]
Reset the notation and consider the evaluation space
\[
{\sf Ev}_{\sf v} = \overline{E} \times \mathbb P^1\subset \overline{E} \times (\mathbb P^1)^3.
\]
Both moduli spaces ${\sf GW}_{\sf v}$ and ${\sf PT}_{\sf v}$ have virtual dimension $3$. There is a conjectural matching for any codimension $3$ logarithmic cohomology class insertion. But by the discussion above, it suffices to prove it for insertions from the subspace
\[
H^\star_{A,{\sf log}}(\overline{E})\otimes H_A^\star(\mathbb P^1)\subset H_{A,{\sf log}}^\star({\sf Ev}_{\sf v}).
\]
If we do not choose a point class in the second factor, there are not sufficiently many constraints and the invariant vanishes (even equivariantly), so the choice for the insertions is determined by an equivariant logarithmic cohomology class on $\overline{E}$.

\subsubsection{Tropical evaluations} We have a space $T_{\sf v}(\mathbb R^3)$ of Chow $1$-complexes. We can fix a point class in the distinguished $\mathbb P^1$-direction -- the second factor in the above tensor product. So we can assume for this discussion that all $1$-complexes are contained in the height $0$ slice of $\mathbb R^2\times \RR$. 

We have an evaluation map to the product of these three boundary curves:
\[
\mathsf {GW}_{\sf v} \ \ \mathsf {PT}_{\sf v} \to \overline {\sf E}.
\]
We also have a map from the space of $1$-complexes to the corresponding space of $0$-complexes:
\[
T_{\sf v}(\mathbb R^3)\to \overline{\mathsf P}.
\]
The space $\overline{\mathsf P}$ is the cone space of $\overline {\mathsf E}$ -- one point each on the fans of the curves $\overline D_1,\overline D_2$ and $\overline D_3$. 

We now explain that all the invariants associated with a trivalent star can be reconstructed from any single nontrivial invariant. The $1$-complexes are contained in a copy of $\mathbb R^2$ and we have three distinguished vectors $v_1$, $v_2$, and $v_3$ given by the star. By basic toric geometry, the cocharacter space of $\overline D_i$ is the quotient of this $\mathbb R^2$ by the primitive vector in the direction of $v_i$, compatibly with the lattice structures. Putting these maps together, we have a map
\[
\mathbb R^2\to \mathbb R^2/\langle v_1\rangle\times \mathbb R^2/\langle v_2\rangle \times \mathbb R^2/\langle v_3\rangle = \overline{\mathsf P}.
\]
Call the image lattice $\mathsf L_{\mathsf v}$. The linear space $\mathsf L_{\mathsf v}$ determines a subtorus $L\otimes \mathbb G_m \times \mathbb G_m \subset {\sf Ev}_{\sf v}$. 

For every blowup of the evaluation space ${\sf Ev}_{\sf v}$, we can take the closure of this subtorus. Since the inclusion is $A$-equivariant, it determines a logarithmic cohomology class:
\[
[R_{\mathsf v}]\in H^\star_{A,{\sf log}}({\sf Ev}_{\sf v}).
\]

\begin{lemma}
    For each fixed value of the genus resp. holomorphic Euler characteristic, the pushforward of the virtual class of ${\sf GW}_{\mathsf{v},g}$ resp. ${\sf PT}_{\mathsf{v},\chi}$ in $H^\star_{\sf log}({\sf Ev}_{\sf v})$ is a multiple of $[R_{\mathsf v}]$. 
\end{lemma}

\begin{proof}
    Let $\mathbb P_{\sf v}$ be the ``logarithmic linear system'' of curves on the surface $S$ in the curve class determined by ${\sf v}$, with maximal tangency with the toric divisors. This coincides with the PT space with minimal possible Euler characteristic~\cite[Section~6]{MR20}. There are $A$-equivariant maps:
\[
\mathsf{GW}_{\mathsf{v},g}, \ \mathsf{PT}_{\mathsf{v},\chi} \to \mathbb P_{\mathsf v}\to \overline E
\]
by projecting the underlying cycle to $S$, and evaluating. We show the image of $\mathbb P_{\sf v}$ in $\overline E$ is contained in a surface with homology class proportional to $[R_{\sf v}]$. The result then follows for dimension reasons.

To prove that the image of $\mathbb P_{\sf v}$ is contained in such a surface, we show that the tropicalization of the map has its image contained in $L\otimes \RR\subset \overline P$. By~\cite{Tev07}, this implies the image is contained in a compact surface in $\overline{L\otimes \mathbb G_m}\times\mathbb G_m$, which implies the claim. 

For the tropical statement, note that the cone complex of $\mathbb P_{\sf v}$ is the set of balanced tropical curves in $\mathbb R^2$ with three ends, given by $\mathsf v$. If we fix the positions of two of the ends, the location of the last one is fixed, independent of the tropical curve — this is a consequence of the {\it Menelaus condition}, see~\cite[Proposition~6.12]{Mik17}. The condition forces the required evaluation condition.
\end{proof}


For each genus/Euler characteristic, the pushforward of the virtual class of the GW/PT space to $\mathsf{Ev}_{\mathsf v}$ is a rational multiple of the class $R_{\mathsf v}$. The proportionality factor is the {\it principal GW/PT} invariant. 

\begin{definition}
 The generating series of principal GW/PT invariants will be called the {\it principal GW/PT series} of the star $\mathsf v$. 
\end{definition}

The principal series determine the GW/PT series for arbitrary, possibly equivariant, insertions with star $\mathsf v$ -- given $\xi$ in $H^\star_{\sf log}(\mathsf {Ev}_{\mathsf v})$, the PT series is given by the principal contribution times the degree of $[R_{\mathsf v}]\cdot \xi$. More usefully, the principal contributions are also determined by the series for any choice of $\tau$, provided $[R_{\mathsf v}]\cdot \xi$ is nonzero. We match the principal contributions by selecting insertions to compute particular invariants {\it non-equivariantly}. 

\subsection{Consistency relations} The principal series associated to $\mathsf v$ appear in the computation of more general invariants by degeneration -- for example $4$-valent invariants. Each $4$-valent invariant can be computed in different ways, so that different trivalent vertices appear in the computation. The resulting consistency relations constrain the trivalent vertices, and this is what we explain next.

The computation in this section closely follows the calculation of B. Parker~\cite{Par17}. The combinatorics is the same; the geometric arguments are different because the structures of our degeneration formulas are different.

\subsubsection{Reduction to multiplicity} A trivalent vertex gives three integer vectors that sum to $0$. By $GL_3(\mathbb Z)$ symmetry, the star can be represented by the rows of:
\[
\begin{pmatrix}
n & 0 & 0\\
0 & mn  & 0\\
-n & -mn & 0
\end{pmatrix},
\]
where $m$ and $n$ are positive integers.  To see this, let $K=\langle v_1,v_2\rangle\subset \ZZ^3$. Since $v_1+v_2+v_3=0$, the vectors span a rank $2$ lattice $K$. By existence of Smith normal form, there is a matrix $A$ in $\mathrm{GL}_3(\mathbb{Z})$ such that
\[
A(K)=\langle (n,0,0),(0,mn,0)\rangle.
\]
In particular,
\[
A(v_1)=(n,0,0),\quad A(v_2)=(0,mn,0),\quad A(v_3)=(-n,-mn,0).
\]
The rows represent the multiplicity times the tangent vector.

\begin{definition}
Given a trivalent star $\mathsf v$ in the form above, its {\it multiplicity} is the positive integer $mn^2$.  
\end{definition}

\begin{proposition}\label{prop: multiplicity-only}
    The principal GW (resp. PT) series of a star $\mathsf v$ depends only its multiplicity
\end{proposition}

As a result, it suffices to prove the correspondence for one trivalent star of each multiplicity. 

\begin{proof} The proof proceeds in several steps.

\noindent
{\sc Step I. A $4$-valent calculation.}
    Consider the star given by the rows $v_i$ of
\[
\begin{pmatrix}
n & 0 & 0\\
0 & mn+1  & 0\\
-n& -mn-1&0
\end{pmatrix}.
\]
Fix maximal tangency in the directions $v_1$ and $v_3$. For $v_2$, fix the tangency partition $(mn,1)$.

The GW/PT virtual dimension is $4$, and the evaluation space has dimension $8$. We use $4$ incidence constraints to obtain an invariant. Assign a point condition to the end $v_1$. Note that this fixes the third coordinate for the underlying cycle of any stable map/stable pair meeting the conditions. In the $v_2$ direction, assign a codimension $1$ condition to each of the two evaluations; in both cases, assign a condition of the form $\mathsf{pt}\times\mathbb P^1$, where $\mathbb P^1$ is the third coordinate in $Y$.

For each of the $v_2$ evaluations, choose one of the two fixed points on the corresponding $\mathbb P^1$. Depending on the choice, the contributing tropical curves differ, with two possibilities. See Figure~\ref{fig: 4-valent-figure}.

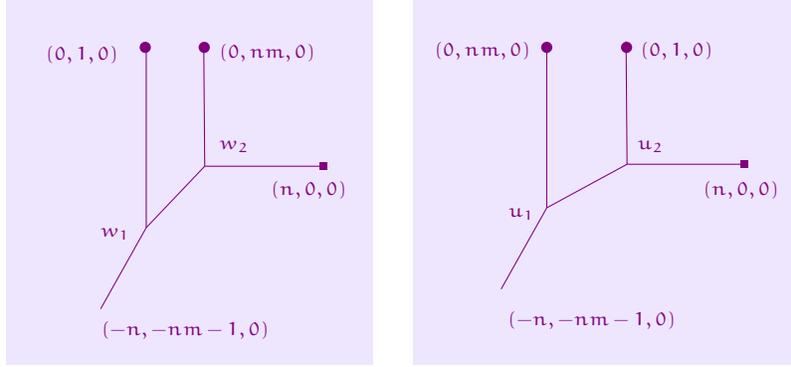
\begin{figure}
\begin{tikzpicture}[x=0.75pt,y=0.75pt,yscale=-1,xscale=1]

\definecolor{lavender}{RGB}{238,230,255}

\fill[lavender]
  (60,655) rectangle (245,840);
\fill[lavender]
  (265,655) rectangle (460,840);

\draw[violet] (159.5,678) -- (160,740);
\draw[violet] (219.5,740) -- (160,740);
\draw[violet] (160,740) -- (130.5,771);
\draw[violet] (130.5,678) -- (130.5,771);
\draw[violet] (130.5,771) -- (107.5,812);

\draw[violet] (372.5,677) -- (373,739);
\draw[violet] (432.5,739) -- (373,739);
\draw[violet] (373,739) -- (332.5,761);
\draw[violet] (332.5,678) -- (332.5,761);
\draw[violet] (332.5,761) -- (309.5,802);

\draw[violet, line width=0.75] (159.75,680)
  node[circle, fill=violet, inner sep=1.5pt] {};
\draw[violet, line width=0.75] (220,740)
  node[rectangle, fill=violet, inner sep=1.5pt] {};  
\draw[violet, line width=0.75] (130,680)
  node[circle, fill=violet, inner sep=1.5pt] {};
\draw[violet, line width=0.75] (432,739)
  node[rectangle, fill=violet, inner sep=1.5pt] {};  
\draw[violet, line width=0.75] (372.75,680)
  node[circle, fill=violet, inner sep=1.5pt] {};
\draw[violet, line width=0.75] (332.5,680)
  node[circle, fill=violet, inner sep=1.5pt] {};

\draw[violet] (410,746) node [anchor=north west][inner sep=0.75pt] {\tiny $( n,0,0)$};
\draw[violet] (192,746) node [anchor=north west][inner sep=0.75pt] {\tiny $( n,0,0)$};
\draw[violet] (166,677) node [anchor=north west][inner sep=0.75pt] {\tiny $( 0,nm,0)$};
\draw[violet] (378.5,676) node [anchor=north west][inner sep=0.75pt] {\tiny $( 0,1,0)$};
\draw[violet] (78.5,678) node [anchor=north west][inner sep=0.75pt] {\tiny $( 0,1,0)$};
\draw[violet] (106.5,817) node [anchor=north west][inner sep=0.75pt] {\tiny $(-n,-nm-1,0)$};
\draw[violet] (311.5,812) node [anchor=north west][inner sep=0.75pt] {\tiny $(-n,-nm-1,0)$};
\draw[violet] (300,676) node [anchor=north ][inner sep=0.75pt] {\tiny $( 0,nm,0)$};
\draw[violet] (320,760) node [anchor=north ][inner sep=0.75pt] {\tiny $u_1$};
\draw[violet] (385,726) node [anchor=north ][inner sep=0.75pt] {\tiny $u_2$};
\draw[violet] (115,770) node [anchor=north ][inner sep=0.75pt] {\tiny $w_1$};
\draw[violet] (175,726) node [anchor=north ][inner sep=0.75pt] {\tiny $w_2$};

\end{tikzpicture}

\caption{The two Chow $1$-complexes, depending on the choice of constraint for the two vertical ends in the picture. The square indicates the end that carries the point condition, while the circles indicate divisorial conditions.}\label{fig: 4-valent-figure}
\end{figure}
For each Chow $1$-complex, apply the degeneration formula. Depending on the cone structure, these $1$-complexes may carry $2$-valent vertices. Adding a $2$-valent vertex introduces a new bounded edge, say, of multiplicity $d$. By the earlier computations in Section~\ref{sec: linear-stars}, the new vertex contributes $\frac{1}{d}$ when the splitting partition along the bounded edge is maximal and $0$ otherwise. This $\frac{1}{d}$ cancels with the gluing factor coming from the bounded edge, so we can ignore the $2$-valent vertices in the coming discussion.\footnote{Alternatively, if one is careful with cone structures, the $2$-valent vertices can actually be avoided.}

\noindent
{\sc Step II. Reduction to maximal gluing tangency.} We now apply the degeneration formula to split along the bounded edge in the figure above. This involves summing over splitting partitions along the bounded edge. To justify the claim, write the gluing divisor $D$, up to subdivision, as $\overline D\times\PP^1$, corresponding to the decomposition of the threefold as $S\times\PP^1$. Suppose there are two parts in the partition along the gluing divisor; the case of more parts is similar. Let $\mathsf M_1$ and $\mathsf M_2$ be the moduli spaces at the two vertices \textit{with the boundary incidence conditions imposed}. For each vertex, and then for each part, we have a map to the gluing divisor $\overline D\times\PP^1$. Together we have maps
\[
\mathsf M_1\to \overline D^2\times (\PP^1)^2\leftarrow \mathsf M_2.
\]
By the degeneration formula, the invariant we want is obtained by subdividing such that both maps become combinatorially flat, pushing forward, and intersecting the cycles. The evaluation space is $4$-dimensional. The pushforward of $[\mathsf M_1]^{\sf vir}$ is $1$-dimensional, and that of $[\mathsf M_2]^{\sf vir}$ is $3$-dimensional. 

The flattening subdivision can be described explicitly. The tropical space associated to the gluing divisor is a copy of $\mathbb R^2$. The support of the fan of $\overline D\times\PP^1$ can be written as $\RR\times\RR$. Let us call the second coordinate the ``height''. For each $\mathsf M_i$, the tropical evaluation maps to two copies of this $\RR\times\RR$ -- one for each part in the gluing partition. However, because the $1$-complex is contained in an $xy$-plane, the height coordinate for the two parts is equal. Paying attention only to the height coordinate, the tropical evaluation associated to $\mathsf M_i$ maps onto the diagonal.

It follows from the above discussion that we can flatten the evaluations by subdividing the fan of $\overline D^2\times (\PP^1)^2$, supported on $\RR^2\times\RR^2$, by separately subdividing the two factors in the product. In fact, in the second copy of $\mathbb R^2$, it is sufficient to introduce the diagonal as a subfan. Translating to geometry, this means we can flatten both evaluation maps by a single blowup of the form
\[
(\overline D^{2})^\dagger\times (\PP^1\times\PP^1)^\dagger\to \overline D^2\times (\PP^1)^2,
\]
where $(\PP^1\times\PP^1)^\dagger$ is obtained by blowing up $(0,0)$ and $(\infty,\infty)$.


Using the K\"unneth decomposition above, we claim that the pushforward of $[\mathsf M_1]^{\sf vir}$ to this space has the form $[\mathsf{curve} \otimes \mathsf{pt}]$. Similarly, the pushforward for $[\mathsf M_2]^{\sf vir}$ has the form $[\mathsf{surface} \otimes \mathsf{curve}]$. Note that the dimensions of the two sides are different because only one of the two sides carries a point insertion. In both cases, the claims follow from elementary dimension counting. Apply the K\"unneth decomposition for the cap product on $(\overline D^{2})^\dagger\times (\PP^1\times\PP^1)^\dagger$, so the invariant vanishes.  

\noindent
{\sc Step III. Maximal gluing tangency.} If the partition has a single part, the degeneration formula takes its naive form, in the sense that no blowups are needed to flatten the evaluation map. 

Let $\mathsf u_1$ and $\mathsf u_2$ be the stars in the first degeneration, and $\mathsf w_1$ and $\mathsf w_2$ in the second, as shown in the figure. The degeneration formula gives an identity of GW/PT series:
\[
\mathsf Z(\mathsf u_1)\cdot \mathsf Z(\mathsf u_2) = \mathsf Z(\mathsf w_1)\cdot \mathsf Z(\mathsf w_2).
\]
The stars $\mathsf u_2$ and $\mathsf w_1$ are identified by the $\mathsf{GL}_3(\ZZ)$-action. Finally, both of these series are nonzero. We can see this by taking the genus $g$ to be $0$ or $\chi$ to be the minimal possible Euler characteristic. In both cases, one obtains genuinely enumerative invariants—counts of rational curves on surfaces—that are known to be nonzero. The proposition follows by canceling these terms.
\end{proof}

\subsubsection{Multiplicity $1$} We start with the multiplicity $1$ case. 

\begin{proposition}\label{prop: multiplicity-one}
    Let $\mathsf v$ be a trivalent vertex of multiplicity $1$. The GW/PT correspondence holds for $\mathsf v$. 
\end{proposition}

\begin{proof}
    Take $\mathsf v$ to be the rows of:
    \[
\begin{pmatrix}
1 & 0 & 0\\
0 & 1  & 0\\
-1 & -1 & 0
\end{pmatrix}.
\]
The star represents the class of a generic line in $\mathbb P^2$, at fixed height in $\PP^2\times\PP^1$. Two point conditions on the $\mathbb P^2$ gives a rigid curve -- a line. Degenerate the threefold to the normal cone of this line. The GW/PT spaces degenerate to a local curve theory, where the correspondence is known~\cite{BP08,OP10}.\footnote{In fact, we are only using a very small piece of this -- the PT side is classical calculation, see~\cite[Section~4.2]{PT09} and the GW side is a known Hodge integral~\cite{FP00}.}
\end{proof}

\begin{remark}
    The GW principal series is known for all multiplicities~\cite{Bou17,KHSUK}. In multiplicity $n$ it is
\[
\mathsf{GW}^{\sf prin}_n(u) = \frac{2}{u^3}\cdot \sin\left(\frac{nu}{2}\right).
\]
The correspondence predicts $\mathsf{PT}^{\sf prin}_n(u) = q\cdot (1+q)^n.$
\end{remark}

\subsubsection{Higher multiplicity} We relate vertices of higher multiplicity to those of multiplicity $1$. 

\begin{proposition}\label{prop: trivalent-recursion}
    The GW/PT correspondence holds for trivalent vertices of any multiplicity. 
\end{proposition}

\begin{proof}
    Here we closely follow Parker~\cite[Section~3]{Par17} for the tropical combinatorics, and apply our degeneration formula to analyze the geometry. Start with a $4$-valent star given by the rows of the matrix with rows $v_i$:
     \[
\begin{pmatrix}
1 & 0 & 0\\
0 & 1  & 0\\
1& 0 &n\\
-2&-1&-n.
\end{pmatrix}.
\]
See Figure~\ref{fig: 4-valent-setup}. We take maximal tangency for all divisors. The moduli spaces have virtual dimension $4$, so we obtain an invariant with two point conditions. Place them on $v_1$ and $v_2$. For each $v_i$, we have a choice, in the third coordinate $\mathbb P^1$, of which fixed point they are constrained to. Tropically, this corresponds to which of $v_1$ and $v_2$ is ``higher''. Equating the computations gives relations between principal series.

\begin{figure}[h!]
\tdplotsetmaincoords{20}{0}

\begin{tikzpicture}[tdplot_main_coords, scale=1.25]

\definecolor{lavender}{RGB}{238,230,255}

\fill[lavender] (-2.5,-1.2,-2.5) rectangle (2.5,1.8,2.5);

\coordinate (O) at (0,0,0);

\coordinate (A) at (2,0,0);
\coordinate (B) at (0,2,0);
\coordinate (C) at (1,0,2);
\coordinate (D) at (-1.6,-0.8,-1.6);

\draw[violet, thick,->] (O) -- (A);
\draw[violet, thick,->] (O) -- (B);
\draw[violet, thick,->] (O) -- (C);
\draw[violet, thick,->] (O) -- (D);

\node[violet] at ($(A)+(0.1,-0.25,0)$) {\tiny $(1,0,0)$};
\node[violet] at ($(B)+(0.1,0.25,0)$) {\tiny $(0,1,0)$};
\node[violet] at ($(C)+(0.25,0.25,0.15)$) {\tiny $(1,0,n)$};
\node[violet] at ($(D)+(-0.2,-0.2,-0.3)$) {\tiny $(-2,-1,-n)$};

\shade[ball color=violet] (O) circle (1.2pt);

\end{tikzpicture}
\caption{The $4$-valent star considered above.}\label{fig: 4-valent-setup}
\end{figure}
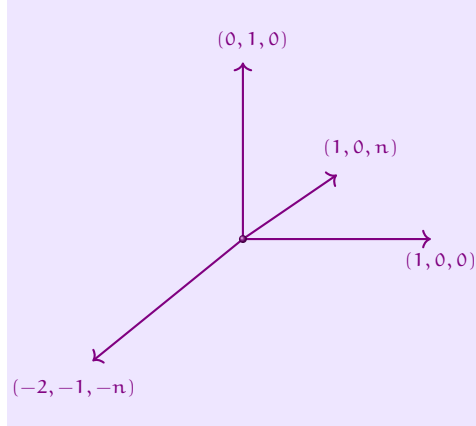

\noindent
{\sc Case I.} The only tropical curve is shown in Figure~\ref{fig: case-I-trivalent}.
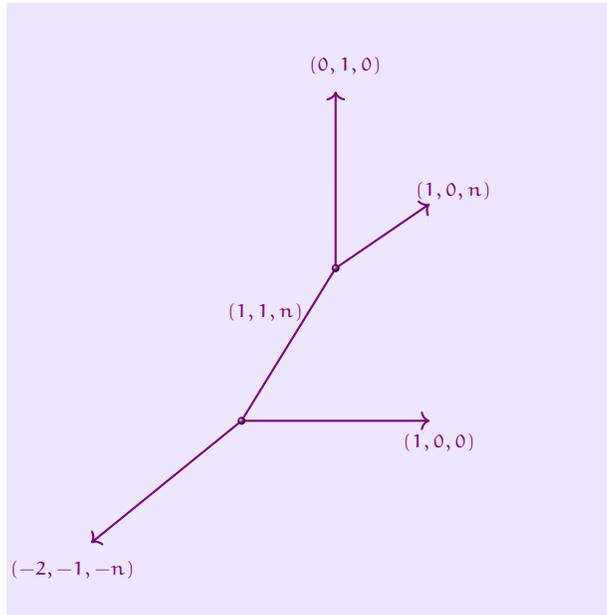
\begin{figure}
\tdplotsetmaincoords{20}{0}

\begin{tikzpicture}[tdplot_main_coords, scale=1.25]

\definecolor{lavender}{RGB}{238,230,255}

\fill[lavender] (-2.5,-1.5,-2) rectangle (4,4,2);

\coordinate (O) at (0,0,0);

\coordinate (A1) at (2,0,0);       
\coordinate (D1) at (-1.6,-0.8,-1.6); 

\draw[violet, thick,->] (O) -- (A1);
\draw[violet, thick,->] (O) -- (D1);

\coordinate (shift) at (1,1,2);
\coordinate (B2) at ($(shift)+(0,2,0)$);  
\coordinate (C2) at ($(shift)+(1,0,2)$); 

\draw[violet, thick,->] (shift) -- (B2);
\draw[violet, thick,->] (shift) -- (C2);

\draw[violet, thick] (O) -- (shift);

\node[violet] at ($(A1)+(0.1,-0.25,0)$) {\tiny $(1,0,0)$};
\node[violet] at ($(D1)+(-0.2,-0.2,-0.3)$) {\tiny $(-2,-1,-n)$};
\node[violet] at ($(B2)+(0.1,0.25,0.1)$) {\tiny $(0,1,0)$};
\node[violet] at ($(C2)+(0.25,0,0.4)$) {\tiny $(1,0,n)$};
\node[violet] at ($(shift)+(-0.75,-0.5,0.0)$) {\tiny $(1,1,n)$}; 

\shade[ball color=violet] (O) circle (1.2pt);
\shade[ball color=violet] (shift) circle (1.2pt);

\end{tikzpicture}
\caption{The unique tropical curve appearing in  Case I for the distribution of points above.}\label{fig: case-I-trivalent}
\end{figure}
The multiplicity along the bounded edge is equal to $1$, so the degeneration formula takes a naive form, and both vertices have multiplicity $1$ so GW/PT series are given, up to a positive multiple, by the square of the principal multiplicity $1$ series. 

\noindent
{\sc Case II.} In the second case, there are several tropical curves: one curve with an internal edge with direction $(2,0,n)$ and another set, when $n\geq 3$: for each $n_1> n_2>0$ with $n_1+n_2 = n$, pictured in Figure~\ref{fig: case-II-trivalent} -- in addition to the ends, there is a triangle with edge directions $(1,0,n_1)$, $(1,0,n_2)$, and $(-2,0,-n_1)$. Let us call these type I and type II$(n_1,n_2)$, respectively. 

The contribution for type I can be computed using the degeneration formula. When $n$ is odd, there is a unique splitting partition, and the contribution is a positive multiple of a product of the principal series for multiplicity $n$ times the principal series for multiplicity $1$.

\begin{figure}[h!]
\tdplotsetmaincoords{20}{0}

\begin{tikzpicture}[tdplot_main_coords, scale=1.3]

\definecolor{lavender}{RGB}{238,230,255}

\begin{scope}[shift={(0,0,0)}]

\fill[lavender] (-2.5,-1.2,-2.5) rectangle (4,2.5,3.5);

\coordinate (O) at (0,0,0);

\coordinate (B1) at (0,2,0);       
\coordinate (D1) at (-1.6,-0.8,-1.6); 
\draw[violet, thick,->] (O) -- (B1);
\draw[violet, thick,->] (O) -- (D1);

\coordinate (P) at (1.5,0,1.5);

\coordinate (A2) at ($(P)+(2,0,0)$);   
\coordinate (C2) at ($(P)+(1,0,2)$);   
\draw[violet, thick,->] (P) -- (A2);
\draw[violet, thick,->] (P) -- (C2);

\draw[violet, thick] (O) -- (P);

\node[violet] at ($(B1)+(0.1,0.25,0)$) {\tiny $(0,1,0)$};
\node[violet] at ($(D1)+(-0.2,-0.2,-0.3)$) {\tiny $(-2,-1,-n)$};
\node[violet] at ($(P)+(-0.5,0.05,0.1)$) {\tiny $(2,0,n)$};
\node[violet] at ($(A2)+(0.1,-0.25,0)$) {\tiny $(1,0,0)$};
\node[violet] at ($(C2)+(0.25,0.15,0.15)$) {\tiny $(1,0,n)$};

\shade[ball color=violet] (O) circle (1.2pt);

\end{scope}

\begin{scope}[shift={(6.75,0,0)}]

\fill[lavender] (-2.5,-1.2,-2.5) rectangle (3.5,2.5,3.5);

\coordinate (O) at (0,0,0);

\coordinate (B1) at (0,2,0);          
\coordinate (D1) at (-1.6,-0.8,-1.6); 
\draw[violet, thick,->] (O) -- (B1);
\draw[violet, thick,->] (O) -- (D1);

\def\a{0.3}
\def\b{1.7}
\coordinate (T1) at (1,0,\a);
\coordinate (T2) at (1,0,\b);

\draw[violet, thick] (O) -- (T1) -- (T2) -- (O);

\coordinate (A2) at ($(T1)+(2,0,0)$); 
\coordinate (C2) at ($(T2)+(1,0,2)$); 
\draw[violet, thick,->] (T1) -- (A2);
\draw[violet, thick,->] (T2) -- (C2);

\node[violet] at ($(B1)+(0.1,0.25,0.1)$) {\tiny $(0,1,0)$};
\node[violet] at ($(D1)+(-0.2,-0.2,-0.3)$) {\tiny $(-2,-1,-n)$};
\node[violet] at ($(T1)+(-0.2,-0.25,0.05)$) {\tiny $(1,0,n_1)$};
\node[violet] at ($(T2)+(-0.5,0,0.05)$) {\tiny $(1,0,n_2)$};
\node[violet] at ($(A2)+(0.1,-0.25,0)$) {\tiny $(1,0,0)$};
\node[violet] at ($(C2)+(0.25,0.15,0.15)$) {\tiny $(1,0,n)$};

\shade[ball color=violet] (O) circle (1.2pt);

\end{scope}

\end{tikzpicture}
\caption{The Chow $1$-complexes appearing in Case II above. The picture on the right exists for each positive pair $n_1>n_2$ above, adding up to $n$.}\label{fig: case-II-trivalent}
\end{figure}
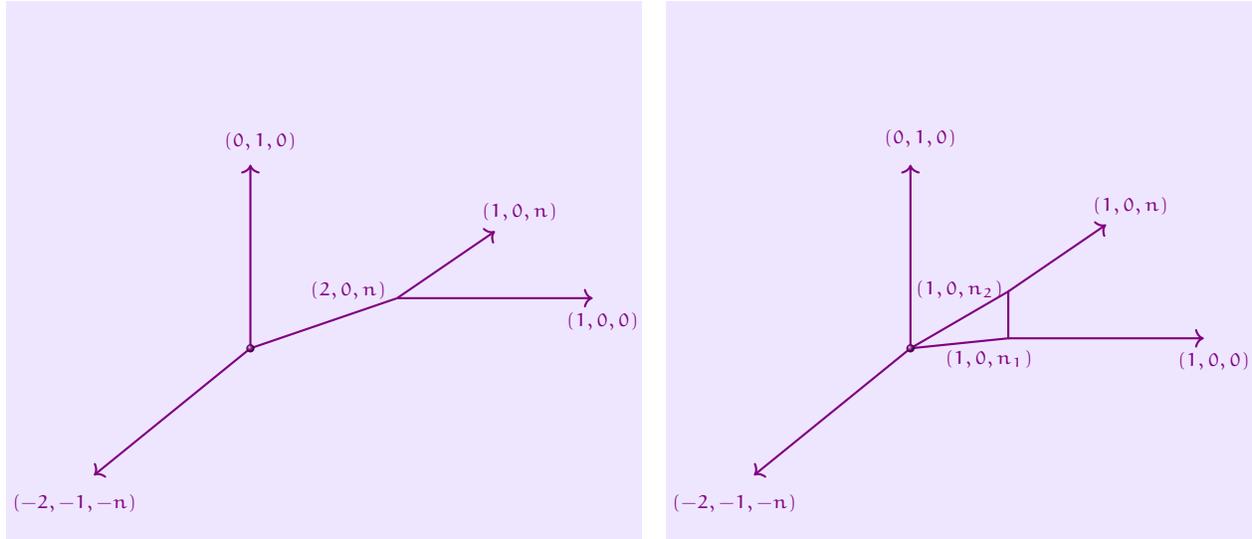

When $n$ is even, it is twice a primitive vector, so there are two possible partitions. When the partition is $(1+1)$, both vertices contribute via a $4$-valent star. By the argument of Section~\ref{sec: linear-stars}, we can reduce further to trivalent stars, and these are of smaller multiplicity than $n$. When the partition is $(2)$, the contribution is a positive multiple of the multiplicity $2$ series with the multiplicity $n$ series as above. Note that $n$ can be $2$, so in this case we get a square of the multiplicity $2$ series.

For type II$(n_1,n_2)$, the degeneration formula expresses the contribution in terms of two principal series and an invariant associated to a $4$-valent vertex. We compute the $4$-valent vertex by specializing the insertions as before. The contribution is a polynomial in the principal series of multiplicities smaller than $n$.

We conclude the result. Start with $n = 2$; the contribution of Case II, type I, gives us the square of the multiplicity $2$ series. The remaining terms involve principal series of multiplicity $1$, so we know the squares of the multiplicity $2$ series on the GW/PT sides match. The multiplicity $2$ series themselves match up to sign. To fix signs, we check the leading term of the PT series, i.e., when the Euler characteristic is minimal. By an elementary calculation, this term is positive. The GW series is known, so we easily see the correspondence in multiplicity $2$ holds. The relations above determine all higher multiplicities in terms of multiplicity $1$ and $2$, compatibly with GW/PT.
\end{proof}

\noindent
{\it Proof of Theorem~\ref{thm: full-boundary}}. Apply Propositions~\ref{prop: trivalent-reduction} and~\ref{prop: multiplicity-only} to reduce to checking a trivalent vertex of each multiplicity, the recursion in Proposition~\ref{prop: trivalent-recursion} to reduce to multiplicity $1$, and the calculation in Proposition~\ref{prop: multiplicity-one} to check multiplicity $1$.
\qed


\section{Nef geometries}\label{sec: nef-geometries}

In this section, we restrict attention to elementary geometries with the following property.

\begin{definition}
An elementary toric pair $(Y|\partial Y)$ is \textit{nef} if every toric divisor not in $\partial Y$ is nef on $Y$.  
\end{definition}

This property is preserved by log blowups. In the one non-boundary case, it means we consider birational models of $\mathbb P(L\oplus \mathcal O)\to S$ where $L$ is nef on $S$. In the two non-boundary case,
we consider models of $\mathbb P(L_1\oplus L_2\oplus \mathcal O)$ where $L_1$ and $L_2$ have non-negative degree on $\mathbb{P}^1$.  In the three non-boundary case, the condition is automatic. The key property of these spaces is the following.

\begin{lemma}
A component of an expansion of a nef elementary toric pair is a nef elementary toric pair.
\end{lemma}

\begin{proof}
    Up to subdivision, a component in an expansion is a toric bundle over a stratum in $(Y|\partial Y)$, relative to the full boundary in the fiber direction. The non-boundary comes from the divisors in the base of this bundle. The result follows by direct analysis of the cases. 
\end{proof}

We can therefore restrict the induction to nef elementary geometries; the general case is done later.

\subsection{$1$-non-boundary}

Let $(Y| \partial Y)$ be a nef $1$-non-boundary elementary geometry, with non-boundary divisor $D$. For a GW/PT problem with curve class $\beta$, $k$ markings, and partition lengths $\ell_i$ we have
\[
\mathsf{vdim} = D\cdot \beta + k + \sum \ell_i
\]

\begin{proposition}
Suppose $(Y| \partial Y)$ is a straight $1$-non-boundary geometry and fix discrete data $[\mathsf v,k,\bm\mu]$ with
\[
\mathsf{vdim} \geq 3
\]
or $(Y| \partial Y)$ is a ruled or general $1$-non-boundary geometry.
Any primary GW/PT series for $[\mathsf v, k, \bm{\mu}]$ is effectively determined by the GW/PT series of strictly smaller stars, compatibly with the correspondence.  Furthermore, the PT series
$\mathsf{Z}_{\mathsf{PT}}\left(Y|\partial Y;q | \upalpha| \bm{\mu} \right)_{\sf v}$
 is a polynomial function of $q$, $q^{-1}$, and the PT series for stars $\mathsf v' \prec \mathsf v$.  
\end{proposition}

\begin{proof}
The proof is parallel to Proposition~\ref{prop: trivalent-reduction}. By Proposition~\ref{prop: elementary-cohomology} any insertion $\upalpha$ can be replaced with a strata class $[\![\sigma, 1, 1]\!]$. By dimension counting, the dimension of $\sigma$ is at least the virtual dimension. In the straight geometry, this is at least $3$ by assumption. The tropical space of translated stars is $2$-dimensional, so any $1$-complex in the interior of $\sigma$ breaks into inductively smaller vertices.

In the ruled/general case, the set of translated stars is at most $1$-dimensional. If the virtual dimension is at least $2$, the $1$-complex breaks non-trivially, but this is automatic -- either there are multiple rays or there is one ray with $\ell = 1$. In this case, $\beta$ is a fiber class of $Y$ and $D\cdot \beta \geq 1$, so $\mathsf{vdim} \geq 2$.
\end{proof}

For the straight geometry, this proposition handles most cases. The remaining base cases have virtual dimension less than $3$. The first is when there are two rays with $\ell_1 = \ell_2 = 1$ and $D\cdot \beta = 0$. In this case, we have a linear star with maximal contact order at each ray, and this calculation
is identical to Lemma~\ref{lem: linear-calculation}. The second is when there is one ray with $\ell = 1$, $k=0$, and $D \cdot \beta = 1$. In this case, the curve class must be a fiber class, with a point insertion at the boundary. This calculation
is equal to the degree $1$ invariant of $(\mathbb{P}^1, \infty) \times \mathbb{C}^2$, which is known by~\cite{BP08,OP10}.

\subsection{$2$- and $3$-non-boundary}\label{sec: 2-nb-3-nb}

Fix discrete data $[\sf v, k, \bm{\mu}]$ in a $2$- or $3$-non-boundary nef toric pair. 

\begin{proposition}\label{2/3NBcase}
Any primary GW/PT series for $[\mathsf v, k, \bm\mu]$ is effectively determined by the GW/PT series of strictly lower-order stars, compatibly with the GW/PT correspondence.  Furthermore, any primary PT series for $[v, k, \mu]$
is a polynomial function of $q$, $q^{-1}$, and the PT series for lower-order stars.
\end{proposition}

\begin{proof}
We argue similarly to the $1$-non-boundary case, but now with no base cases. In the $2$-non-boundary geometries with non-boundary divisors $D_1$ and $D_2$, there is at most a one-dimensional tropical space of translated stars, so by dimension counting, if
\[
\mathsf{vdim}= D_1\cdot \beta + D_2 \cdot \beta + k + \sum \ell_i \geq 2
\]
then the strata invariant breaks into lower-order stars.
We claim this always happens. Suppose $\mathsf{vdim} \leq 1$. The star must have at least one ray and since both $D_1$ and $D_2$ are nef, we must have 
\[
D_1\cdot \beta = D_2 \cdot \beta = k =0
\]
and $\sf v$ has a single ray with $\ell =1$.  
Such a ray must lie on the $xy$ plane in the cone complex $\mathbb R_{\geq 0}^2\times\RR$.  However, the corresponding curve class $\beta$ must intersect either $D_1$ or $D_2$ positively.

In the $3$-non-boundary case, there is at least one ray, so the virtual dimension is positive. By dimension counting, we can replace the insertion with a stratum where the curve breaks.
\end{proof}

\subsection{Polynomiality}\label{sec: polynomiality} We prove Theorem~\ref{thm: polynomiality} and
~\cite[Conjecture~2]{MOOP}.

\begin{corollary}
Suppose $(Y|\partial Y)$ is a toric threefold pair such that every toric divisor not in $\partial Y$ is nef. The equivariant PT series of $Y$ is a Laurent polynomial in $q$. 
\end{corollary}

\begin{proof}
When we degenerate $(Y|\partial Y)$ into elementary geometries, only nef elementary geometries appear, so the discussion of the previous sections implies the polynomiality.
\end{proof}

The {\it capped vertex} was introduced in~\cite[Section~2]{MOOP} to study GW/DT for toric threefolds. Let $Y^\circ$ be $(\mathbb P^1)^3$ minus the toric curves through $(\infty,\infty,\infty)$. Let $\partial Y^\circ$ be the union of the infinity divisors. Invariants are defined by localization residues with respect to the $\mathbb G_m^3$-action. The capped vertex is the equivariant PT series for $(Y^\circ|\partial Y^\circ)$, with contact orders $\mu,\nu,\lambda$. See Figure~\ref{fig: capped-vertex}.

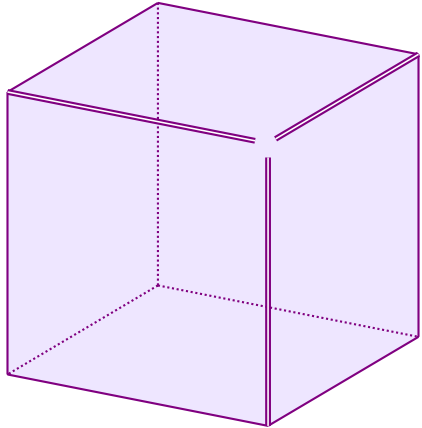
\begin{figure}[h!]
\begin{tikzpicture}[scale=4]
  \definecolor{lavender}{RGB}{238,230,255}

  \tdplotsetmaincoords{70}{120}

  \begin{scope}[tdplot_main_coords]

    \def\eps{0.05}

    \fill[lavender] (0,0,0) -- (1,0,0) -- (1,1,0) -- (0,1,0) -- cycle; 
    \fill[lavender] (0,0,0) -- (1,0,0) -- (1,0,1) -- (0,0,1) -- cycle; 
    \fill[lavender] (0,0,0) -- (0,1,0) -- (0,1,1) -- (0,0,1) -- cycle; 
    \draw[violet, densely dotted, thick]
      (0,0,0)--(1,0,0)
      (0,0,0)--(0,1,0)
      (0,0,0)--(0,0,1);

    \draw[violet, thick]
      (1,0,0)--(1,1,0)
      (0,1,0)--(1,1,0)
      (0,0,1)--(1,0,1)
      (0,1,0)--(0,1,1)
      (1,0,0)--(1,0,1)
      (0,0,1)--(0,1,1);

    \draw[violet, double, thick]
      (1,1,0) -- (1,1,{1-\eps});

    \draw[violet, double, thick]
      (1,0,1) -- (1,{1-\eps},1);

    \draw[violet, double, thick]
      (0,1,1) -- ({1-\eps},1,1);

  \end{scope}
\end{tikzpicture}
\caption{The open subset of $(\mathbb P^1)^3$ defining the capped vertex, depicted here by deleting the relevant faces from the moment polytope of $(\mathbb P^1)^3$.}\label{fig: capped-vertex}
\end{figure}

Because it is a residue theory, the capped vertex does not directly fall under the framework above, since our targets are projective.  Instead, we compare the residue GW/PT theory of $(Y^{\circ}|\partial Y^{\circ})$ with the log GW/PT theory of $(Y, \partial Y)$ and show the discrepancy is given by lower-order strata invariants.

Let $\mathsf{M}(Y|\partial Y)_{\mathsf v}$ denote one of the GW or PT spaces associated to $Y$, and let 
\[j\colon \mathsf{M}(Y^{\circ}|\partial Y^{\circ})_{\mathsf v} \subset \mathsf{M}(Y|\partial Y)_{\mathsf v}\]
denote the open subset corresponding to $Y^{\circ}$.  In what follows, we tensor with the fraction field
$$K = \mathsf {Frac} \ H_{A}^{\star}(\mathsf{pt})$$
\begin{proposition}\label{Chow-residue}
We have the following equality
\[
j_\star[\mathsf{M}(Y^{\circ}|\partial Y^{\circ})_{\mathsf v}]^{\mathsf{vir}} = \left(1 + \sum [\![\sigma, \theta, 1]\!]\right)\cap [\mathsf{M}(Y|\partial Y)_{\mathsf v}]^{\mathsf{vir}} \in H^{\mathrm{BM}}_{A,\star}(\mathsf{M}(Y|\partial Y)_{\mathsf v})\otimes K\]
where
the right-hand side is a summation of strata classes with positive-dimensional cones $\sigma$.
\end{proposition}
Here, because the $A$-fixed locus is proper, the left-hand side is well defined after tensoring with $K$.
This proposition immediately implies the corollary since every term in the right-hand side is a Laurent polynomial in $q$ by our earlier sections.

\begin{proof}

To prove the above identity, we will use the log Chow variety for the irreducible boundary components of $\partial Y$, whose basic properties are studied in Appendix~\ref{app: log-linear-system}.  
For each boundary component, $D_i = (\mathbb{P}^1| \infty)^2$, $i= 1, 2, 3$, we have a
smooth logarithmic Chow variety for $D_i$, equipped with a map to its Artin fan $\mathsf{Chow}_i \rightarrow a\mathsf{W}_{i}.$

By projecting a stable Chow $1$-complex along each coordinate direction, 
we have tropical maps $\mathsf T_{\mathsf v} \rightarrow \mathsf{W}_i$.  Above these, we have corresponding maps of geometric moduli spaces, obtained by taking the one-dimensional support of a stable map/pair 
on an expansion of $X$ and projecting to the corresponding expansion of $D_i$.  After taking products and flattening the tropical map, we have a commutative diagram

\[
\begin{tikzcd}
\mathsf{M}(Y| \partial Y)_{\mathsf{v}} \arrow[swap]{d}{\pi} \arrow{r}{f} &\mathsf{Chow}: = (\prod \mathsf{Chow}_{i})^{\dagger}\arrow{d}{\rho}\\
\mathsf{a T}_{\mathsf{v}} \arrow{r} & \mathsf{aW}:= (\prod \mathsf{aW}_{i})^{\dagger}.
\end{tikzcd}
\]
We claim that
\begin{equation}\label{preimage-of-Chow}
f^{-1}(\mathsf{Chow}^{\circ}) = \mathsf{M}(Y^{\circ}|\partial Y^{\circ})_{\mathsf v}.
\end{equation}
To see this, we first show the tropical statement:  if a stable Chow $1$-complex $[\Gamma]\in \mathsf T_{\mathsf v}$ projects to the $0$-dimensional cone in each $ \mathsf{W}_i$ then it must be the $0$-dimensional cone in $\mathsf{T}_{\mathsf{v}}$.
Indeed, by stability, any vertex of $\Gamma$ is non-linear, so one of its coordinate projections will be non-linear as well.
Once the tropical statement is known, then we have $f^{-1}(\mathsf{Chow}^{\circ}) = \pi^{-1}(\mathsf a \mathsf{T}_{\mathsf v}^{\circ})$.  The latter equals $\mathsf{M}(Y^{\circ}|\partial Y^{\circ})_{\mathsf v}$ because both correspond to only allowing expansions where any 
curve in an expanded component is contracted in the projection to $Y$.

Now there is a unique fixed point $p \in (\mathsf{Chow}^{\circ})^{A}$ contained in the interior of $\mathsf{Chow}$, corresponding to the non-boundary toric divisors in each $D_i$.
By $A$-equivariant localization on $\mathsf{Chow}$, we have a decomposition
\[
1 = c(a)[p] + \mathsf{Res}_{\partial} \in H^\star_{A}(\mathsf{Chow}) \otimes K
\]
where $\mathsf{Res}_{\partial}$ is the sum of equivariant residues for the fixed components of $(\partial \mathsf{Chow})^{A}$.
If we pull this back, and pair with the virtual class, we have
\[
c(a)f^{*}([p])\cap [\mathsf{M}(Y| \partial Y)_{\mathsf v}]^{\mathsf{vir}} = \left(1 -f^\star(\mathsf{Res}_{\partial})\right)\cap [\mathsf{M}(Y|\partial Y)_{\mathsf v}]^{\mathsf{vir}} 
\]

We first show the left-hand side equals $j_\star[\mathsf{M}(Y^{\circ}|\partial Y^{\circ})_{\mathsf v}]^{\mathsf{vir}}$.
Indeed, by equation \eqref{preimage-of-Chow}, we have a sequence of Cartesian squares

\[
\begin{tikzcd}
f^{-1}(p)\arrow[d,"f"]\arrow[r,"i"] &\mathsf{M}(Y^{\circ}|\partial Y^{\circ})_{\mathsf v} \arrow{d}{f}\arrow{r} & \mathsf{M}(Y, \partial Y)_{\mathsf{v}} \arrow{d}{f} \\
p\arrow{r}{i} & \mathsf{Chow}^{\circ} \arrow{r} & \mathsf{Chow},
\end{tikzcd}
\]
where the vertical maps are virtually smooth.
By the concentration theorem of \cite[Theorem~4.1]{stacky-concentration-adeel}, both horizontal maps $i$ induce isomorphisms $i_\star$ on localized Borel--Moore homology, which commute with virtual pullback along $f$.
As a result, we have
\[
c(a)f^{*}([p])\cap [\mathsf{M}(Y, \partial Y)_{\mathsf v}]^{\mathsf{vir}} = (i_{*})^{-1}\circ f^{!}\left([\mathsf{Chow}^{\circ}]) = (i_{*})^{-1}([\mathsf{M}(Y^{\circ}|\partial Y^{\circ})_{\mathsf v}]^{\mathsf{vir}}\right)
\]
which, by the residue definition of $j_\star$, finishes the computation.

It remains to show the contribution of $f^{*}(\mathsf{Res}_{\partial})$ is given by strata classes.
Each fixed locus $F$ of $\mathsf{Chow}$ is contained in some boundary divisor $E$, so we can group the boundary residues as a sum over divisors
\[
\mathsf{Res}_{\partial} = \sum_{E} i_{E,*}(\gamma_E)
\]
where $\gamma_E \in H^\star_{A}(E)\otimes K$ and $i_{E}\colon E \hookrightarrow \mathsf{Chow}$ is the inclusion.

By Corollary~\ref{cor: Chow-chow-groups} of the appendix, each $\gamma_E$ can be generated by boundary divisors in $E$ so descends to a class $\delta_E \in \mathsf{aW}$; furthermore, $[E]$ descends to a boundary divisor $[S_r]$ in $\mathsf{aW}$ for a ray $r \in \mathsf{W}$.
By combinatorial flatness, the preimage of $[S_r]$ is a combination of boundary divisors $S_\sigma \subset \mathsf{aT}_{\mathsf{v}}$, where $\sigma$ runs over rays of $T_{\mathsf v}$ which surject onto $r$.
Therefore, we have
\[f^\star(i_{E,\star}(\gamma_E)) = f^\star\circ\rho^\star(\delta_E[S_r]) = \sum_{\sigma} \pi^\star(\theta_\sigma [S_\sigma]).\]
By summing over $E$, we have that $f^\star\mathsf{Res}_{\partial}$ is a combination of strata classes $[\![\sigma, \theta_\sigma, 1]\!]$, as desired.
\end{proof}

\begin{corollary}
    The capped vertex is a Laurent polynomial in $q$. 
\end{corollary}

\begin{proof}

    This follows from Proposition \ref{Chow-residue} and Proposition \ref{2/3NBcase}.
\end{proof}

\section{Non-nef geometries}

We study elementary toric threefold pairs $(Y|\partial Y)$ where the non-boundary divisors are not nef. 

\subsection{General $1$-non-boundary} We start with the $1$-non-boundary case -- projective line bundles over a surface $S$ and on section of the bundle is not in the boundary. The next proposition allows us to conclude the correspondence inductively in many cases.

\begin{proposition}\label{prop: negative-induction}
Let $(V|\partial V)$ a full-boundary toric surface pair and suppose there exists a surjection $A \rightarrow A_V$ onto the torus of $V$, and an $A$-equivariant logarithmic map $$\mathsf{Ev} \rightarrow V.$$
Then any GW/PT series for a star $\mathsf v$ is effectively determined by the GW/PT series of smaller stars, compatibly with the correspondence.
\end{proposition}

\begin{proof}
Replace $V$ with a blowup, and pass to subdivisions of $T$ and $P$ to get a commutative diagram
\[
\begin{tikzcd}
\mathscr M\arrow{d}\arrow{r} & \mathsf{Ev} \arrow{d}\arrow{r}{f} & V\arrow{d} \\
\mathsf a T_{\sf v}\arrow{r} &\mathsf{a} \mathsf P\arrow{r} & \mathsf{a} \Sigma_V,
\end{tikzcd}
\]
where the bottom arrows are both flat.  By the assumption on the action of $A$ on $V$, even after blowup, $V$ has isolated $A$-fixed points $p_1, \dots, p_r$. In order to calculate an equivariant invariant, it suffices to do so after tensoring with 
\[K = \mathsf {Frac} \ H_{A}^{\star}(\mathsf{pt}).\]
By the equivariant localization theorem, we have:
\[1_V = \sum c_i(a)[p_i] \in H_{A}^{*}(V) \otimes K\]
with $c_i(a) \in K.$
Given $\alpha \in H_{A}^{*}(\mathsf{Ev})$, we have
\[\alpha = \sum c_i(a) f^\star([p_i])\]
so by linearity, we can assume $\alpha = \alpha' \cdot f^\star([p_i])$ for some $\alpha'$.

By flatness, we can descend $[p_i]$ to the class of a stratum in $\mathsf{a}\Sigma_V$ and pull back to a stratum associated to a $2$-dimensional cone $\sigma$ of $T_\mathsf{v}$, so it suffices to consider the GW/PT series associated to $[\![\sigma, 1, \alpha']\!] \in H_{A}^{*}(\mathscr{M})$. Since we are in either a ruled or general $1$-non-boundary geometry, the tropical space of translated stars is at most $1$-dimensional, so the strata invariant breaks into smaller stars.
\end{proof}

We apply the proposition in a couple of different ways. First, suppose $\mathsf v$ has an internal marking or a ray that points into the interior of $\Sigma_Y$. Then the corresponding factor of the evaluation space gives a toric surface $V$ satisfying the conditions of the proposition, either $V = S$ or $V = D_i$.

Next, suppose $\mathsf v$ has two rays $r_1, r_2$ along the boundary of $\Sigma_Y$ that are not in opposite directions (inside the $\mathbb R^3$ containing $\Sigma_Y$).  Each factor $D_1, D_2$ of the evaluation space admits a projection to $C_1, C_2 = \mathbb P^1$ with full boundary. Since $r_1$ and $r_2$ are not parallel, $A$ acts on each $C_i$ with linearly independent weights and the proposition applies with $V = C_1\times C_2$.

\subsection{Linear case} The case not handled inductively by the proposition is the {\it linear geometry}:
\[
S = \mathbb{P}^1 \times \mathbb{P}^1, \ L = \mathcal{O}(-m,0), \ m > 0, \ \beta = d[\mathbb{P}^1 \times 0].
\]

\begin{proposition}\label{prop: linear-negative-geometry}
    The GW/PT corespondence holds for the linear geometry above, for all $m\geq0$ and $d>0$.
\end{proposition}

The idea is that these invariants are similar to local curves of the form $\mathsf{Tot}_{\mathbb{P}^{1}}(\mathcal{O}(-m) \oplus \mathcal{O}(0))$. The GW/PT correspondence in this case is known by~\cite{BP08,OP10}. In fact, we will give an independent argument for this calculation in Proposition \ref{localcurveredux} of Section \ref{localcurve-section}.

The local curve case does not have a boundary divisor at infinity and does not exactly coincide with the invariants above. But the local curve theory determines some logarithmic invariants, for a particular choice of equivariant insertions. An induction argument inspired by~\cite[Section~2]{OP06vira} determines all other invariants from this case.

\begin{proof}
    To reduce to local curves, which do not have boundary at infinity, we analyze two targets simultaneously.
The first, denoted $(Y| \partial Y)$, is the $1$-non-boundary geometry from the last section: take $(Z|\partial Z)$ to be the projective completion of $\mathcal O_{\PP^1}(-m)$ and the zero section $Z_0$ to be the only non-boundary; then take
$$
(Y|\partial Y) = (Z| \partial Z) \times (\mathbb{P}^1|0, \infty).
$$

\noindent
The second target is the same, with a different boundary:
$$
(X| \partial X) = (Z| \partial Z) \times (\mathbb{P}^1|\infty).
$$
Equivalently, $(X| \partial X)$ is the $2$-non-boundary geometry associated to line bundles $\mathcal{O}(-m)$ and $\mathcal{O}(0)$. See Figure~\ref{fig: ratchet} below:
\begin{figure}[h!]
\begin{tikzpicture}[x=0.75pt,y=0.75pt,yscale=-1,xscale=1]

\definecolor{lavender}{RGB}{238,230,255}

\fill[lavender] (60,200) rectangle (245,385);
\fill[lavender] (265,200) rectangle (450,385);

\def\amp{1.5}
\def\seg{6}
\def\pad{12}

\coordinate (A) at (115,235);
\coordinate (B) at (205,235);
\coordinate (C) at (205,325);
\coordinate (D) at (115,325);

\draw[violet,line width=1] (A) -- (D);
\draw[violet,line width=1,decorate,decoration={snake,amplitude=\amp,segment length=\seg}] (A) -- (B);
\draw[violet,line width=1,decorate,decoration={snake,amplitude=\amp,segment length=\seg}] (B) -- (C);
\draw[violet,line width=1,decorate,decoration={snake,amplitude=\amp,segment length=\seg}] (D) -- (C);
\draw[purple,line width=1.4] ($(A)!0.5!(B)$) -- ($(D)!0.5!(C)$);

\foreach \pt in {A,B,C,D} {\draw[violet,fill=violet] (\pt) circle (2pt);}

\draw[violet,line width=0.75,rounded corners=10pt] 
  ($(A)+(-\pad,-\pad)$) rectangle ($(C)+(\pad,\pad)$);

\draw[violet] ($(D)+(0,20)$) node [anchor=north] {\small $X|\partial X$};

\coordinate (A2) at (320,235);
\coordinate (B2) at (410,235);
\coordinate (C2) at (410,325);
\coordinate (D2) at (320,325);

\draw[violet,line width=1,decorate,decoration={snake,amplitude=\amp,segment length=\seg}] (A2) -- (B2);
\draw[violet,line width=1,decorate,decoration={snake,amplitude=\amp,segment length=\seg}] (B2) -- (C2);
\draw[violet,line width=1,decorate,decoration={snake,amplitude=\amp,segment length=\seg}] (D2) -- (C2);
\draw[violet,line width=1,decorate,decoration={snake,amplitude=\amp,segment length=\seg}] (A2) -- (D2);
\draw[purple,line width=1.4] ($(A2)!0.5!(B2)$) -- ($(D2)!0.5!(C2)$);

\foreach \pt in {A2,B2,C2,D2} {\draw[violet,fill=violet] (\pt) circle (2pt);}

\draw[violet,line width=0.75,rounded corners=10pt] 
  ($(A2)+(-\pad,-\pad)$) rectangle ($(C2)+(\pad,\pad)$);

\draw[violet] ($(D2)+(0,20)$) node [anchor=north] {\small $Y|\partial Y$};

\end{tikzpicture}
\caption{The two geometries involved in the proof.  We show $Z_0 \times \mathbb{P}^1$, ignoring the $\mathcal O(-m)$-direction. The logarithmic divisors are depicted with wavey lines. The curve is shown in red in the middle.}\label{fig: ratchet}
\end{figure}
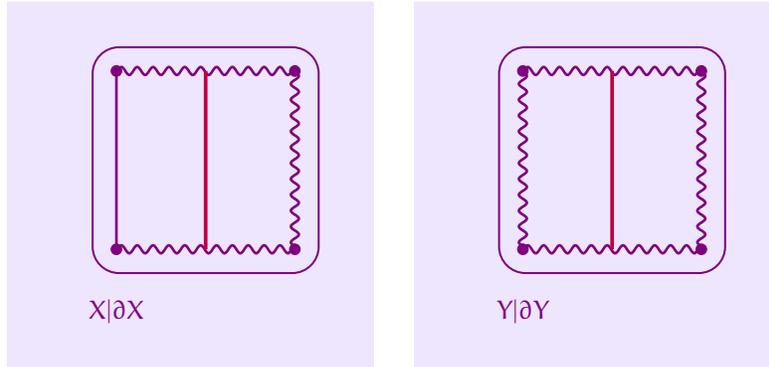

Consider $\beta = d[Z_0 \times 0]$, which intersects the boundary divisors 
$D_{0 \ \textrm{or } \infty} = [\pi^{-1}(0 \ \textrm{or } \infty)] \times \mathbb{P}^1$, 
each with multiplicity $d$. For both targets, we have the $3$-dimensional torus $A$ acting on $X$ and $Y$, 
respectively, as well as the subtorus $A_0$, which fixes the second factor in the product. 
We keep track of the insertions along $D_0$, which are either $A$- or $A_0$-equivariant 
classes on the second $\mathbb{P}^1$ factor.

The four cases mix the two equivariant settings and the two targets:
\begin{enumerate}[(i)]
\item $A$-equivariant $X$ with $D_0$-insertion 
$[0]^{\otimes a} \otimes [\infty]^{\otimes b}$,

\item $A_0$-equivariant $X$ with $D_0$-insertion 
$[p]^{\otimes a} \otimes [1]^{\otimes b}$,

\item $A_0$-equivariant $Y$ with $D_0$-insertion 
$[p]^{\otimes a} \otimes [1]^{\otimes b}$,

\item $A$-equivariant $Y$ with $D_0$-insertion 
$[0]^{\otimes a} \otimes [\infty]^{\otimes b}$.
\end{enumerate}

We induct on the degree $d$ and assume GW/PT is known in all cases for $d' < d$. 
In each of these cases, the degeneration formula takes a naive form; there are no 
corrections from blowups.

\noindent
{\sc Step I.  Case $(i)$ is known when $a > 0$: }

If $b > 0$, we degenerate to the normal cone along $Z \times 0$. 
In the degeneration formula, the degree on each component is strictly 
smaller than $d$, so the invariants are known inductively.  If $b = 0$, all $D_0$-insertions are $[0]$, and the curve is forced to have support $\mathbb{P}^1 \times [0]$. In this case, the invariant equals the corresponding local curve invariant, which is known. 
Due to the negative normal bundle, the curve does not deform into the boundary.

\noindent
{\sc Step II.  Case $(ii)$ is known when $a > 0$}

This follows from Step~I by choosing the equivariant lift of $[p]$ given by $[0]$ and expanding $[1]$ in the fixed-point basis.

\noindent
{\sc Step III.  Case $(iii)$ is known when $a > 0$}

If we start with the invariant in Step II and degenerate $X$ to the normal cone along $Z \times 0$, we get a copy of $X$ glued to a copy of $Y$; we can assume the point insertion $[p]$ degenerates to a point class in $Y$. This is where working on $A_0$-equivariantly is important. It means that we can choose where the insertions specialize in the degeneration. We do not have this flexibility for fully $A$-equivariant insertions. We can compute this known invariant from Step II by using the degeneration formula. There is a leading term where the entire curve class stays in $Y$. Every other term involves a nontrivial splitting of the degree, and so is known inductively. We can therefore compute the invariant in terms of known ones. 

\noindent
{\sc Step IV.  Cases $(i)$ and $(iv)$ are known in general}

We can inductively assume either $a=0$ or $b=0$.  In particular, the evaluation insertion can be replaced with a nontrivial boundary stratum of the moduli space, where the entire cycle lies in the boundary of an expansion of $X$ or $Y$.
Since this is a boundary stratum, there is an action of a rubber torus on the target. Due to this action, the action of $A$ on this stratum factors through the action of $A_0$. Therefore, it suffices to calculate the stratum invariant $A_0$-equivariantly.  When we apply the rubber calculus algorithm, each component 
of the degenerate target acquires a point insertion in the rigidification process.  We can assume this insertion is added to a relative point mapping to $D_0$.  After splitting, we have a collection of invariants of the form
of Case (iii), so is known.

\noindent
{\sc Step V.  Case $(ii)$ and $(iii)$ are known in general}

This follows from Step III, by picking an equivariant lift.  
\end{proof}

\subsection{General $2$-non-boundary}

We again start with a proposition to handle the induction. 

\begin{proposition}
Suppose there exists a nontrivial action of $A$ on $\mathbb P^1$ and an $A$-equivariant logarithmic map $$\mathsf{Ev} \rightarrow (\mathbb P^1| 0,\infty).$$
Then any GW/PT series for star $\mathsf v$ is effectively determined by the GW/PT series of lower-order stars, compatibly with the correspondence.
\end{proposition}

\begin{proof}
The general $2$-boundary case has no translations, so this follows from the argument for Proposition~\ref{prop: negative-induction}. 
\end{proof}

The only rays in $\Sigma_Y$ for which the evaluation space doesn't have a map to $(\mathbb P^1|0,\infty)$ are those corresponding to the divisors $\pi^{-1}(0)$ and $\pi^{-1}(\infty)$; in the tropical description of $\Sigma_Y$ from Section \ref{sec: elementary-list},
these are the ray
$y= \psi_1(\mathbb{R}_{\geq 0}), z = \psi_2(\mathbb{R}_{\geq 0})$ and the ray $y= \psi_1(\mathbb{R}_{\leq 0}), z = \psi_2(\mathbb{R}_{\leq 0})$.

To analyze these remaining cases, degenerate $\mathcal{O}(a)\oplus\mathcal{O}(b)$ on $\mathbb{P}^1$ to building blocks $\mathcal{O}(1)\oplus\mathcal{O}(0)$, which we know from the nef case,
and $\mathcal{O}(-1)\oplus\mathcal{O}(0)$, which was solved as case $(i)$ in the analysis of the $1$-non-boundary case via local curves.

\subsection{Local curve calculation}\label{localcurvesection}

In this last section, we prove the following, which makes the argument for non-nef geometries independent of \cite{BP08,OP10}.

\begin{proposition}\label{localcurveredux}
The residue GW/PT theories for the local curve $\mathcal{O}(a)\oplus \mathcal{O}(b)$ on $(\mathbb{P}^1| \partial \mathbb{P}^1)$ are determined from the nef elementary geometries of Section~\ref{sec: nef-geometries}.
\end{proposition}

\begin{proof}
By a standard degeneration argument, these invariants are determined by the case of $(a,b) = (1,0), (0,0), (-1,0)$ and the case of $(0,0)$ is known already by Section~\ref{sec: linear-stars}.
Furthermore, if viewed as operators, the degeneration formula shows that the matrix of $(1,0)$-invariants is invertible, with inverse given by the $(-1,0)$ invariants, so it suffices to calculate the residue local curve theory for $(a,b) = (1,0)$.  This invertibility property fails if we compactify the local curve geometry and work with logarithmic theory on the compactification, since the degeneration formula has more terms, so we can't use this move directly in the earlier section.

As in the last section, let $(X| \partial X)$ denote the projective compactification of $\mathsf{Tot}_{\mathbb{P}^{1}}(\mathcal{O}(1) \oplus \mathcal{O}(0))$.  
We want to apply Proposition \ref{Chow-residue} to this geometry as well.  To do this, we replace the argument with the log Chow variety for the elementary geometry 
$\mathbb{P}(\mathcal{O}(1)\oplus \mathcal{O}(0))$ with a single non-boundary divisor given by the zero section of $\mathcal{O}(1)$.  The results for the appendix apply to this geometry as well, so the same proof goes through.

The statement of Proposition \ref{Chow-residue} then determines the residue theory from strata invariants of $(X| \partial X)$ all of which are known from Section~\ref{sec: 2-nb-3-nb}.
\end{proof}

\begin{remark}
    The appearance of the residue local curve theory in calculating the non-nef geometries may seem surprising. The reason it appears is that the local curve theory satisfies an ``invertibility property'' -- the local curve theories for $\mathcal O(1)\oplus\mathcal O$ and  $\mathcal O\oplus\mathcal O$ over $\mathbb P^1$ determine the theory for $\mathcal O(-1)\oplus\mathcal O$; thought of as operators, the theory for $\mathcal O(1)\oplus\mathcal O$ is invertible. However, this property fails when we pass to the compactified logarithmic theory, so to use this invertibility we need to relate the compactified and residue theories. This also happens in Section~\ref{sec: polynomiality}.
\end{remark}

\appendix
\section{The logarithmic Chow variety of a surface}
\label{app: log-linear-system}

Let $S$ be a toric surface with torus $A$ and $\partial S$ an $A$-invariant curve. The pair $(S|\partial S)$ is {\it elementary} if, up to blowup along strata of $\partial S$, the surface $S$ is a $\mathbb P^1$-bundle over $\mathbb P^1$ and $\partial S$ is either (i) the full toric boundary, (ii) three out of the four boundary curves, or (ii) two intersecting boundary curves. Such elementary toric surface pairs are the boundary divisors in elementary toric threefolds. 

Fix a star $\mathsf v$ in $\Sigma_{S|\partial S}$. It determines a curve class, which we also denote $\mathsf v$; there should be no room for confusion. Let $\mathsf{Chow}_{\mathsf v}(S|\partial S)$ be the logarithmic Chow space of curves on expansions of $S$. We are working on a surface, so this coincides with a Hilbert scheme of pure $1$-dimensional subschemes of degree $\mathsf v$. The space $\mathsf{Chow}_{\mathsf v}(S|\partial S)$ is equipped with a logarithmic structure $\mathsf{Chow}_{\mathsf v}(S|\partial S)\to \mathsf aT_{\mathsf v}(S|\partial S)$. The morphism $\mathsf{Chow}_{\mathsf v}(S|\partial S)\to \mathsf aT_{\mathsf v}(S|\partial S)$ is smooth (though the domain stack can have nontrivial finite isotropy groups) and stratifies $\mathsf{Chow}_{\mathsf v}(S|\partial S)$. By subdivision, we assume that $\mathsf a T_{\mathsf v}(S|\partial S)$ is smooth. 

\begin{proposition}\label{prop: chow-generation}
    Each locally closed logarithmic stratum of $\mathsf{Chow}_{\mathsf v}(S|\partial S)$ is a simplicial affine toric variety, stable under the action of $A$. In particular, for every closed logarithmic stratum $W$ of $\mathsf{Chow}_{\mathsf v}(S|\partial S)$, the $A$-equivariant Chow/Borel-Moore homology of $W$ is generated by logarithmic boundary strata. 
\end{proposition}

When $\partial S$ is the full toric boundary, these results can be read off Kennedy-Hunt's description of the Chow space as the secondary toric variety of the primary polytope determined by $\mathsf v$, see~\cite{KH21}. 

\begin{proof}[Proof of Proposition~\ref{prop: chow-generation}]

Consider the open stratum $\mathsf{Chow}^\circ_{\mathsf v}(S|\partial S)$. There is a map to the linear system of all degree $\mathsf v$ curves on $S$:
\[
\mathsf{Chow}^\circ_{\mathsf v}(S|\partial S)\to \mathbb P_{\mathsf v}.
\]
Since $S$ is toric, the linear system has canonical coordinates given by monomials. The open stratum of the Chow variety is precisely the locus of curves that are dimensionally transverse to $\partial S$. The transversality condition is stable under scaling the monomials of the defining polynomial. This implies that $\mathsf{Chow}^\circ_{\mathsf v}(S|\partial S)$ is a torus-invariant affine subset of $\mathbb P_{\mathsf v}$. 

We now show that the complement of $\mathsf{Chow}^\circ_{\mathsf v}(S|\partial S)$ in $\mathbb P_{\mathsf v}$ has pure codimension $1$, so $\mathsf{Chow}^\circ_{\mathsf v}(S|\partial S)$ is affine and toric. We want to remove from $\mathbb P_{\mathsf v}$ two types of ``bad'' loci: (i) any locus where a curve meets a toric point in $\partial S$, and (ii) any locus where the curve contains a component of $\partial S$. By the elementary condition, if a curve contains a component of $\partial S$, it meets a boundary point. Therefore it suffices to show that the union of loci of type (i) is pure codimension $1$. Now if we fix a point $q$ on $S$, the locus in $\mathbb P_{\mathsf v}$ where the curves pass through $q$ is either pure of codimension $1$ or all of $\mathbb P_{\mathsf v}$. The bad locus is easily seen to be non-empty e.g. because any torus invariant curve in an elementary target lies in the bad locus. We conclude that $\mathsf{Chow}^\circ_{\mathsf v}(S|\partial S)$ is toric and affine.

Fix a locally closed boundary stratum $\mathsf{Chow}^\circ_\Gamma(S|\partial S)\hookrightarrow \mathsf{Chow}_{\mathsf v}(S|\partial S)$. Apply the rigidification of Section~\ref{sec: rigidification} and the results in~\cite{MR23}. Passing to a torus bundle $\mathsf{Chow}^\sim_\Gamma(S|\partial S)\to \mathsf{Chow}^\circ_\Gamma(S|\partial S)$, we can write $\mathsf{Chow}^\sim_\Gamma(S|\partial S)$ as a fiber product\footnote{Note that this statement is only being used on the locally closed stratum, so the flattening and logarithmic fiber product procedures that play a key role in~\cite{MR23} are irrelevant.} of the interiors of logarithmic Chow spaces associated to the vertices of $\Gamma$. For each vertex, we have the interior of an associated logarithmic Chow space, and for each edge $\mathsf e$ we have a logarithmic symmetric product of points on $\mathbb P^1_{\mathsf e}$, where $\mathbb P^1_{\mathsf e}$ is relative to either $1$ or $2$ of its torus fixed points. Evaluation gives a map from the product of vertex spaces to the product of the {\it squares} of the edge spaces. We are interested in the pullback of the diagonal locus:
\[
\begin{tikzcd}
\mathsf{Chow}^\sim_\Gamma(S|\partial S)\arrow{r}\arrow{d} & \prod_{\mathsf w\in V(\Gamma)} \mathsf{Chow}_{\mathsf w}^\circ(S_{\mathsf w}|\partial S_{\mathsf w})\arrow{d}  \\
\prod_{\mathsf e\in E(\Gamma)} \mathsf{Sym}^{n_{\mathsf e},\circ}(\mathbb P^1_{\mathsf e}|\partial \mathbb P^1_{\mathsf e}) \arrow{r} & \left(\prod_{\mathsf e\in E(\Gamma)} \mathsf{Sym}^{n_{\mathsf e},\circ}(\mathbb P^1_{\mathsf e}|\partial \mathbb P^1_{\mathsf e})\right)^2.
\end{tikzcd}
\]
The top right and the bottom two spaces are toric, and the maps are torus-equivariant. If the fiber product is irreducible, then it is automatically an affine toric variety. Indeed, the fiber product is automatically an affine binomial scheme, and irreducible affine binomial varieties are toric by~\cite{ES96}. The Chow $1$-complex $\Gamma$ determines a degeneration of $S$ and a line bundle on each component. The space $\mathsf{Chow}^\sim_\Gamma(S|\partial S)$ is an open subspace in the projective space of sections of this line bundle, because the transversality condition is an open condition\footnote{Transversality is equivalent to flatness over the Artin fan of the degeneration, and the flat locus is open.}. It is therefore irreducible, and so the stratum is affine and toric.

To conclude, observe that $\mathsf{Chow}^{\sim}_\Gamma(S|\partial S)\to \mathsf{Chow}^{\circ}_\Gamma(S|\partial S)$ is a torus torsor. The total space is affine toric, so the base is affine and toric. The stratum is stable under the $A$-action. The rest of the proposition follows by the excision sequence for Chow groups and the fact that an affine toric orbifold has trivial rational Chow groups.
\end{proof}

In particular, we record the following:

\begin{corollary}\label{cor: Chow-chow-groups}
    Given a boundary divisor $D$ on $\mathsf{Chow}_{\mathsf v}(S|\partial S)$, any cohomology class on $D$ is generated logarithmic boundary strata of $D$. 
\end{corollary}

\bibliographystyle{siam} 
\bibliography{PrimaryLMNOP}

\begin{thebibliography}{10}

\bibitem{AC11}
{\sc D.~Abramovich and Q.~Chen}, {\em Stable logarithmic maps to
  {D}eligne-{F}altings pairs {II}}, Asian J. Math., 18 (2014), pp.~465--488.

\bibitem{AW}
{\sc D.~Abramovich and J.~Wise}, {\em {Birational invariance in logarithmic
  Gromov--Witten theory}}, Comp. Math., 154 (2018), pp.~595--620.

\bibitem{AKMV}
{\sc M.~Aganagic, A.~Klemm, M.~Mari\~no, and C.~Vafa}, {\em The topological
  vertex}, Comm. Math. Phys., 254 (2005), pp.~425--478.

\bibitem{stacky-concentration-adeel}
{\sc D.~Aranha, A.~A. Khan, A.~Latyntsev, H.~Park, and C.~Ravi}, {\em The
  stacky concentration theorem}, arXiv preprint arXiv:2407.08747,  (2024).

\bibitem{AB25}
{\sc H.~Arg{\"u}z and P.~Bousseau}, {\em {BPS polynomials and Welschinger
  invariants}}, arXiv:2506.02770,  (2025).

\bibitem{BG16}
{\sc F.~Block and L.~G\"ottsche}, {\em Refined curve counting with tropical
  geometry}, Comp. Math., 152 (2016), pp.~115--151.

\bibitem{Bou17}
{\sc P.~Bousseau}, {\em Tropical refined curve counting from higher genera and
  lambda classes}, Invent. Math., 215 (2019), pp.~1--79.

\bibitem{BP08}
{\sc J.~Bryan and R.~Pandharipande}, {\em {The local Gromov-Witten theory of
  curves}}, J. Amer. Math. Soc., 21 (2008), pp.~101--136.

\bibitem{BR21}
{\sc A.~Buryak and P.~Rossi}, {\em Quadratic double ramification integrals and
  the noncommutative {K}d{V} hierarchy}, Bull. Lond. Math. Soc., 53 (2021),
  pp.~843--854.

\bibitem{CN21}
{\sc F.~Carocci and N.~Nabijou}, {\em Rubber tori in the boundary of expanded
  stable maps}, J. Lond. Math. Soc. (2), 109 (2024), pp.~Paper No. e12874, 36.

\bibitem{CCUW}
{\sc R.~Cavalieri, M.~Chan, M.~Ulirsch, and J.~Wise}, {\em A moduli stack of
  tropical curves}, {Forum Math. Sigma}, 8 (2020), pp.~1--93.

\bibitem{Che10}
{\sc Q.~Chen}, {\em Stable logarithmic maps to {D}eligne-{F}altings pairs {I}},
  Ann. of Math., 180 (2014), pp.~341--392.

\bibitem{ES96}
{\sc D.~Eisenbud and B.~Sturmfels}, {\em Binomial ideals}, Duke Math. J., 84
  (1996), pp.~1--45.

\bibitem{FP00}
{\sc C.~Faber and R.~Pandharipande}, {\em Hodge integrals and {G}romov-{W}itten
  theory}, Invent. Math., 139 (2000), pp.~173--199.

\bibitem{Ful93}
{\sc W.~Fulton}, {\em Introduction to Toric Varieties}, Princeton University
  Press, 1993.

\bibitem{GS14}
{\sc L.~G\"ottsche and V.~Shende}, {\em Refined curve counting on complex
  surfaces}, Geom. Topol., 18 (2014), pp.~2245--2307.

\bibitem{GP99}
{\sc T.~Graber and R.~Pandharipande}, {\em Localization of virtual classes},
  Invent. Math., 135 (1999), pp.~487--518.

\bibitem{Groj96}
{\sc I.~Grojnowski}, {\em Instantons and affine algebras. {I}. {T}he {H}ilbert
  scheme and vertex operators}, Math. Res. Lett., 3 (1996), pp.~275--291.

\bibitem{GS13}
{\sc M.~Gross and B.~Siebert}, {\em {Logarithmic Gromov-Witten invariants}}, J.
  Amer. Math. Soc., 26 (2013), pp.~451--510.

\bibitem{Guz25}
{\sc J.~Guzman}, {\em {Logarithmic cobordism and Donaldson-Thomas invariants}},
  arXiv:2510.15085,  (2025).

\bibitem{HMPPS}
{\sc D.~Holmes, S.~Molcho, R.~Pandharipande, A.~Pixton, and J.~Schmitt}, {\em
  Logarithmic double ramification cycles}, Invent. Math., 240 (2025),
  pp.~35--121.

\bibitem{HS21}
{\sc D.~Holmes and R.~Schwarz}, {\em Logarithmic intersections of double
  ramification cycles}, Algebr. Geom., 9 (2022), pp.~574--605.

\bibitem{KKMSD}
{\sc G.~Kempf, F.~Knudsen, D.~Mumford, and B.~Saint-Donat}, {\em Toroidal
  embeddings {I}}, Lecture Notes in Mathematics, 339 (1973).

\bibitem{KH21}
{\sc P.~Kennedy-Hunt}, {\em {Logarithmic Pandharipande-Thomas spaces and the
  secondary polytope}}, arXiv:2112.00809,  (2021).

\bibitem{KHSUK}
{\sc P.~Kennedy-Hunt, Q.~Shafi, and A.~U. Kumaran}, {\em Tropical refined curve
  counting with descendants}, Comm. Math. Phys., 405 (2024), pp.~Paper No. 240,
  41.

\bibitem{Li01}
{\sc J.~Li}, {\em Stable morphisms to singular schemes and relative stable
  morphisms}, J. Diff. Geom., 57 (2001), pp.~509--578.

\bibitem{Li02}
\leavevmode\vrule height 2pt depth -1.6pt width 23pt, {\em {A degeneration
  formula of GW-invariants}}, J. Diff. Geom., 60 (2002), pp.~199--293.

\bibitem{LiWu15}
{\sc J.~Li and B.~Wu}, {\em {Good degeneration of Quot-schemes and coherent
  systems}}, Comm. Anal. Geom., 23 (2015), pp.~841--921.

\bibitem{Maulik-Negut}
{\sc D.~Maulik and A.~Neguț}, {\em Lehn's formula in {C}how and conjectures of
  {B}eauville and {V}oisin}, J. Inst. Math. Jussieu, 21 (2022), pp.~933--971.

\bibitem{MNOP06a}
{\sc D.~Maulik, N.~Nekrasov, A.~Okounkov, and R.~Pandharipande}, {\em
  {Gromov--Witten theory and Donaldson--Thomas theory, I}}, Compos. Math., 142
  (2006), pp.~1263--1285.

\bibitem{MNOP06b}
\leavevmode\vrule height 2pt depth -1.6pt width 23pt, {\em {Gromov--Witten
  theory and Donaldson--Thomas theory, II}}, Compos. Math., 142 (2006),
  pp.~1286--1304.

\bibitem{MOOP}
{\sc D.~Maulik, A.~Oblomkov, A.~Okounkov, and R.~Pandharipande}, {\em
  {Gromov-Witten/Donaldson-Thomas correspondence for toric 3-folds}}, Invent.
  Math., 186 (2011), pp.~435--479.

\bibitem{MP06}
{\sc D.~Maulik and R.~Pandharipande}, {\em A topological view of
  {G}romov--{W}itten theory}, Topology, 45 (2006), pp.~887--918.

\bibitem{MPT10}
{\sc D.~Maulik, R.~Pandharipande, and R.~P. Thomas}, {\em Curves on {K3}
  surfaces and modular forms}, J. Topol., 3 (2010), pp.~937--996.

\bibitem{MR20}
{\sc D.~Maulik and D.~Ranganathan}, {\em Logarithmic {D}onaldson--{T}homas
  theory}, Forum Math. Pi, 12 (2024), p.~Paper No. e9.

\bibitem{MR25}
\leavevmode\vrule height 2pt depth -1.6pt width 23pt, {\em {Gromov-Witten
  theory, degenerations, and the tautological ring}}, arXiv:2510.04779,
  (2025).

\bibitem{MR23}
\leavevmode\vrule height 2pt depth -1.6pt width 23pt, {\em Logarithmic
  enumerative geometry for curves and sheaves}, Camb. J. Math., 13 (2025),
  pp.~51--172.

\bibitem{Mik17}
{\sc G.~Mikhalkin}, {\em Quantum indices and refined enumeration of real plane
  curves}, Acta Math., 219 (2017), pp.~135--180.

\bibitem{MR21}
{\sc S.~Molcho and D.~Ranganathan}, {\em A case study of intersections on
  blowups of the moduli of curves}, Algebra Number Theory, 18 (2024),
  pp.~1767--1816.

\bibitem{Nak97}
{\sc H.~Nakajima}, {\em Heisenberg algebra and {H}ilbert schemes of points on
  projective surfaces}, Ann. of Math. (2), 145 (1997), pp.~379--388.

\bibitem{Negut}
{\sc A.~Neguț}, {\em Hecke correspondences for smooth moduli spaces of
  sheaves}, Publ. Math. Inst. Hautes \'Etudes Sci., 135 (2022), pp.~337--418.

\bibitem{Negut2}
\leavevmode\vrule height 2pt depth -1.6pt width 23pt, {\em {$W$}-algebras
  associated to surfaces}, Proc. Lond. Math. Soc. (3), 124 (2022),
  pp.~601--679.

\bibitem{NO16}
{\sc N.~Nekrasov and A.~Okounkov}, {\em Membranes and sheaves}, Algebr. Geom.,
  3 (2016), pp.~320--369.

\bibitem{NPS18}
{\sc J.~Nicaise, S.~Payne, and F.~Schroeter}, {\em Tropical refined curve
  counting via motivic integration}, Geom. Topol., 22 (2018), pp.~3175--3234.

\bibitem{Ob21}
{\sc G.~Oberdieck}, {\em Marked relative invariants and {GW}/{PT}
  correspondences}, Adv. Math., 439 (2024), pp.~Paper No. 109472, 98.

\bibitem{OP06vira}
{\sc A.~Okounkov and R.~Pandharipande}, {\em Virasoro constraints for target
  curves}, Invent. Math., 163 (2006), pp.~47--108.

\bibitem{OP10}
\leavevmode\vrule height 2pt depth -1.6pt width 23pt, {\em The local
  {Donaldson}-{Thomas} theory of curves}, Geom. Topol., 14 (2010),
  pp.~1503--1567.

\bibitem{PP13}
{\sc R.~Pandharipande and A.~Pixton}, {\em Descendents on local curves:
  rationality}, Comp. Math., 149 (2013), pp.~81--124.

\bibitem{PP14}
\leavevmode\vrule height 2pt depth -1.6pt width 23pt, {\em
  {G}romov--{W}itten/pairs descendent correspondence for toric 3-folds}, Geom.
  Top., 18 (2014), pp.~2747--2821.

\bibitem{PP12}
\leavevmode\vrule height 2pt depth -1.6pt width 23pt, {\em
  Gromov-{W}itten/pairs correspondence for the quintic 3-fold}, J. Amer. Math.
  Soc., 30 (2017), pp.~389--449.

\bibitem{PT09}
{\sc R.~Pandharipande and R.~P. Thomas}, {\em Curve counting via stable pairs
  in the derived category}, Invent. Math., 178 (2009), pp.~407--447.

\bibitem{Par23}
{\sc J.~Pardon}, {\em {Universally counting curves in Calabi-Yau threefolds}},
  arXiv:2308.02948,  (2023).

\bibitem{Par17}
{\sc B.~Parker}, {\em Three dimensional tropical correspondence formula},
  Commun. Math. Phys., 353 (2017), pp.~791--819.

\bibitem{R19}
{\sc D.~Ranganathan}, {\em Logarithmic {G}romov-{W}itten theory with
  expansions}, Algebr. Geom., 9 (2022), pp.~714--761.

\bibitem{R26-SRI}
\leavevmode\vrule height 2pt depth -1.6pt width 23pt, {\em An invitation to the
  enumerative geometry of degenerations}, arXiv:2602.20840,  (2026).

\bibitem{RUK22}
{\sc D.~Ranganathan and A.~Urundolil~Kumaran}, {\em Logarithmic
  {G}romov-{W}itten theory and double ramification cycles}, J. Reine Angew.
  Math., 809 (2024), pp.~1--40.

\bibitem{Tev07}
{\sc J.~Tevelev}, {\em Compactifications of subvarieties of tori}, Amer. J.
  Math., 129 (2007), pp.~1087--1104.

\end{thebibliography}

\end{document}